\documentclass[11pt]{amsart}

\usepackage{amsmath,amssymb,amsthm,tikz, nicefrac, mathrsfs,upgreek}

\RequirePackage{color}
\RequirePackage[colorlinks, urlcolor=my-blue,linkcolor=my-red,citecolor=my-red]{hyperref}
\definecolor{my-blue}{rgb}{0.0,0.0,0.6}
\definecolor{my-red}{rgb}{0.5,0.0,0.0}
\definecolor{my-green}{rgb}{0.0,0.5,0.0}
\definecolor{nicos-red}{rgb}{0.75,0.0,0.0}
\definecolor{nicos-green}{rgb}{0.0,0.75,0.0}
\definecolor{light-gray}{gray}{0.6}
\definecolor{really-light-gray}{gray}{0.8}
\definecolor{sussexg}{rgb}{0.0,0.5,0.5}
\definecolor{sussexp}{rgb}{0.5,0.0,0.5}

\usepackage{dsfont} 

\usepackage[utf8]{inputenc}

\usepackage{nicefrac}
\usepackage{MnSymbol}

\usepackage{multicol}
\PassOptionsToPackage{dvipsnames}{xcolor}
\usepackage{tikz,xcolor}

\usetikzlibrary{shapes,arrows}
\usetikzlibrary{hobby}
\usetikzlibrary{decorations.markings}

\newtheorem{theorem}{\color{my-red}{\sc Theorem}}[section]
\newtheorem{lemma}[theorem]{\color{my-red} \sc Lemma}
\newtheorem{proposition}[theorem]{\color{my-red} \sc Proposition}
\newtheorem{corollary}[theorem]{\color{my-red} \sc Corollary}
\newtheorem{conjecture}[theorem]{\color{my-red} \sc Conjecture}

\numberwithin{equation}{section}
\theoremstyle{remark}
\newtheorem{remark}[theorem]{\color{my-red} Remark}

\newcommand{\be}{\begin{equation}}
\newcommand{\ee}{\end{equation}}

\providecommand{\abs}[1]{\vert#1\vert}

\newcommand{\eone}{\textup{e}_1}
\newcommand{\etwo}{\textup{e}_2}
\newcommand{\TV}[1]{{\left\lVert #1 \right\rVert}_{\normalfont
\text{TV}}}

\def\Eb{\overset{\bullet}{\normalfont\textsf{E}}}
\def\Ew{\overset{\circ}{\normalfont\textsf{E}}}
\def\Nc{\normalfont\textsf{N}}
\def\Sc{\normalfont\textsf{S}}

\def\bE{\mathbb{E}}
\def\bN{\mathbb{N}}
\def\bP{\mathbb{P}}
\def\bQ{\mathbb{Q}}
\def\bR{\mathbb{R}}
\def\bZ{\mathbb{Z}}

\def\TF{\textup{TF}}

 \def\Z{\bZ}
\def\Q{\bQ}
\def\R{\bR}
\def\N{\bN}
\def\P{\bP}

\def\Pb{\mathbf{P}}

\usepackage{stackengine}

\newcommand{\dif}{\textup{d}}

\def\E{\bE}
\def\P{\bP} 

\definecolor{partcolor1}{rgb}{0.0,0.5,0.0}
\definecolor{partcolor2}{rgb}{0.0,0.5,0.0}

\definecolor{darkgreen}{rgb}{0.0,0.5,0.0}
\definecolor{darkblue}{rgb}{0.5,0.1,0.5}
\definecolor{deepblue}{rgb}{0.25,0.41,0.88}
\definecolor{nicosred}{rgb}{0.65,0.1,0.1}
\definecolor{light-gray}{gray}{0.7}
\allowdisplaybreaks[1]

\RequirePackage{datetime} 
\allowdisplaybreaks
\begin{document}
\usdate
\title[Approximating open ASEP in equilibrium]
{Approximating the stationary distribution of the ASEP with open boundaries}
\author{Evita Nestoridi}
\address{Evita Nestoridi, Princeton University and Stony Brook University, United States}
\email{evrydiki.nestoridi@stonybrook.edu}
\author{Dominik Schmid}
\address{Dominik Schmid, University of Bonn, Germany}
\email{d.schmid@uni-bonn.de}
\date{\today}
\begin{abstract}
We investigate the stationary distribution of asymmetric and weakly asymmetric simple exclusion processes with open boundaries. We project the stationary distribution onto a subinterval, whose size is allowed to grow with the length of the underlying segment. Depending on the boundary parameters of the exclusion process, we provide conditions such that the stationary distribution projected onto a subinterval is close in total variation distance to a product measure. 
\end{abstract}
\maketitle
\vspace*{-0.6cm}
\section{Introduction} \label{sec:Introduction}

The asymmetric simple exclusion process is an interacting particle system, which is intensively studied from various different perspectives; see  \cite{BSV:SlowBond,BE:Nonequilibrium,CW:TableauxCombinatorics,C:KPZReview,L:Book2} for a selection of surveys in statistical mechanics, probability theory and combinatorics on this model. 
In this article, we focus on the asymmetric simple exclusion process with open boundaries, also called the open ASEP. We consider a segment of length $N$ such that each site is either occupied by a particle or left empty. Each site is equipped with rate $1+q$ Poisson clocks. Whenever a clock rings and the respective site is occupied, we let the particle move to the right with probability $(1+q)^{-1}$, and to the left with probability $q(1+q)^{-1}$, provided the target is a vacant site. In addition, particles enter at the left-hand side boundary at rate $\alpha>0$, and exit at the right-hand side boundary at rate $\beta>0$.  Among the most fundamental tasks for exclusion processes is the characterization of their invariant measures, and to this end, many elaborate tools, such as the Matrix product ansatz and tableaux combinatorics are used to understand the equilibrium \cite{BECE:ExactSolutionsPASEP,CW:TableauxCombinatorics}. \\
A classical result by Liggett states that the exclusion process converges on any finite interval sufficiently far away from the boundary to a homogeneous product measure, when excluding the special case $\alpha=\beta<(1-q)/2$; see \cite{L:ErgodicI}. On the other hand, Bryc et al.\ characterize the stationary distribution on the macroscopic scale $N$ in a series of works using Askey--Wilson processes; see \cite{BW:AskeyWilsonProcess,BW:QuadraticHarnesses,BW:Density} for $q\in (0,1)$, and more recently \cite{BWW:ASEPtoKPZ} when $q\rightarrow 1$. Depending on the boundary parameters, they verify convergence in finite dimensional distribution of the height function representation of the open ASEP to sums of Brownian motion, Brownian excursion, and Brownian meander.
Both results motivate the following question: Under which conditions is the stationary distribution of the open ASEP projected onto a subinterval, whose size grows with $N$, close to a product measure?
It turns out that an answer to this question depends on the choice of boundary parameters for the open ASEP. In the fan region of the open ASEP, where informally speaking the effective density at the left end of the segment is larger than the effective density at the right end, we express the stationary distribution as certain re-weighted simple random walks. We then link the approximation of the stationary distribution to the question whether a certain random polymer model with a hard wall and pinning is localized. Conversely, in the shock region, where heuristically the effective density on the right end of the segment is larger than on the left end, we approximate the stationary distribution by coupling the open ASEP with different boundary parameters, and using a special representation for certain choices of $\alpha,\beta,q$. For the open TASEP, where $q=0$,  we approximate the stationary distribution by studying its formulation as a last passage percolation model on a strip; see \cite{ES:HighLow,S:MixingTASEP,SS:TASEPcircle}.

\subsection{Model and results}  \label{sec:ModelResults}

Formally, we define the \textbf{ASEP with open boundaries}, also called \textbf{open ASEP}, as a continuous-time Markov chain  $(\eta_t)_{t\geq 0}$  on  $\Omega_N=\{0,1\}^{N}$ for some $N \in \N$, and with generator
\begin{align}\label{def:Gen}
\begin{split}
\mathcal{L}f(\eta) &= \sum_{x =1}^{N-1} \big(\eta(x)(1-\eta(x+1)) + q \eta(x+1)(1-\eta(x)) \big) \left[ f(\eta^{x,x+1})-f(\eta) \right]  \\
 &+ \alpha (1-\eta(1))\ \left[ f(\eta^{1})-f(\eta) \right] \hspace{2pt} + \beta \eta(N)\left[ f(\eta^{N})-f(\eta) \right]  \end{split}
\end{align}
 for all functions $f\colon \Omega_{N} \rightarrow \R$.  Here, we use the standard notation
\begin{equation*}
\eta^{x,y} (z) = \begin{cases}
 \eta (z) & \textrm{ for } z \neq x,y\\
 \eta(x) &  \textrm{ for } z = y\\
 \eta(y) &  \textrm{ for } z = x \, 
 \end{cases} \qquad \text{ and } \qquad \eta^{w} (z) = \begin{cases}
 \eta (z) & \textrm{ for } z \neq w\\
 1-\eta(z) &  \textrm{ for } z = w \, 
 \end{cases}
\end{equation*}
to denote swapping of values in a configuration $\eta \in \Omega_{N}$ at sites $x,y \in \lsem N \rsem := \{1,\dots, \lfloor N \rfloor \}$, and flipping at $w\in \lsem N \rsem$, respectively. We say that site $x$ is \textbf{occupied} if $\eta(x)=1$, and \textbf{vacant} otherwise. A visualization of the open ASEP  is given in Figure~\ref{fig:ASEP}. 
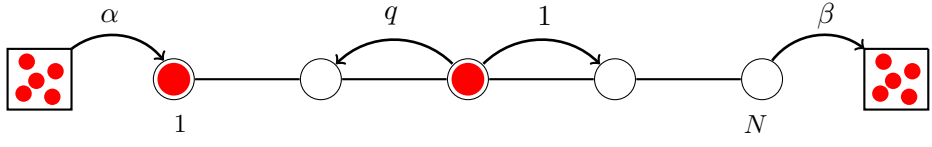
\begin{figure}
\centering
\begin{tikzpicture}[scale=0.85]

\def\spiral[#1](#2)(#3:#4:#5){
\pgfmathsetmacro{\domain}{pi*#3/180+#4*2*pi}
\draw [#1,
       shift={(#2)},
       domain=0:\domain,
       variable=\t,
       smooth,
       samples=int(\domain/0.08)] plot ({\t r}: {#5*\t/\domain})
}

\def\particles(#1)(#2){

  \draw[black,thick](-3.9+#1,0.55-0.075+#2) -- (-4.9+#1,0.55-0.075+#2) -- (-4.9+#1,-0.4-0.075+#2) -- (-3.9+#1,-0.4-0.075+#2) -- (-3.9+#1,0.55-0.075+#2);
  
  	\node[shape=circle,scale=0.6,fill=red] (Y1) at (-4.15+#1,0.2-0.075+#2) {};
  	\node[shape=circle,scale=0.6,fill=red] (Y2) at (-4.6+#1,0.35-0.075+#2) {};
  	\node[shape=circle,scale=0.6,fill=red] (Y3) at (-4.2+#1,-0.2-0.075+#2) {};
   	\node[shape=circle,scale=0.6,fill=red] (Y4) at (-4.45+#1,0.05-0.075+#2) {};
  	\node[shape=circle,scale=0.6,fill=red] (Y5) at (-4.65+#1,-0.15-0.075+#2) {}; }

  \def\annhil(#1)(#2){	  \spiral[black,thick](9.0+#1,0.09+#2)(0:3:0.42);
  \draw[black,thick](8.5+#1,0.55+#2) -- (9.5+#1,0.55+#2) -- (9.5+#1,-0.4+#2) -- (8.5+#1,-0.4+#2) -- (8.5+#1,0.55+#2); }

	\node[shape=circle,scale=1.5,draw] (B) at (2.3,0){} ;
	\node[shape=circle,scale=1.5,draw] (C) at (4.6,0) {};
	\node[shape=circle,scale=1.2,fill=red] (CB) at (2.3,0) {};
    \node[shape=circle,scale=1.5,draw] (A) at (0,0){} ;
 	\node[shape=circle,scale=1.5,draw] (D) at (6.9,0){} ; 
 	 	\node[shape=circle,scale=1.5,draw] (Z) at (-2.3,0){} ;
	\node[shape=circle,scale=1.2,fill=red] (YZ) at (-2.3,0) {};
   \node[line width=0pt,shape=circle,scale=1.6] (B2) at (2.3,0){};
	\node[line width=0pt,shape=circle,scale=2.5] (D2) at (6.9,0){};
		\node[line width=0pt,shape=circle,scale=2.5] (Z2) at (-2.3,0){};
		
	\node[line width=0pt,shape=circle,scale=2.5] (X10) at (6.8,0){};
		\node[line width=0pt,shape=circle,scale=2.5] (X11) at (-2.2,0){};	
			\node[line width=0pt,shape=circle,scale=2.5] (X12) at (8.4,0){};
		\node[line width=0pt,shape=circle,scale=2.5] (X13) at (-3.8,0){};

		\draw[thick] (Z) to (A);	
	\draw[thick] (A) to (B);	
		\draw[thick] (B) to (C);	
  \draw[thick] (C) to (D);

\particles(0)(0);
\particles(6.9+4.9+1.6)(0);


\draw [->,line width=1pt]  (B2) to [bend right,in=135,out=45] (C);
  
   \draw [->,line width=1pt] (B2) to [bend right,in=-135,out=-45] (A);
    \node (text1) at (3.5,1){$1$} ;    
	\node (text2) at (1.1,1){$q$} ;   
	\node (text3) at (-2.3-1,1){$\alpha$}; 
	\node (text4) at (6.9+1,1){$\beta$};

    \node[scale=0.9] (text1) at (-2.2,-0.7){$1$} ;    
    \node[scale=0.9] (text1) at (6.8,-0.7){$N$} ;   
  	
  \draw [->,line width=1pt] (-3.9,0.475) to [bend right,in=135,out=45] (Z);
   \draw [->,line width=1pt] (D) to [bend right,in=135,out=45] (6.9+1.6,0.475);

	\end{tikzpicture}	
\caption{\label{fig:ASEP}Simple exclusion process with open boundaries for parameters $(q,\alpha,\beta)$.}
 \end{figure}
It is easy to verify that the open ASEP has for all choices of $q\in [0,1]$ and $\alpha,\beta>0$ a unique stationary distribution $\mu=\mu^{N,q,\alpha,\beta}$.
In the following, we study $\mu$ for different choices of $q \in [0,1)$ and $\alpha,\beta >0$ when projecting to subintervals. For  $I= \lsem a, b\rsem := \Z \cap [a,b]$, we let $\eta_I \in \{0,1\}^{|I|}$ with $ \eta_I(x) = \eta(x+\lceil a \rceil -1 )$ for $x\in \lsem |I| \rsem$, and set for all configurations $\zeta \in \{0,1\}^{|I|}$
\begin{equation}\label{def:ProjectedMu}
\mu_I(\zeta) := \sum_{ \eta \, : \, \eta_I = \zeta } \mu^{N,q,\alpha,\beta}(\eta) \, .
\end{equation} In other words, we denote by $\mu_I$ the probability measure where we project $\mu$ onto the coordinates in $I$, and similarly for a  configuration $\eta_{I}$. As a standard measure of distance, for two probability measure $\nu,\nu^{\prime}$ on $\Omega_{N}$, let
\begin{equation}\label{def:TVDistance}
\TV{ \nu - \nu^{\prime} } := \frac{1}{2}\sum_{x \in \Omega_{N}} \abs{\nu(x)-\nu^{\prime}(x)} = \max_{A \subseteq \Omega_{N}} \left(\nu(A)-\nu^{\prime}(A)\right)
\end{equation} be the \textbf{total variation distance} between $\nu$ and $\nu^{\prime}$; see Chapters 4 and 5 in \cite{LPW:markov-mixing} for further equivalent formulations. In the following, let $ \textup{Ber}_I(\rho)$ denote the Bernoulli-$\rho$-product measure on $\{0,1\}^{|I|}$ for an interval $I$. 
We start with the following result by Liggett on the local structure of the stationary measure of the open ASEP.  
\begin{proposition}[Liggett \cite{L:ErgodicI}]\label{pro:Finite} Let $q\in [0,1)$, $\alpha>0$ and $\beta>0$. Let $C \in \N$ and consider a family of finite intervals $I=\lsem a_N, b_N\rsem$ with $|I|= C$ and 
\begin{equation}\label{eq:LiggettAssumption}
\lim_{N \rightarrow \infty} \min(a_N,N-b_N) = \infty . 
\end{equation} We have the following weak convergence of $\mu_I$:
\begin{equation}
\mu_{I} \overset{}{\rightarrow} \begin{cases}  \textup{Ber}_I\Big(\frac{\alpha}{1-q}\Big) & \text{ if } \alpha < \min\Big( \beta,\frac{1-q}{2}\Big) \\
\textup{Ber}_I\Big(1- \frac{\beta}{1-q}\Big) &  \text{ if } \beta < \min\Big( \alpha,\frac{1-q}{2}\Big) \\
\textup{Ber}_I\Big(\frac{1}{2}\Big) & \text{ if } \min(\alpha,\beta) > \frac{1-q}{2}  .
\end{cases}
\end{equation}
\end{proposition}
Let us point out that \cite{L:ErgodicI} strictly speaking only covers the case $\alpha=\beta=1$. However, for general parameters $\alpha,\beta>0$, the result  follows along the lines of Theorem~3.29 in Part III of \cite{L:Book2}. For the sake of completeness, we include a proof in the appendix.
Motivated by Proposition \ref{pro:Finite}, let $q\in [0,1)$ and $\alpha,\beta>0$, and define
\begin{equation}\label{def:uAndv}
u=u(\alpha,q):=\frac{1-q}{\alpha}-1 \in (-1,\infty)\quad \text{ and } \quad v=v(\beta,q):=\frac{1-q}{\beta}-1 \in (-1,\infty) . 
\end{equation} We say that the open ASEP  is in the \textbf{high density phase} if $v > \max(u,1)$, it is in the \textbf{low density phase} if $u > \max(1,v)$, and it is in the \textbf{maximal current phase} if $\max(u,v)<1$. 
Moreover, we distinguish between the \textbf{fan region} of the ASEP with open boundaries where $uv < 1$ and the \textbf{shock region} of the ASEP with open boundaries where $uv > 1$. The different phases for the open ASEP are visualized at the left-hand side of Figure \ref{fig:Regimes}. Let us remark that when $uv=1$, it is straightforward to verify that the stationary distribution of the open ASEP has a product form. 
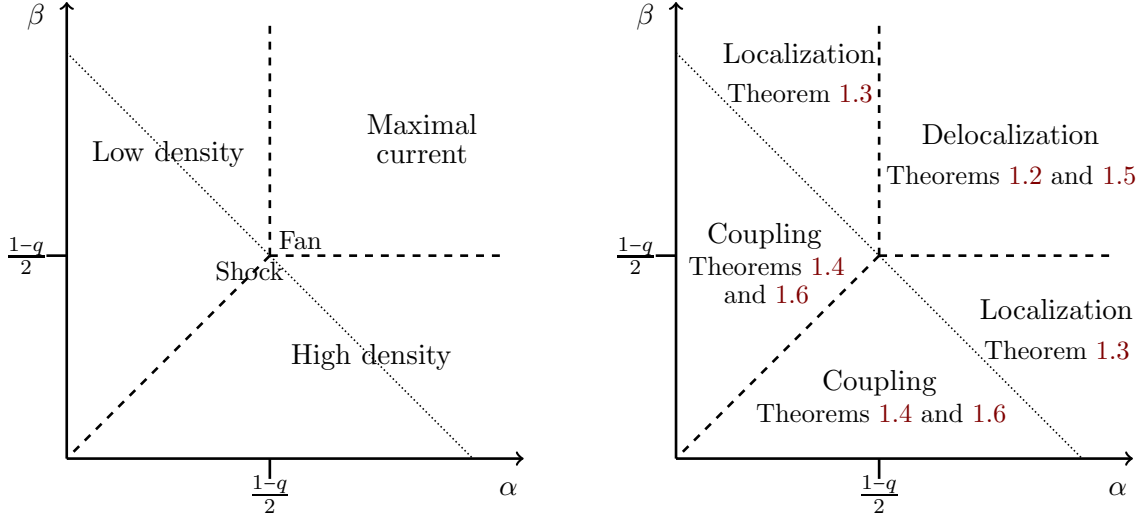
\begin{figure}
    \centering
\begin{tikzpicture}[scale=1.35]
\draw [->,line width=1pt] (0,0) to (4.5,0); 	
\draw [->,line width=1pt] (0,0) to (0,4.5); 	
\draw [line width=1pt] (0,2) to  (-0.2,2); 	
\draw [line width=1pt] (2,0) to  (2,-0.2); 		
\draw [line width=1pt,dashed] (2,2) to  (0,0); 		
\draw [line width=1pt,dashed] (2,2) to (4.3,2); 	
\draw [line width=1pt,dashed] (2,2) to (2,4.3); 		
 \node (H1) at (-0.4,2) {$\frac{1-q}{2}$};	 
 \node (H2) at (2,-0.35) {$\frac{1-q}{2}$};	 
 \node (H3) at (-0.3,3.85+0.5) {$\beta$};	 
 \node (H4) at (3.85+0.5,-0.3) {$\alpha$};

 \node (X1) at (3.5,3.3) {Maximal};
 \node (X11) at (3.5,3) {current};
 \node (X2) at (1,3) {Low density};
 \node (X3) at (3,1) {High density};

 \node (X4) at (2.3,2.15) {\small Fan}; 

 \node (X45) at (1.8,1.85) {\small Shock};


 \draw [line width=0.6pt,densely dotted] (0,4) to (4,0);

\def\x{6}; 
 
%
%
%
%
%
%

\draw [->,line width=1pt] (0+\x,0) to (4.5+\x,0); 	
\draw [->,line width=1pt] (0+\x,0) to (0+\x,4.5); 	
\draw [line width=1pt] (0+\x,2) to  (-0.2+\x,2); 	
\draw [line width=1pt] (2+\x,0) to  (2+\x,-0.2); 		
\draw [line width=1pt,dashed] (2+\x,2) to  (0+\x,0); 		
\draw [line width=1pt,dashed] (2+\x,2) to (4.3+\x,2); 	
\draw [line width=1pt,dashed] (2+\x,2) to (2+\x,4.3); 		
 \node (H1) at (-0.4+\x,2) {$\frac{1-q}{2}$};	 
 \node (H2) at (2+\x,-0.35) {$\frac{1-q}{2}$};	 
 \node (H3) at (-0.3+\x,3.85+0.5) {$\beta$};	 
 \node (H4) at (3.85+0.5+\x,-0.3) {$\alpha$};

 \draw [line width=0.6pt,densely dotted] (0+\x,4) to (4+\x,0); 	

 \node (X1) at (3.3+\x,3.2) {Delocalization};
 \node (X11) at (3.3+\x,2.8) {\small Theorems \ref{thm:MaxCurrent} and \ref{thm:MaxCurrentWASEP}};
 
 \node (X2) at (1.22+\x,4) {Localization};
 \node (X11) at (1.22+\x,3.6) {\small Theorem \ref{thm:HighLowFan}};

 \node (X3) at (3.75+\x,1.47) {Localization};
 \node (X11) at (3.75+\x,1.07) {\small Theorem \ref{thm:HighLowFan}};

 \node (X4) at (2.+\x,0.75) {Coupling};
 \node (X11) at (2.02+\x,0.45) {\small Theorems \ref{thm:HighLowShock} and \ref{thm:HighLowShockWASEP}};
 
 \node (X5) at (0.77+\x+0.1,2.2) {Coupling};
 \node (X11) at (0.79+\x+0.1,1.9) {\small Theorems \ref{thm:HighLowShock}};
  \node (X12) at (0.79+\x+0.1,1.62) {\small and \ref{thm:HighLowShockWASEP}};

\end{tikzpicture}
    \caption{On the left-hand side, we see the different phases for the open ASEP with respect to the two boundary parameters $\alpha,\beta>0$. On the right-hand side, we provide an overview on the different regimes covered in our main theorems on approximating the stationary distribution.}
    \label{fig:Regimes}
\end{figure}
\subsubsection{The asymmetric simple exclusion process with open boundaries}

We state now our first result on the stationary distribution of the open ASEP in the maximal current phase. 

\begin{theorem}\label{thm:MaxCurrent} Consider an interval $I=\lsem a,b \rsem$ with $\min(a,N-b) \gg \max(|I|,\log^2(N))$\footnote{In the following, for any pair of functions $f,g \, \colon \,  \N \rightarrow \R$, we write 
\begin{equation*}
f \gg g \quad \Leftrightarrow \quad \liminf_{N \rightarrow \infty} \frac{f(N)}{g(N)} = \infty  .
\end{equation*}} for $a=a(N)$ and $b=b(N)$, and $\max(u,v)<1$, i.e.\ the maximal current phase. Then 
\begin{equation}
\lim_{N \rightarrow \infty} \TV{ \mu_I - \textup{Ber}_{I}\left(\frac{1}{2}\right)} = 0.
\end{equation} 
\end{theorem}
 Note that the above result is optimal in the sense that it can not be extended to macroscopic intervals with a size of order $N$, as this is ruled out by the characterization of the macroscopic densities in \cite{BW:Density}. However, the stationary distribution is well-approximated even on macroscopic intervals in the fan region of the high and the low density phase. 

\begin{theorem}\label{thm:HighLowFan}
Consider an interval $I=\lsem a,b \rsem$ with $\min(a,N-b) \gg 1$. 
 Moreover let $uv \leq 1$, and $u > \max(1,v)$, i.e.\ we consider the low density phase of the ASEP with open boundaries in the fan region. Then we have
\begin{equation}\label{eq:HighLowFan}
\lim_{N \rightarrow \infty}\TV{\mu_I - \textup{Ber}_I\left(\frac{\alpha}{1-q}\right)}  = 0 .
\end{equation}  By symmetry, a similar statement holds for the high density phase of the open ASEP. 
\end{theorem}

Our arguments are specific to the fan region. In order to obtain similar approximation results in the shock region, we compare the invariant measure to Bernoulli shock measures; see Section \ref{sec:ShockPhase} for a precise definition. We use a representation for the open ASEP in the special case where there exists some $k\in \N_0 := \{0,1,\dots\}$ such that
\begin{equation}\label{eq:FiniteRepTheorems}
u v q^{k} = 1. 
\end{equation} We have the following result on approximating the stationary measure in the shock region.

\begin{theorem}\label{thm:HighLowShock} Consider the low density phase in the shock region with $\alpha,\beta>0$. Assume there exist $\beta^{\prime}$ and $\beta^{\prime\prime}$ with $\beta \in [\beta^{\prime},\beta^{\prime\prime}]$, and some constant $k\in \mathbb{N}$ such that the respective parameters $v^{\prime}:=v(\beta^{\prime},q)$ and $v^{\prime\prime}:=v(\beta^{\prime\prime},q)$ from \eqref{def:uAndv} satisfy
\begin{equation}\label{eq:ShockConditions}
u > \max(v^{\prime},v^{\prime\prime},1)  \quad\text{ and } \quad uv^{\prime}  q^{k} = uv^{\prime\prime} q^{k-1} = 1 .
\end{equation}  Then for all $I=\lsem a,b \rsem$ with $(N-b) \gg 1$, we get
\begin{equation}
\lim_{N \rightarrow \infty}\TV{\mu_I - \textup{Ber}_I\left(\frac{\alpha}{1-q}\right)}  = 0 .
\end{equation} By symmetry, a similar statement holds for the high density phase of the open ASEP. 
\end{theorem}

The strategies in our theorems for the different regimes of the open ASEP are visualized at the right-hand side of Figure \ref{fig:Regimes}.

\subsubsection{The weakly asymmetric simple exclusion process with open boundaries}

In the following, we consider the \textbf{weakly asymmetric simple exclusion process with open boundaries}, also called \textbf{open WASEP}, i.e.\ we let the bias parameter $q=q(N)$ be
\begin{equation}\label{eq:qAssumption}
q = \exp(-c_q N^{-\varepsilon}) = 1- c_q N^{-\varepsilon} + \mathcal{O}(N^{-2\varepsilon})
\end{equation} for some $\varepsilon>0$ and $c_q>0$. Moreover, we assume that $\alpha=\alpha(N)$ and $\beta=\beta(N)$ are such that
\begin{equation}\label{def:uAndvWASEP}
u=u(\alpha,q,N)=\frac{1-q}{\alpha}-1 \in (-1,\infty)\quad \text{ and } \quad v=v(\beta,q,N)=\frac{1-q}{\beta}-1 \in (-1,\infty)
\end{equation} do not depend on $N$. Note that this implies that $\alpha(N)$ and $\beta(N)$ are of order $N^{-\varepsilon}$, while the effective density at the left and right end of the segment remains constant.
 We have the following result on approximating the stationary distribution of the open WASEP in the maximal current phase, similarly to Theorem \ref{thm:MaxCurrent}.

\begin{theorem}\label{thm:MaxCurrentWASEP} Consider an interval $I=\lsem a,b \rsem$ and take $q$ from \eqref{eq:qAssumption} with $\varepsilon>0$.
Assume
\begin{equation}\label{eq:AssumptionsMaxCurrentWASEP}
\min(a,N-b) \gg \begin{cases} \max\Big( N^{2\varepsilon}\log^2(N) , |I| \Big) & \text{ if } u+v \leq 0 \text{ and } \varepsilon < \frac{1}{2} \\ 
\max\Big( N^{3\varepsilon}\log(N) , |I| \Big) & \text{ if } u+v > 0 \text{ and } \varepsilon < \frac{1}{3} 
\end{cases}
\end{equation} 
for $\max(u,v)<1$, i.e.\ we consider the maximal current phase. Then 
\begin{equation}
\lim_{N \rightarrow \infty} \TV{ \mu_I - \textup{Ber}_{I}\left(\frac{1}{2}\right)} = 0.
\end{equation} 
\end{theorem}

We believe that the exponent $2\varepsilon$ in \eqref{eq:AssumptionsMaxCurrentWASEP} and the restriction $\varepsilon<\frac{1}{2}$ are optimal; see \cite{CK:StationaryKPZ} for a seminal result by Corwin and Knizel when $\varepsilon=\frac{1}{2}$ on relating the stationary distribution of the open WASEP to the open KPZ equation under a slightly different choice of boundary parameters, but again with a constant effective density at the boundaries. Next, we consider the open WASEP in the shock regime. 

\begin{theorem}\label{thm:HighLowShockWASEP} Consider the shock region of the low density phase, i.e.\ $uv>1$ and $u>\max(1,v)$, and let $I=\lsem a,b \rsem$ with 
\begin{equation}
\min(a,N-b) \gg N^{\varepsilon} . 
\end{equation} For $q$ from \eqref{eq:qAssumption} with $\varepsilon>0$, and $u,v$ from \eqref{def:uAndvWASEP}, we get
\begin{equation}
\lim_{N \rightarrow \infty}\TV{\mu_I - \textup{Ber}_I\left(\frac{\alpha}{1-q}\right)}  = 0 .
\end{equation} A similar statement holds for the high density phase of the open WASEP in the shock region. 
\end{theorem}

Let us stress that in contrast to Theorem \ref{thm:HighLowShock}, our approximation result covers the entire shock regime of the open WASEP whenever $u\neq v$.  
We conjecture that Theorem~\ref{thm:HighLowShock} extends to the entire high and low density phase of the open WASEP.
\begin{conjecture}\label{conj:HighLow}
Consider the low density phase $u>\max(1,v)$, and let $I=\lsem a,b \rsem$ satisfy
\begin{equation}
\min(a,N-b) \gg N^{\varepsilon} . 
\end{equation} For $q$ from \eqref{eq:qAssumption} with $\varepsilon > 0$, we get
\begin{equation}
\lim_{N \rightarrow \infty}\TV{\mu_I - \textup{Ber}_I\left(\frac{\alpha}{1-q}\right)}  = 0 .
\end{equation} A similar statement holds for the high density phase $v>\max(1,u)$ of the open WASEP. 
\end{conjecture}


\subsubsection{The totally asymmetric simple exclusion process with open boundaries}

Let us note that in the above Theorems \ref{thm:MaxCurrent} to \ref{thm:HighLowShockWASEP}, we exclude the boundaries between the different phases. For the special $q=0$, the  \textbf{TASEP with open boundaries}, also called \textbf{open TASEP}, the following theorem covers the entire range of boundary parameters, apart from the so-called co-existence line where $\alpha=\beta< \frac{1}{2}$. To achieve this, we rely on an alternative approach to approximating the stationary distribution using the representation of the open TASEP as a last passage percolation model on a strip; see \cite{ES:HighLow,S:MixingTASEP}. 
 \begin{theorem}\label{thm:TASEP}
 Consider an interval $I=\lsem a,b \rsem$ with $a=a(N)$ and $b=b(N)$, and let $q=0$. If $\min(\alpha,\beta) \geq \frac{1}{2}$ and there exists some $\delta>0$ such that
$\min(a,N-b) \geq \delta N \gg  |I|$, then
\begin{equation}\label{eq:TASEPStatementMaxCurrent}
\lim_{N \rightarrow \infty} \TV{ \mu_I - \textup{Ber}_{I}\left(\frac{1}{2}\right)} = 0.
\end{equation} If $\alpha < \min(\frac{1}{2},\beta)$ and $N-b \gg N^{\frac{1}{3}}\log(N)$ then 
\begin{equation}\label{eq:TASEPStatementLow}
\lim_{N \rightarrow \infty} \TV{ \mu_I - \textup{Ber}_{I}\left(\alpha\right)} = 0.
\end{equation}
Similarly, if $\beta < \min(\frac{1}{2},\alpha)$ and $a \gg N^{\frac{1}{3}}\log(N)$ then 
\begin{equation}\label{eq:TASEPStatementHigh}
\lim_{N \rightarrow \infty} \TV{ \mu_I - \textup{Ber}_{I}\left(1-\beta\right)} = 0.
\end{equation}
\end{theorem}
Note that we impose slightly stronger assumptions on the size and location of the segment compared to the previous theorems. We expect the above results to hold also for the open ASEP with any constant $q\in (0,1)$, and that the assumptions on the location and the size of the segment can be weakened to match the assumptions in Theorems \ref{thm:MaxCurrent} to \ref{thm:HighLowShock}.

\subsection{Related work} \label{sec:RelatedWork}

Exclusion processes are among the most studied examples of interacting particle systems, introduced to the mathematical literature by Spitzer in \cite{S:InteractionMP} over 50 years ago; see \cite{L:Book2} for a more comprehensive discussion. The open ASEP was studied by Liggett in \cite{L:ErgodicI}  who provides a remarkable recursive construction of the stationary measure. In \cite{DEHP:ASEPCombinatorics}, Derrida et al.\ introduce the Matrix product ansatz as a celebrated tool to represent the stationary distribution of the open TASEP; see also \cite{BECE:ExactSolutionsPASEP} for an extension to the open ASEP, and \cite{BE:Nonequilibrium} for an introductory survey to this technique. These insights led to countless articles on the stationary distribution of the open ASEP from various different perspectives. In a series of papers, Bryc et al.\ investigate the moment generating function of the stationary distribution, which allows to characterize the limiting density fluctuations and to obtain large deviations for the number of particles under the stationary distribution, among other applications \cite{BW:AskeyWilsonProcess,BW:QuadraticHarnesses,BW:Density}. A key tool are Askey--Wilson processes, related to Askey--Wilson polynomials found in \cite{USW:PASEPcurrent} when investigating the current of the open ASEP, to order to study the Laplace transform of the height function representation of the open ASEP. In turns out that, depending on the different phases in the fan region, the height function converges to the sum of a Brownian motion and a Brownian excursion in the maximal current phase, to a Brownian motion in the high and the low density phase, and to a sum of Brownian motion and a Brownian meander on the boundary of the phases. Very recently, the phase diagram in the shock region was established by Wang et al.\ using signed Askey--Wilson measures \cite{WWY:ASEPshock}, verifying that the height function converges to a Brownian motion in the entire high and low density phase. \\

The Matrix product ansatz serves as a key tool for various combinatorial interpretations of the stationary distribution of open ASEP using lattice paths and tableaux, allowing also for more general boundary parameters, see  \cite{BCEPR:CombinatoricsPASEP,CW:TableauxCombinatorics,M:TASEPCombinatorics,SW:CombinatoricSemiopen} for a non-exhaustive list, and \cite{W:SurveyCombinatorics} for a survey. We are particularly interested in the representation of the stationary distribution of the open ASEP given in \cite{BCEPR:CombinatoricsPASEP} as weighted bi-colored Motzkin paths; see also  Derrida et al.\ in  \cite{DEL:DensityExcursion} for the special case of the open TASEP described in Section~\ref{sec:TASEPwithopenCombinatorics}. Note that studying the fluctuations of random bi-colored Motzkin paths is also of independent interest; see \cite{BW:Motzkin2,BW:Motzkin1}.
Let us further mention that the above combinatorial representations of the invariant measure can be extended to multispecies exclusion processes, offering also an alternative descriptions of the invariant measure in terms of queuing systems, as well as to asymmetric exclusion processes on the circle \cite{A:TASEPring, CGGW:KoornwinderMulti,CMW:TwoSpecies,FM:TASEPmulti,M:CombinatoricsMultispecies,M:StationaryASEP}.  \\

When the Matrix product ansatz allows for a representation using only finite dimensional matrices, much more can be said about the stationary distribution. Jafarpour and Masharian note in \cite{JM:Finite} that the invariant measure can be written as a convex combination of so-called shock measures; see Section \ref{sec:ShockMeasuresSchutz}, and \cite{FG:ApproximationSymmetric} for a very recent similar result for the symmetric simple exclusion process with open boundaries. Recently, Schütz established a much deeper relation in this case between the open ASEP and an exclusion process on a closed segment, introducing the concept of reverse duality  \cite{S:ASEPReverse}. Let us mention that the Matrix product ansatz is used in physics to investigate higher moments of the current of open ASEP or multispecies systems -- see \cite{AR:ExactPhaseDiagram,L:MatrixAnsatz,U:TwoSpeciesPASEP} --  and can be rigorously extended also to other models; see for example recent work by Yang for the six-vertex model on a strip \cite{Y:SixVertex}. \\

For asymmetric simple exclusion processes out of equilibrium, mixing and relaxation times are a standard way to study the speed of convergence to the stationary distribution, see  \cite{BN:CutoffASEP,LL:CutoffASEP,LL:CutoffWeakly} for results on the ASEP and WASEP on a closed segment, and \cite{GE:ExactSpectralGap,GE:BetheAnsatzPASEP,ES:HighLow,GNS:MixingOpen,S:MixingTASEP,SS:TASEPcircle} for the open ASEP and TASEP. 
Let us remark at this point that due to the representation as a last passage percolation model, many properties such as current fluctuations or the existence of the TASEP speed process are very well understood for totally asymmetric simple exclusion processes -- see for example  \cite{AAV:TASEPSpeed,BBCS:Halfspace,FS:SpaceTimeCovariance} -- while corresponding results for asymmetric simple exclusion processes are sparse and often only subject to recent breakthroughs \cite{A:CurrentFluctuations,ACG:ASEPspeed,BS:OrderCurrent,BBCW:HalfspaceASEP,H:Boundary}. We will see in  Section \ref{sec:LPP} that the set of available techniques significantly increases when approximating the stationary distribution in the special case  of the open TASEP. Finally, let us stress that the open ASEP plays a crucial role in the KPZ university class as it allows under a suitable weakly scaling to construct a unique stationary solution to the open KPZ equation; see  \cite{BL:StationaryKPZ,BKWW:MarkovKPZ,BWW:ASEPtoKPZ,CK:StationaryKPZ,CS:OpenASEPWeakly,KM:StrongFeller,P:KPZlimit} for a non-exhaustive list of seminal  articles on this relation, and \cite{C:SurveyStationary} for a recent survey by Corwin. 

\subsection{Main ideas and concepts} \label{sec:OrganizationPaper}

In order to approximate the stationary distribution in the fan region, we rely on \cite{BCEPR:CombinatoricsPASEP} to express the stationary distribution of the open ASEP on a segment of length $N$ as a weighted lazy simple random conditioned to stay non-negative, and to return to the origin after $N$ steps. In order to study these weighted random walks, we borrow ideas from random polymer models. In the maximal current phase of the open ASEP, we apply a change of measure to the uniform distribution on the set of bi-colored Motzkin paths. In the maximal current phase of the open WASEP, we use Holley's inequality and  stochastic domination on the set of bi-colored Motzkin paths when $u+v\leq 0$, and a supermartingale argument when $u+v>0$ in order to control the weight and fluctuations of typical paths. In all three strategies, our goal is to show that the corresponding random polymer is delocalized, i.e.\ that a random walk path sampled according to the polymer measure is likely to spend most of the time away from the $x$-axis. Let us note that this agrees with the usual notion of delocalization for random polymers, i.e.\ that the (normalized) free energy is zero; see also Section \ref{sec:IntroductionPolymers} and Remark \ref{rem:PartitionFunctionRestricted}. In the fan region of the high and low density phase of the open ASEP, we use renewal techniques in order to couple the weighted random walk trajectories with a suitable bi-infinite polymer. In particular, a random walk path sampled according to the polymer measure is likely to remain very close to the $x$-axis, or again equivalently, the associated free energy is strictly positive. In the shock regime of the open ASEP and open WASEP, we rely on the canonical coupling to compare exclusion processes with different boundary parameters, and an explicit expression of the invariant measure as a sum of shock measures, using an ASEP on a closed segment with finitely many particles and particle-depending hopping rates. Finally, for the open TASEP, we exploit its equivalent formulation as a last passage percolation on the strip, and we use recent results on the coalescence of geodesics, as well as the random extension and time-change technique from \cite{SS:TASEPcircle} in the maximal current phase, in order to approximate the stationary distribution. 

\subsection{Outline of the paper} \label{sec:OutlinePaper}

This paper is structured as follows.  In Section~\ref{sec:CombinatorialRep}, we review combinatorial representations of the invariant measure of the ASEP with open boundaries. 
In Section  \ref{sec:CharacterizationASEP}, we establish the correspondence between the representation of the invariant measure as bi-colored Motzkin paths and a random polymer model with a hard wall and pinning. 
In Section \ref{sec:DelocalizationMaxCurrent}, we consider the open ASEP and open WASEP in the maximal current phase, and show that the corresponding random polymer model is delocalized. This allows us to conclude Theorem \ref{thm:MaxCurrent} and Theorem \ref{thm:MaxCurrentWASEP}. 
In Section \ref{sec:LocalizationPhase}, we study the high and the low density phase of the open ASEP in the fan region, and show that the corresponding random polymer localizes, allowing us to obtain Theorem \ref{thm:HighLowFan}. In Section \ref{sec:ShockPhase}, we treat the shock region of the high and low density phase of the open ASEP and open WASEP, proving Theorem \ref{thm:HighLowShock} and Theorem \ref{thm:HighLowShockWASEP}. In Section \ref{sec:LPP}, we investigate the open TASEP using last passage percolation, and hereby establish Theorem \ref{thm:TASEP}. Let us remark that Sections \ref{sec:ShockPhase} and~\ref{sec:LPP} can be read independently, while Sections \ref{sec:CombinatorialRep} to \ref{sec:LocalizationPhase} build on top of each other.

\section{Combinatorial representation of the stationary distribution}\label{sec:CombinatorialRep}

We start by recalling combinatorial representations of the stationary distribution of the open ASEP. This includes the celebrated Matrix product ansatz and a representation of its weights using bi-colored Motzkin paths. 

\subsection{Recursive construction of the stationary distribution}\label{sec:MatrixProduct}

The idea to use a recursive construction in order to express the stationary distribution of the open ASEP is due to Liggett \cite{L:ErgodicI}. In a celebrated result, Derrida et al.\ introduce the framework of the Matrix product ansatz, where the weight of each configuration in the stationary distribution is represented as a product of, in general infinite dimensional, matrices; see also Appendix \ref{sec:FiniteMPAAppendix}. In the following, we recall the recursive construction of the stationary distribution from  \cite{BCEPR:CombinatoricsPASEP}. \\

Let $\Omega = \{ \emptyset \} \cup \bigcup_{n \in \N} \Omega_N$ be the set of all $\{0,1\}$ configurations with arbitrary length, including the configuration of length zero, and we set $\ell(\eta)=n$ when $\eta$ has length $n$. Moreover, with a slight abuse of notation, for all $\eta,\zeta \in \Omega$, we write $\xi=[\eta,\zeta]$ for the configuration in $\Omega$, where we concatenate $\eta$ and $\zeta$, i.e. $\xi \in \Omega_{\ell(\eta)+\ell(\zeta)}$ with
\begin{equation}
\xi(x) = \begin{cases}
\eta(x) & \text{ if } 1 \leq x \leq \ell(\eta) \\
\zeta(x-\ell(\eta)) & \text{ if } \ell(\eta) < x \leq \ell(\eta)+\ell(\zeta) .
\end{cases}
\end{equation}
We say that $B \colon \Omega \rightarrow \R$ is a \textbf{basic weight function} if the following relations hold:
\begin{align}\label{def:AlgebraBasic}
\begin{split}
B(\emptyset)&=1   \\
B(\eta) &= \alpha B([0,\eta]) \\
B(\eta) &= \beta B([\eta,1]) \\
B([\eta,0,\zeta])+ B([\eta,1,\zeta]) &=  B([\eta,1,0,\zeta]) - q B([\eta,0,1,\zeta])
\end{split}
\end{align} for all $\eta,\zeta \in \Omega$. The following statement is Theorem 1 in \cite{BCEPR:CombinatoricsPASEP}.
\begin{lemma}[Brak et al.\ \cite{BCEPR:CombinatoricsPASEP}]\label{lem:Weights}
Let $B$ be a basic weight function on $\Omega$. Then for all $\eta \in \Omega_{N}$,
\begin{equation}\label{eq:MPAIdentity}
\mu^{N,q,\alpha,\beta}(\eta)=\mu(\eta) = \frac{B(\eta)}{Z_N} \quad \text{ where  } \ Z_N= \sum_{\zeta \in \Omega_N} B(\zeta) .
\end{equation} 
\end{lemma}
The quantity $Z_N$ in \eqref{eq:MPAIdentity} is called the \textbf{partition function} for the open ASEP. It is closely related to various statistics of the open ASEP,  for example the average current of particles through a segment of size $N$ is given by $Z_{N-1}/Z_N$; see \cite{BE:Nonequilibrium} for an overview, and  \cite{J:PASEP} for an alternative combinatorial representation of the partition function.   

\subsection{Bi-colored Motzkin paths} \label{sec:PathRepresentation}

In order to construct a basic weight function for given parameters $\alpha,\beta,q$, we again follow the approach in \cite{BCEPR:CombinatoricsPASEP}. Recall $\N_0 = \{0,1,2,\dots\}$ and define
\begin{align}\label{def:MotzkinPaths}
\begin{split}\textup{MP}_N := \Big\{ (v_0,v_1,\dots,v_N) \colon \  &v_{i-1} \in \N_0 \times \N_0 \text{ for all } i \in \lsem N+1 \rsem , \ v_0 = (0,0), \ v_N=(N,0), \\
&v_i-v_{i-1} \in \{ (1,1),(1,0),(1,-1)\} \text{ for all } i\in  \lsem N \rsem \Big\}
\end{split}
\end{align} to be the set of \textbf{Motzkin paths} of length $N$, i.e.\ the set of all lattice path starting from the origin, which stay non-negative, perform only horizontal or diagonal steps to the right, and return to the $x$-axis after $N$ steps.
We require in the following a modified version of Motzkin paths, where we color all horizontal steps. More precisely, let $\mathcal{A}$ be the set
\begin{equation}\label{def:Acal}
\mathcal{A}:= \left\{ \Nc, \Eb, \Ew, \Sc  \right\} , 
\end{equation} where we refer to $\Nc$ as a north step, to $\Sc$ as a south step, and  the remaining elements as east steps. For each $\omega \in \mathcal{A}^{N}$, we assign a lattice path $v_{\omega}= (v_0,v_1,\dots,v_N)$ by $v_0=(0,0)$~and 
\begin{equation}\label{eq:BiColorToPath}
v_i - v_{i-1} = \begin{cases}
(1,1) &\text{ if } \omega(i)=\Nc \\
(1,0) &\text{ if } \omega(i) \in \Big\{\Eb, \Ew \Big\} \\
(1,-1) &\text{ if } \omega(i)=\Sc 
\end{cases}
\end{equation} for all $i\in \lsem N \rsem$. The set of \textbf{bi-colored Motzkin paths} is now given as
\begin{equation}\label{def:BiColoredMotzkin}
\Psi_N := \left\{ \omega \in \mathcal{A}^{N} \colon v_{\omega} \in \textup{MP}_N \right\} . 
\end{equation} For each $\omega \in \Psi_N$, we define its \textbf{height} at position $i \in \{ 0,1,\dots,N\}$ by
\begin{equation}\label{def:HeightPath}
h_i(\omega) := \sum_{j=1}^{i} \mathds{1}_{ \{\omega(j)=\Nc\} } - \mathds{1}_{ \{\omega(j)=\Sc \}} 
\end{equation} i.e.\ $h_i(\omega)$ corresponds to the height at position $i$ in the Motzkin path $v_{\omega}$ associated to $\omega$. With a slight abuse of notation, we will also write $h_i(\zeta)$ for the height of a Motzkin path $\zeta \in \textup{MP}_N$ at position $i\in \lsem N \rsem$ as the height is independent of the coloring.
Moreover, for fixed  $\alpha,\beta>0$ and $q\in [0,1)$, we define for each $\omega \in \Psi_N$ its \textbf{weight} at $i\in  \lsem N \rsem$ by
\begin{equation}\label{def:WeightsIndividually}
W_i(\omega) := \begin{cases}
(1-q)\big(1-q^{h_i(\omega)+1}\big) &\text{ if } \omega(i)=\Nc \\
(1-q)\big(1+uq^{h_i(\omega)}\big) &\text{ if } \omega(i)=\Eb \\
(1-q)\big(1+vq^{h_i(\omega)}\big) &\text{ if } \omega(i)=\Ew \\
(1-q)\big(1-uvq^{h_i(\omega)-1}\big) &\text{ if } \omega(i)=\Sc ,
\end{cases}
\end{equation} where we recall $u=u(\alpha,q)$ and $v=v(\beta,q)$ from \eqref{def:uAndv}, respectively \eqref{def:uAndvWASEP} when $q=q(N)$. We set
\begin{equation}\label{def:TotalWeight}
W(\omega) := \prod_{i=1}^{N} W_i(\omega)
\end{equation} as the \textbf{total weight} of $\omega \in \Psi_N$. The following result is Theorem 13 in \cite{BCEPR:CombinatoricsPASEP}.
\begin{lemma}[Brak et al.~\cite{BCEPR:CombinatoricsPASEP}]\label{lem:Paths} 
The function $B$ given by $B(\eta)=\sum_{\omega \in C_{\eta}} W(\omega)$ with 
\begin{equation}\label{def:SetOfPathsForConfiguration}
\mathcal{C}_{\eta} = \left\{ \omega \in \mathcal{A}_N \ \colon \ \omega(i) \in \big\{ \Nc, \Eb \big\} \text{ if and only if }  \eta(i)=1   \right\}
\end{equation}
is a basic weight function. In particular, the stationary distribution  satisfies for all $\omega \in \Omega_N$
\begin{equation}
\mu_N(\eta) = \left( \sum_{\omega \in \mathcal{C}_{\eta}} W(\omega) \right) \left( \sum_{\omega \in \Psi_N} W(\omega)  \right)^{-1} .
\end{equation}
\end{lemma}
\begin{figure}
    \centering
\begin{tikzpicture}[scale=1]
\draw [->,line width=1pt] (0,0) to (12.5,0); 	
\draw [->,line width=1pt] (0,0) to (0,3.3); 	

	\node[shape=circle,scale=1,draw] (A1) at (0.5,-1){};
	\node[shape=circle,scale=1,draw] (A2) at (1.5,-1){};
	\node[shape=circle,scale=1,draw] (A3) at (2.5,-1){};
	\node[shape=circle,scale=1,draw] (A4) at (3.5,-1){};
	\node[shape=circle,scale=1,draw] (A5) at (4.5,-1){};
	\node[shape=circle,scale=1,draw] (A6) at (5.5,-1){};
	\node[shape=circle,scale=1,draw] (A7) at (6.5,-1){};
	\node[shape=circle,scale=1,draw] (A8) at (7.5,-1){};
	\node[shape=circle,scale=1,draw] (A9) at (8.5,-1){};
	\node[shape=circle,scale=1,draw] (A10) at (9.5,-1){};
	\node[shape=circle,scale=1,draw] (A11) at (10.5,-1){};
	\node[shape=circle,scale=1,draw] (A12) at (11.5,-1){};

 \draw [->,line width=1pt] (-0.4,-1+0.4) to [bend right,in=135,out=45] (A1);
 \draw [->,line width=1pt] (A12) to [bend right,in=135,out=45] (12.4,-1+0.4);	
	
	\node[shape=circle,scale=0.8,fill=red] (CB) at (0.5,-1) {}; 
	\node[shape=circle,scale=0.8,fill=red] (CB) at (2.5,-1) {}; 
	\node[shape=circle,scale=0.8,fill=red] (CB) at (3.5,-1) {}; 
	\node[shape=circle,scale=0.8,fill=red] (CB) at (5.5,-1) {}; 
	\node[shape=circle,scale=0.8,fill=red] (CB) at (7.5,-1) {}; 
	\node[shape=circle,scale=0.8,fill=red] (CB) at (10.5,-1) {};

	\draw [line width=1pt] (A1) to (A2); 	
	\draw [line width=1pt] (A2) to (A3); 	
	\draw [line width=1pt] (A3) to (A4); 	
	\draw [line width=1pt] (A4) to (A5); 	
	\draw [line width=1pt] (A5) to (A6); 	
	\draw [line width=1pt] (A6) to (A7); 	
	\draw [line width=1pt] (A7) to (A8); 	
	\draw [line width=1pt] (A8) to (A9); 	
	\draw [line width=1pt] (A9) to (A10); 	
	\draw [line width=1pt] (A10) to (A11); 	
	\draw [line width=1pt] (A11) to (A12); 	
 

	\draw [->,line width=1pt] (0,0) to (1,1); 	 
 	\draw [->,line width=1pt,dotted] (1,1) to (2,1); 	 
 	\draw [->,line width=1pt] (2,1) to (3,2); 	 
 	\draw [->,line width=1pt] (3,2) to (4,3); 	
 	\draw [->,line width=1pt] (4,3) to (5,2); 	 
 	\draw [->,line width=1pt] (5,2) to (6,2); 	 
 	\draw [->,line width=1pt] (6,2) to (7,1); 	 
 	\draw [->,line width=1pt] (7,1) to (8,1); 	
 	\draw [->,line width=1pt,dotted] (8,1) to (9,1); 	 
 	\draw [->,line width=1pt] (9,1) to (10,0); 	
 	\draw [->,line width=1pt] (10,0) to (11,1); 	 
 	\draw [->,line width=1pt] (11,1) to (12,0); 		
 

\end{tikzpicture}
 \caption{Correspondence between bi-colored Motzkin paths and the particle configuration. The horizontal moves $\Ew$ are marked as dotted lines. }
    \label{fig:MotzkinPaths}
\end{figure}
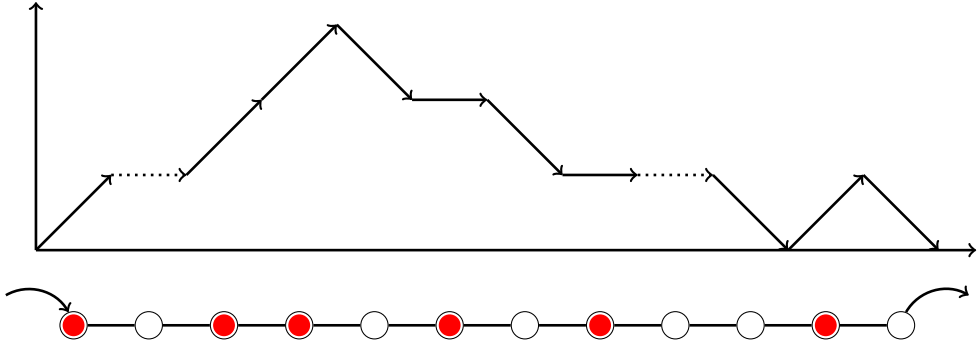 We refer to Figure \ref{fig:MotzkinPaths} for a visualization of the correspondence between bi-colored Motzkin paths and particle configurations. 
Let us remark that combinatorial representations of the stationary distribution of the ASEP with open boundaries are also available when we  allow for particles to enter at the right and exit on the left. In this case,  
Corteel and Williams found a remarkable expression of the stationary distribution using  staircase tableaux \cite{CW:TableauxCombinatorics}. We focus in the  following on the open ASEP model with respect to parameters $q,\alpha,\beta$, and leave the general case of five parameters, with entering rates $\alpha, \delta>0$ and exiting rates $\beta,\gamma>0$ of particles at both ends of the segment, for future work.

\subsection{Two special cases}\label{sec:TwoSpecialCases}

We discuss now two special cases where the above representation of the stationary distribution in terms of bi-colored Motzkin paths simplifies. 

\subsubsection{The TASEP with open boundaries} \label{sec:TASEPwithopenCombinatorics}

Suppose that $q=0$, i.e.\ we consider the TASEP with open boundaries; see also Section 2.2 in \cite{DEL:DensityExcursion}. Then the weights satisfy for all $i\in  \lsem N \rsem $ 
\begin{equation*}
W_i(\omega) = 1 \text{ whenever } h_{i}(\omega)>0. 
\end{equation*}
In particular, note that the weight of a bi-colored Motzkin path $\omega \in \Psi_N$ only depends on the number of times its associated lattice path stays, respectively returns, to the $x$-axis. This further simplifies when $\alpha=\beta=1$ as all configurations $\omega \in \Psi_N$ receive the same weight, and so the partition function is given by the $(N+1)^{\text{th}}$ Catalan number, i.e.
\begin{equation}\label{eq:CatanCount}
\sum_{\omega \in \Psi_{N}} W_{N} = |\Psi_{N}| = \frac{1}{N+2} \binom{2N+2}{N+1} ;
\end{equation}
see also \cite{CW:CombinatoricsTASEP,W:qEulerian} for a more detailed discussion of combinatorial expressions of the partition function with other parameters, and \cite{WBE:CombinatorialMappings} for a survey on combinatorial representations. Let us mention that we will revisit the TASEP with open boundaries in Section \ref{sec:LPP}, studying its representation as a last passage percolation model on a strip. 

\subsubsection{A finite Matrix product ansatz representation} \label{sec:MPAFiniteRep}

Observe that in the case where \begin{equation}\label{eq:FiniteRep}
uv q^{k}=1
\end{equation} for some finite $k \in \{0,1,\dots\}$, we have that $W(\omega)=0$ whenever $h_{i}(\omega) > k$ for some $i\in  \lsem N \rsem $, i.e.\ we restrict to bi-colored Motzkin paths whose associated lattice paths  have height at most $k$. In this case, it turns out the matrix product ansatz allows for a representation using only finite dimensional matrices, and an explicit characterization of the invariant measures of the ASEP with open boundaries can be given in terms of Bernoulli-shock measures \cite{JM:Finite,S:Schutz2,S:ASEPReverse}. We will elaborate on this idea in more detail in Section~\ref{sec:ShockMeasuresSchutz}, where we approximate the stationary distribution using coupling arguments and shock measures. 

\section{Polymer characterization of the stationary distribution}\label{sec:CharacterizationASEP}

In this section, we further investigate the representations of the stationary distribution of the open ASEP discussed in Section \ref{sec:CombinatorialRep}. Our key observation is that in the fan region, weighted bi-colored Motzkin paths can be studied using ideas from random polymer models. Speaking in the language of polymers, we are interested in the localization and delocalization, i.e.\ whether a random Motzkin path sampled according to its total weight stays typically away from the $x$-axis. 

\subsection{A brief introduction to random polymer models}\label{sec:IntroductionPolymers}

In the following, we recall some basic definitions and properties of random polymer models. As the related literature is way too exhaustive to give a full account at this point, and as we require a slightly different setup compared to the standard models, we only give a brief overview, and instead refer the interested reader to the notes by Giacomin \cite{G:RandomPolymer} for a more exhaustive introduction. \\

Let $(S_n)_{n \in \N_0}$ be a lazy simple random walk with i.i.d.\ increments, that is $S_0=0$ and $S_n=\sum_{i=1}^{n}X_i$ for all $n\in \N$, where $(X_i)_{i\in \N}$ are i.i.d.\  and satisfy
\begin{equation}
\P(X_1=1) = \P(X_1=-1) = \frac{1}{2} \P(X_1=0) = \frac{1}{4} \, .
\end{equation} 
Let $\Pb$ denote the corresponding law on the space of trajectories of the lazy simple random walk $(n,S_n)_{n\in \N_0}$. Let $V \colon \Z \times \{-1,0,1\} \rightarrow \R \cup \{-\infty\}$ be a function taking values in the reals together with $-\infty$. For all $N\in \N$, we define the \textbf{free Polymer measure} $\Pb^{\mathsf{f}}_{N,V}$ by
\begin{equation}\label{def:FreePolymerMeasure}
\frac{\dif \Pb^{\mathsf{f}}_{N,V}}{\dif \Pb} = \frac{4^N}{\mathcal{Z}^{\mathsf{f}}_{N,V}} \exp\left( \sum_{i=1}^{N} V(S_i,X_i) \right) , 
\end{equation} where $\mathcal{Z}^{\mathsf{f}}_{N,V}$ is a suitable normalization constant, to which, with a slight abuse of notation, we refer to as the  \textbf{partition function} of the polymer. Here, we use the convention that the right-hand side in \eqref{def:FreePolymerMeasure} is zero whenever $V(S_i,X_i)=-\infty$ for some $i\in \lsem N \rsem$. Similarly, we define the \textbf{constraint polymer measure} $\Pb^{\mathsf{c}}_{N,V}$ by
\begin{equation}\label{def:ConstraintPolymerMeasure}
\frac{\dif \Pb^{\mathsf{c}}_{N,V}}{\dif \Pb} = \frac{4^N}{\mathcal{Z}^{\mathsf{c}}_{N,V}} \exp\left( \sum_{i=1}^{N} V(S_i,X_i) \right) \mathds{1}_{\{S_N=0\}}
\end{equation} with respect to $N\in \N$ and $V$, where $\mathcal{Z}^{\mathsf{c}}_{N,V}$ is a suitable normalization constant. In other words, we apply the pinning constraint that the random walk returns to the origin after $N$ steps. We will write $\Pb^{\mathsf{f}}$ and $\Pb^{\mathsf{c}}$ whenever $N$ and $V$ are clear from the context. If $V$ satisfies $V(x,\cdot)=-\infty$ for all $x < 0$, we say that we have a \textbf{hard wall constraint}, i.e.\ we restrict the available state space such that the walk according to the respective polymer measure is require to stay almost surely non-negative. Whenever the limit
\begin{equation}\label{def:FreeEnergyPartition}
F^{\cdot}(V) = \lim_{N \rightarrow \infty} \frac{1}{N} \log\left( \mathcal{Z}^{\cdot}_{N,V} \right) 
\end{equation} exists for $\mathcal{Z}^{\cdot}_{N,V}  \in \{ \mathcal{Z}^{\mathsf{f}}_{N,V}, \mathcal{Z}^{\mathsf{c}}_{N,V}\}$, we refer to $F^{\mathsf{f}}(V)$, respectively $F^{\mathsf{c}}(V)$, as the \textbf{free energy}. In the following, we are interested whether $F^{\cdot}(V)>\log(4)$, called the \textbf{localization regime}, or $F^{\cdot}(V)=\log(4)$,  called the \textbf{delocalization regime}, for very specific choices of $V$.

\subsection{Characterizing the equilibrium of the open ASEP}\label{sec:ASEPasPolymer}

We argue in the following that in the fan region of the open ASEP, we can identify the weighted set of bi-colored Motzkin paths with a suitable constraint random polymer model with a hard wall. More precisely, fix $u=u(\alpha,q)$ and $v=v(\beta,q)$ from \eqref{def:uAndv} such that $uv < 1$. We consider the function $V \colon \Z \times \{-1,0,1\} \rightarrow \R$, where 
\begin{equation}\label{eq:PolymerFunctionMotzkin}
V(h,y) = \begin{cases}
\log\left( \sqrt{(1-q^{h+1})(1-uvq^{h})} \right) & \text{ if } h \geq 0 \text{ and } y\in \{-1,1\} \\
\log\left( 2 + (u+v)q^{h} \right)  & \text{ if } h \geq 0 \text{ and } y=0 \\
 -\infty  & \text{ otherwise.}
\end{cases}
\end{equation}
Recall the total weight function $W$ from \eqref{def:TotalWeight}. Let $\P_N$ denote the law on the space of bi-colored Motzkin paths $\Psi_N$, where $\omega \in \Psi_N$ is chosen proportional to $W(\omega)$. Moreover, recall that for all $\omega \in \Psi_N$, we denote by $v_\omega$  its corresponding lattice path in $\textup{MP}_N$.

\begin{lemma}\label{lem:WeightToPolymer} Let $uv< 1$ and $N\in \N$, and consider the constraint polymer measure with respect to $V$ from \eqref{eq:PolymerFunctionMotzkin}. Then for all $\zeta \in \textup{MP}_N$, 
\begin{equation}\label{eq:EqualityOfMeasures}
\Pb_{N,V}^{\mathsf{c}}(\zeta)=\P_N(v_{\omega}=\zeta) . 
\end{equation} Moreover, we have that the partition function $\mathcal{Z}^{\mathsf{c}}_{N,V}$ satisfies
\begin{equation}\label{eq:PartitionMotzkin}
\mathcal{Z}^{\mathsf{c}}_{N,V}  = \frac{1}{(1-q)^N}\sum_{\omega \in \Psi_N} W(\omega) .
\end{equation}
\end{lemma}
\begin{proof} Observe that a path chosen according to $\Pb_{N,V}^{\mathsf{c}}$ is almost surely in $\textup{MP}_N$. Hence, it suffices to show that the weight of a path $\zeta \in \textup{MP}_N$, by summing the weights of its corresponding bi-colored Motzkin paths, is proportional to the weight of $\zeta$ according to $\Pb_{N,V}^{\mathsf{c}}$. Note that in every bi-colored Motzkin path, the number of $\Nc$ moves from height $h$ to $h+1$ equals the number of $\Sc$ moves from height $h+1$ to $h$. Moreover, 
\begin{equation*}
\sum_{\omega \in \Psi_N \, \colon \, v_{\omega}=\zeta} W(\omega) = \left( \prod_{i \,\colon \,\omega(i) \in \{ \Nc,\Sc\}} W_i(\omega) \right) \left( \prod_{i \, \colon \,\omega(i) \notin \{ \Nc,\Sc\}} \big(2+(u+v)q^{h_i(\omega)}\big) \right) . 
\end{equation*} 
With these two observations, the statement \eqref{eq:EqualityOfMeasures} follows from the choice of $V$. Equation \eqref{eq:PartitionMotzkin} follows from \eqref{eq:EqualityOfMeasures} and definition of $\mathcal{Z}^{\mathsf{c}}_{N,V}$ in \eqref{def:ConstraintPolymerMeasure}, noting the extra factor of $(1-q)^{-N}$ by our choice of $V$.
\end{proof}

Let us mention at this point that different representations for the basic weights are known, for example by Enaud and Derrida in \cite{ED:MPAdifferent};  see also recent work by Barraquand and Le Doussal to construct a solution to the open KPZ equation \cite{BL:StationaryKPZ} using their representation. We stress that the representation in \cite{ED:MPAdifferent} is for the five parameter version of the open ASEP, where particles also exit at the left (enter at the right) boundary at rate $\gamma$ (at rate $\delta$),  
 and assuming \textbf{Liggett's condition}
 \begin{equation}\label{eq:LiggettsCondition}
 \alpha + \frac{\gamma}{q} = 1 \quad \text{ and } \quad  \beta + \frac{\delta}{q} = 1 .
 \end{equation} In total, this leaves again three degrees of freedom in the parametrization of the model. In  \cite{ED:MPAdifferent}, the basic weight function is again given by weighted random walks, and the weights take a similar form as in \eqref{def:WeightsIndividually}. However, the random walk trajectories do not need to return to the $x$-axis after $N$ steps, and to our best knowledge, do not allow for a simple Markovian construction of the underlying polymer measure; see also Lemma \ref{lem:SpatialMarkov}.  
Our description of the basic weight function using bi-colored Motzkin paths is specific to the case of the three parameters $(q,\alpha,\beta)$, but we conjecture that our results extend to the open ASEP under general boundary parameters.
Finally, let us mention that the above characterization of the stationary distribution extends directly to the open WASEP.

\subsection{Strategy for the fan region of open ASEP and open WASEP}

In order to approximate the stationary distribution in the fan region, we rely on its characterization as a constraint random polymer model. In the maximal current phase, we show that the respective polymer measure is delocalized. 
More precisely, we argue in Sections \ref{sec:DelocalizationHardWall} and \ref{sec:DelocalizationPinning} that a bi-colored Motzkin path chosen according to $\P_N$ and evaluated at distance $x$ from the boundary has with high probability a height of order $\sqrt{x}$, provided that $x$ is sufficiently large compared to $N$, and depending on the choice of $q=q(N)$.   This is achieved by a stochastic domination when $u+v\leq 0$, and a  supermartingale argument, otherwise. For the last part, when $u+v>0$, we first establish a  delocalization result under the free polymer measure in Section \ref{sec:DelocalizationHardWall} which is then transferred to a delocalization result for the constraint polymer in Section \ref{sec:DelocalizationPinning}. For the open ASEP and open WASEP, we present the proof of the approximation of the stationary distribution in Section \ref{sec:ComparisionMeasures}.
Heuristically, the probability to see each of the moves in $\mathcal{A}$ from \eqref{def:Acal} is roughly equally likely and independent of the previous moves, allowing us in Section \ref{sec:ComparisionMeasures} to conclude the desired approximation.  In the fan regime of the high and low density phase of the open ASEP, we show that the respective polymer measure is localized. We give in Section \ref{sec:ExpectedReturnTimes} bounds on the expected number of steps between two contact points of the associate Motzkin path with the x-axis by exploiting the renewal structure of random polymers. We then couple in Section \ref{sec:RegenerationLocalized} the constraint polymer to a stationary renewal process, which in return gives rise to a product measure in the respective particle configuration. 
%
%
%
%
%
%
%
%
%

\section{Approximation in the maximal current phase}\label{sec:DelocalizationMaxCurrent}

In this section, we investigate that the stationary distribution in the maximal current phase of the open ASEP and open WASEP. We show that the associated random polymer is delocalized, establishing Theorems~\ref{thm:MaxCurrent} and \ref{thm:MaxCurrentWASEP}. We start in Section \ref{sec:DelocalizationASEP} with the open ASEP in the maximal current phase, where we show that the path measure under the weights $V$ from \eqref{eq:PolymerFunctionMotzkin} is equivalent to the uniform measure on the space of bi-colored Motzkin paths.


\subsection{Delocalization for the open ASEP in the maximal current phase}\label{sec:DelocalizationASEP}

For a given path $\omega \in \Psi_N$, recall its height function $(h_x(\omega))_{x \in \lsem 0,N \rsem}$ from \eqref{def:HeightPath}. We let $(A^{i,j}_N)_{N\in \N}$ be a family of events on the space of lazy simple random walk paths $(S_n)_{n \in \lsem 0,N \rsem}$, where 
\begin{equation*}
A^{i,j}_N := \left\{ S_i \geq j \right\} .
\end{equation*} 
We have the following result on the open ASEP in the maximal current phase. 

\begin{lemma}\label{lem:EquivalenceASEPhtransform}

Let $q\in (0,1)$, and $u,v < 1$ be fixed. Recall the function $V$ from \eqref{eq:PolymerFunctionMotzkin}. Then 
\begin{equation}\label{eq:HighProbabilityASEP1}
\lim_{ N\rightarrow \infty} \Pb^{\textsf{c}}_{N,V}\left( A_N^{i,j} \right) =1
\end{equation} for all $j=j_N$ and $i=i_N \leq N/2$ such that $j_N \ll \sqrt{i_N}$. 
Moreover, for every interval $I=\lsem a,b\rsem$ with $|I| \ll \min(a,N-b) $, the height function $(h_x(\zeta))_{x \in \lsem 0,N \rsem}$ satisfies
\begin{equation}\label{eq:HighProbabilityASEP2}
\lim_{ N\rightarrow \infty} \Pb^{\textsf{c}}_{N,V}\left( h_{x}(\zeta)^2 \geq \sqrt{ |I| \min(a,N-b)} \text{ for all } x\in I \right)  = 1 . 
\end{equation}
\end{lemma}
\begin{proof} 
For $q\in (0,1)$, we recall equation (63) in \cite{BECE:ExactSolutionsPASEP}, which states that the partition function $Z_N$ from Lemma \ref{lem:Weights} of a basic weight function satisfies
\begin{equation}
Z_N = \frac{4(q;q)_{\infty}^3}{\sqrt{\pi}(u,v;q)_{\infty}}\frac{4^{N}}{N^{3/2}(1-q)^N} + o\left( \frac{4^{N}}{N^{3/2}(1-q)^N} \right) 
\end{equation} in the maximal current phase of the open ASEP. Here, 
we set $(x,y;q)_{\infty}:=(x;q)_{\infty}(y;q)_{\infty}$, and let
\begin{equation}\label{def:Pochhammer}
(z;q)_{\infty} := \prod_{i=0}^{\infty} (1-zq^{i})
\end{equation}
for $z\in \R$ be the $q$-Pochhammer symbol.
Together with Lemma \ref{lem:Paths} yields that 
\begin{equation}\label{eq:PartitionFunctionAsymptotics}
\bar{\mathcal{Z}}_N := \frac{1}{(1-q)^N}\sum_{\omega \in \Psi_N} W(\omega) = \frac{4(q;q)_{\infty}^3}{\sqrt{\pi}(u,v;q)_{\infty}}\frac{4^{N}}{N^{3/2}} + o\left( \frac{4^{N}}{N^{3/2}} \right) . 
\end{equation}
Recall the mapping $v_{\omega}$ for $\omega \in \Psi_N$ from Lemma~\ref{lem:WeightToPolymer}. Then by enlarging the underlying state space, and equation \eqref{eq:PartitionMotzkin} to express the partition function $\mathcal{Z}^{\mathsf{c}}_{N,V}$ by the weights $W$, we see that by a change of measure, 
\begin{equation*}
\lim_{ N\rightarrow \infty} \Pb^{\textsf{c}}_{N,V}(A^{i,j}_N) = \lim_{ N\rightarrow \infty} \frac{1}{\E_N[W(\omega)]} \E_N\Big[W(\omega) \mathds{1}_{\{ v_\omega \in A^{i,j}_N \}}\Big] = \lim_{ N\rightarrow \infty} \frac{|\Psi_N|}{\bar{\mathcal{Z}}_N}  \E_N\Big[W(\omega)  \mathds{1}_{\{ v_\omega \in A^{i,j}_N \} }\Big] .
\end{equation*} Here,  $\E_N[\, \cdot \,]$ denotes the expectation with respect to the uniform distribution on $\Psi_N$. 
Moreover, recall from \eqref{eq:CatanCount} that the number of bi-colored Motzkin paths $|\Psi_N|$ equals the $(N+1)^{\textup{th}}$ Catalan number. It is a well-known fact that
\begin{equation}\label{eq:CatalanAsymp}
c_1 \frac{4^{N}}{N^{3/2}} \leq |\Psi_N| \leq  c_2 \frac{4^{N}}{N^{3/2}}
\end{equation} 
for some $c_1,c_2>0$ and all $N$ sufficiently large.
Hence, using dominated convergence  together with \eqref{eq:PartitionFunctionAsymptotics} and \eqref{eq:CatalanAsymp}, 
it suffices to show \eqref{eq:HighProbabilityASEP1} and \eqref{eq:HighProbabilityASEP2} with respect to the uniform measure on $\Psi_N$.  Observe that we obtain the uniform measure on $\Psi_N$ by considering a lazy simple random walk conditioned on staying non-negative, to return to the origin after $N$ steps, and flipping independent fair coins to decide for the coloring of each horizontal move. The first claim \eqref{eq:HighProbabilityASEP1} now follows from the standard fact that the above lazy simple random path converges to a Brownian excursion.
The second claim  \eqref{eq:HighProbabilityASEP2} is a consequence of \eqref{eq:HighProbabilityASEP1} for heights at $a$ and $b$, and a standard moderate deviation estimate for the fluctuations of a lazy simple random walk on $I$. 
\end{proof}

\subsection{Delocalization for the open WASEP via stochastic domination}\label{sec:DelocalizationStochasticDomination}

Consider now the open WASEP in the maximal current phase.  Our goal is to compare the measure $\Pb^{\mathsf{c}}$ to the uniform distribution on $\Psi_N$.
We discuss in the following two different approaches to achieve this goal. 
In the first approach, presented in this section, assuming that $u$ and $v$ satisfy $uv < 1$ and $u+v\leq 0$, we establish stochastic domination of the trajectory with respect to a certain non-lazy simple random walk conditioned to stay non-negative and to return to the $x$-axis after $N$ steps. We have the following result on the trajectory.

\begin{proposition}\label{pro:DelocalizationHalfspaceDomination}
Recall $\Pb^{\mathsf{c}}=\Pb_{N,V}^{\mathsf{c}}$ from \eqref{def:ConstraintPolymerMeasure} with $V$ from \eqref{eq:PolymerFunctionMotzkin}. If $q$ satisfies \eqref{eq:qAssumption} for some $\varepsilon>0$, and $u+v \leq 0$, then 
\begin{equation}\label{eq:DelocalizationDominationPoint}
\lim_{N \rightarrow \infty} \Pb^{\mathsf{c}}\big(A^{i,j}_N \big) = 1 
\end{equation} holds for all $j=j_N$ and $i=i_N \leq N/2$ such that $j_N \ll \sqrt{i_N}$ and $\min(i_N,N-i_N)\gg N^{2\varepsilon}\log^2 (N)$. Moreover, for all  $I=\lsem a,b\rsem$ with $\max(|I|,N^{2\varepsilon}\log^2(N)) \ll \min(a,N-b)$, the height function $(h_x(\zeta))_{x \in \lsem 0,N \rsem}$ satisfies
\begin{equation}\label{eq:DelocalizationDominationSegment}
\lim_{ N\rightarrow \infty} \Pb^{\textsf{c}}\left( h_{x}(\zeta)^2 \geq \sqrt{|I| \min(a,N-b)} \text{ for all } x\in I \right)  = 1 . 
\end{equation}
\end{proposition}


For $M\in \N$ and $x\in \Z$, we define the space of lattice paths
\begin{equation}\label{def:LatticePathsNonNegative}
\begin{split}\Lambda^{(x)}_M := \Big\{ (v_0,v_1,\dots,v_M) \colon \  &v_{i-1} \in \lsem -x, \infty \rsem \times \lsem -x, \infty \rsem \text{ for all } i \in \lsem M+1 \rsem, \ v_0 = (0,0), \\
&v_i-v_{i-1} \in \{ (1,1),(1,0),(1,-1)\} \text{ for all } i\in  \lsem M \rsem \Big\} ,
\end{split}
\end{equation} 
and let $\hat{\Lambda}^{(x)}_M \subseteq \Lambda^{(x)}_M$ be the space of lattice paths without horizontal moves, i.e.\ $v_i-v_{i-1} \neq (1,0)$ for all $i$. 
In order to show Proposition \ref{pro:DelocalizationHalfspaceDomination}, we first recall some basic notions for stochastic domination. For $M\in \N$ fixed, consider the natural ordering $\succeq$ on the space of lazy simple random walk paths of length $M$, i.e.\ we say that $(\tilde{S}_{n})_{n \in  \lsem 0,M\rsem } \succeq (S_{n})_{n \in \lsem 0,M\rsem }$ if
\begin{equation*}
\tilde{S}_i \geq S_{i} \text{ for all } i \in \lsem 0,M \rsem . 
\end{equation*}
We will be interested in sets $B$ of trajectories which are \textbf{increasing} with respect to the partial order $\succeq$, i.e.\ we have that 
\begin{equation*}
Y \in B \text{ and } \tilde{Y} \succeq Y \ \Rightarrow  \ \tilde{Y} \in B . 
\end{equation*}
Note that when we restrict ourselves to trajectories in $\hat{\Lambda}^{(x)}_M $ for some $x$ and $M$,  the partial order $\succeq$ gives rise to a distributive lattice. In this case, for any ordered pair of lattice paths $\tilde{Y} \succeq Y$ of length $M$, we denote the respective unique minimal and maximal paths by
\begin{align*}
\begin{split}
\min(\tilde{Y} , Y) &:= ( \min(\tilde{Y}_x,Y_x) \text{ for } x\in \lsem 0,M \rsem  ) \\ 
\max(\tilde{Y} , Y) &:= ( \max(\tilde{Y}_x,Y_x) \text{ for } x\in  \lsem 0,M \rsem ) . 
\end{split}
\end{align*}
Note that the above observations remain valid when conditioning on the height of the endpoint of the paths.
Next, let $\mathcal{N}_{a,b}=\mathcal{N}_{a,b}(Y)$ denote the number of horizontal step in the trajectory $Y$ between positions $a$ and $b$, and let $\hat{Y}_{a,b}\in \hat{\Lambda}^{(0)}_{N-\mathcal{N}_{a,b}}$ be the path $Y$ after removing all horizontal moves between $a$ and $b$. Recall that for all $\omega \in \Psi_N$, we denote by $v_\omega$  its respective lattice path in $\textup{MP}_N$. For $M\leq N$, define the measure $ \hat{\mathbf{P}}_{\lsem a,b \rsem}^{M,N}$ on $\textup{MP}_{M}$ by  
\begin{equation*}
 \hat{\mathbf{P}}_{\lsem a,b \rsem}^{M,N}( \cdot ) \sim \sum_{ \omega \in \Psi_N } W(\omega) \mathds{1}_{\left\{  \hat{Y}_{a,b} = \cdot  \text{ for } Y=v_{\omega} \text{ and } \mathcal{N}_{a,b}(Y) = N-M \right\}} . 
\end{equation*}
In other words, $\hat{\mathbf{P}}_{\lsem a,b \rsem}^{M,N}(Y)$ is proportional to the total weight of all configurations in $\Psi_N$ which reduce to $Y$ after projecting to $\textup{MP}_{N}$ and then removing exactly $N-M$ horizontal steps in $\lsem a,b \rsem$. Let $\tilde{\mathbf{P}}_{a,b}^{M,N}$ be the measure on $\textup{MP}_{M}$ which we get from $\hat{\mathbf{P}}_{\lsem a,b \rsem}^{M,N}$ by replacing the path between $a$ and $b-N+M$ by a uniformly sampled non-lazy non-negative path.


\begin{lemma}\label{lem:StochasticDomHalfSpace} Let $u,v<1$ and recall $V$ from \eqref{eq:PolymerFunctionMotzkin}. Fix $M$ and let $a,b \in \lsem N \rsem$ with $a<b$. Then the measure $\hat{\mathbf{P}}_{\lsem a,b \rsem}^{M,N}$ stochastically dominates $\tilde{\mathbf{P}}^{M,N}_{a,b}$ on $\lsem a, b \rsem$, i.e.\ for every increasing set $B$ with respect to $\succeq$, measurable with respect to the path on $\lsem a, b -N+M\rsem$, and all $y,z$
\begin{equation}
\hat{\mathbf{P}}_{\lsem a, b \rsem }^{M,N}( B \, | \,  Y_{a} = y, Y_{b}=z  ) \geq  \tilde{\mathbf{P}}_{a,b}^{M,N}(B \, | \,   Y_{a} = y, Y_{b}=z ) . 
\end{equation}
\end{lemma}
\begin{proof} 
Let $N$ be sufficiently large such that $uvq^{-2}<1$. Then the function
\begin{equation*}
f(h) := (1-q^{h+1})(1-uvq^{h-1}) 
\end{equation*} is monotone increasing in $h$. From this, observe that for any two paths $Y$ and $\tilde{Y}$ with $\tilde{Y} \succeq Y$, which agree outside of the interval $ \lsem a, b-N+M \rsem$, 
\begin{equation*}
\hat{\mathbf{P}}_{\lsem a, b \rsem }^{M,N}( \tilde{Y} ) \geq  \hat{\mathbf{P}}_{\lsem a, b \rsem }^{M,N}( Y ) 
 \end{equation*} as $V(h,0)$ is increasing in $h$ by our assumption $u+v\leq 0$. Thus, since $\tilde{\mathbf{P}}_{\lsem a, b \rsem }^{M,N}$ assigns the same weight to each lattice path in $\lsem a,b-N+M \rsem$, we get for all above $Y$, $\tilde{Y}$ 
\begin{equation}\label{eq:HolleyCondition}
\hat{\mathbf{P}}_{\lsem a, b \rsem }^{M,N}( \max(\tilde{Y},Y) ) \tilde{\mathbf{P}}^{M,N}_{a,b}( \min(Y,\tilde{Y}) )  \geq  \hat{\mathbf{P}}_{\lsem a, b \rsem }^{M,N}( \tilde{Y} ) \tilde{\mathbf{P}}^{M,N}_{a,b}( Y )   . 
\end{equation} Since the underlying space of trajectories between $a$ and $b-N+M$ conditioned to agree in their heights at $a$ and $b-N+M$, is a distributive lattice, the stochastic domination follows by  \eqref{eq:HolleyCondition} and Holley's inequality on the interval $\lsem a,b-N+M \rsem$; see Corollary 11 in \cite{H:Holleys}.
\end{proof}

Let us stress that in the above result, it is crucial that we only allow for north and south moves and  $u+v \leq 0$ in order to apply Holley's inequality. It remains now to control the number of horizontal moves in a configuration according to~$\Pb^{\mathsf{c}}$. To do so, let $\Pb^{\mathsf{h}}_N$ denote measure on the space $\Lambda^{(0)}_N$ by weighting each configuration proportionally to $2^{\mathcal{N}_{0,N}}$. Intuitively, we obtain $\Pb^{\mathsf{h}}_N$ from a lazy simple random walk $(S_x)_{x \in \lsem 0,N \rsem}$ conditioned to stay non-negative until time $N$. This process is known to converge to an $h$-transformed lazy simple random walk; see also Section \ref{sec:RegenerationLocalized}. 

\begin{lemma}\label{lem:IntermediateIntervals}
Consider an interval $\lsem a,b\rsem$ for some $a=a_N$ and $b=b_N$ with $m_N:=b_N-a_N \rightarrow \infty$ for $N \rightarrow \infty$. Then there exists some $\delta>0$ such that 
\begin{equation}
\lim_{N \rightarrow \infty}\Pb^{\mathsf{c}}\big( \mathcal{N}_{a,b} \leq (1-\delta) m_N \big) = 1 . 
\end{equation}
\end{lemma}
\begin{proof}
First, we argue that for every $\delta_1>0$, there exists $\delta_2>0$,  such that for all $x,y\in \N$ 
\begin{equation}\label{eq:GeometricGenerating1}
\Pb^{\mathsf{h}}_N\big( \mathcal{N}_{a,b} > (1-\delta_2) m_N \, \big| \, S_a=x, S_b=y \big) \leq  (2-\delta_1)^{-m_N} . 
\end{equation} 
To see this, notice that under the law $\Pb^{\mathsf{h}}_N(\, \cdot \, | \, S_a=x, S_b=y)$, the random variable $\mathcal{N}_{a,b}$ is stochastically dominated by the number $\tilde{X}$ of horizontal steps of a lazy simple random walk conditioned to return to the origin after $m_N-|y-x|$ steps. For any choice of $x$ and $y$, a local central limit theorem yields that there exists $C>0$ such that for all $N$ large
\begin{equation*}
\P( \tilde{X} \geq (1-\delta_2) m_N  ) \leq C m_N^{3/2} \P( X^{\prime} \geq (1-\delta_2) m_N  ) 
\end{equation*}
for a Binomial-$(m_N-|y-x|,\frac{1}{2})$-distributed random variable $X^{\prime}$. This implies \eqref{eq:GeometricGenerating1} by a standard tail estimate for $X^{\prime}$. Let  $\delta_1<(1-\max(u,v))/3$, and  note that from the definition of $\Pb^{\textsf{c}}$ and \eqref{eq:GeometricGenerating1}, together with a change of measure by \eqref{def:ConstraintPolymerMeasure}, we get that
\begin{equation}\label{eq:GeometricGenerating2}
\Pb^{\textsf{c}}\big( \mathcal{N}_{a,b} > (1-\delta_2) m_N \big) \leq (1+\max(|u|,|v|))^{m_N}(2-\delta_1)^{-m_N}  \leq \left(\frac{2-2\delta_1}{2-\delta_1}\right)^{m_N} ,
\end{equation} which gives the desired result. 
\end{proof}

\begin{proof}[Proof of Proposition \ref{pro:DelocalizationHalfspaceDomination}]
We will only show \eqref{eq:DelocalizationDominationPoint} as \eqref{eq:DelocalizationDominationSegment} directly follows \eqref{eq:DelocalizationDominationPoint} and a standard moderate deviation estimate for lazy simple random walks. We will use in the following two basic observations for the non-negative non-lazy simple random walk $(\hat{S}_n)_{n \in \lsem 0,M \rsem}$ with law $\hat{\Pb}^{\mathsf{h}}_M$. Let $M=M(N) \leq N$ and $x_N,y_N \gg N^{\varepsilon}\log(N)$. Then for every $f_N \gg 1 $ and $g_N \ll \min(x_N,y_N)$
\begin{align}\label{eq:SRWStatement1}
\lim_{N \rightarrow \infty}\hat{\Pb}^{\mathsf{h}}_M\big( \exists z \in \lsem M \rsem \text{ such that } \hat{S}_{z} \geq \sqrt{M} f_N^{-1} \, \big| \, \hat{S}_0=0, \hat{S}_M=0  \big) &= 1 \\
\lim_{N \rightarrow \infty}\hat{\Pb}^{\mathsf{h}}_M\big( \hat{S}_{z} \geq \sqrt{ g_N N^{\varepsilon}\log(N) }  \text{ for all } z \in \lsem M \rsem \, \big| \, \hat{S}_0=x_N, \hat{S}_M=y_N  \big) &= 1 . \label{eq:SRWStatement2}
\end{align}
Both statements follow from a basic computation using the reflection principle and a local limit theorem for non-negative simple random walks; see also  Chapter 2 in \cite{LL:RWIntroduction}. We claim that for any slowly growing function $f_N \gg 1$
\begin{equation*}
\lim_{N \rightarrow \infty}\Pb^{\textsf{c}}\left( \exists a_{N} \in \left\lsem \frac{1}{4}i_N, \frac{3}{4}i_N \right\rsem \text{ and } b_{N} \in \left\lsem \frac{5}{4}i_N, \frac{7}{4}i_N \right\rsem \, \colon \, \min(h_{a_N},h_{b_N}) \geq f^{-1}_N N^{\varepsilon}\log(N) \right) = 1 .
\end{equation*}
To see this, choose a lattice path according to $\Pb^{\textsf{c}}$ and consider  its  height at positions $(\frac{1}{8}(2\ell-1)i_N)_{\ell \in \lsem 4 \rsem}$. By Lemma~\ref{lem:IntermediateIntervals}, there exists some $\delta>0$ such that
\begin{equation}\label{eq:NotManyHorizontal}
\lim_{N \rightarrow \infty}\Pb^{\textsf{c}}\left( \max\Big( \mathcal{N}_{ \frac{1}{4}i_N, \frac{3}{4}i_N}, \mathcal{N}_{ \frac{5}{4}i_N, \frac{7}{4}i_N } \Big) \leq \frac{1}{2}(1-\delta)i_N \right) = 1 . 
\end{equation}
Now conditioning in addition on the event in \eqref{eq:NotManyHorizontal}, Lemma~\ref{lem:StochasticDomHalfSpace} ensures that the law of the process after removing all horizontal moves is stochastically dominated between $0$ and $i_N/2$ and between $3i_N/2$ and $2i_N$ by a non-lazy simple random walk with at least $\frac{\delta}{2}i_N$ steps. The desired result \eqref{eq:DelocalizationDominationPoint} on the height at $i_N$ now follows from equations \eqref{eq:SRWStatement1} and \eqref{eq:SRWStatement2}. 
%
\end{proof}

\begin{figure}
    \centering
\begin{tikzpicture}[scale=0.95]
\draw [->,line width=1pt] (0,0) to (15,0); 	
\draw [->,line width=1pt] (0,0) to (0,3.6); 	

\draw [line width=1pt,blue] (0,0) -- ++(0.2,0.2) -- ++(0.2,0)-- ++(0.2,0) -- ++(0.2,0) -- ++(0.2,0) -- ++(0.2,0) -- ++(0.2,0.2) -- ++(0.2,0) -- ++(0.2,0) -- ++(0.2,0) -- ++(0.2,0.2)  -- ++(0.2,0) -- ++(0.2,0)  -- ++(0.2,0) -- ++(0.2,0.2) -- ++(0.2,0) -- ++(0.2,0) -- ++(0.2,0.2)  -- ++(0.2,0) -- ++(0.2,0.2)  -- ++(0.2,0) -- ++(0.2,0)  -- ++(0.2,0) -- ++(0.2,0.2)  -- ++(0.2,-0.2) -- ++(0.2,0)  -- ++(0.2,0.2) -- ++(0.2,-0.2) -- ++(0.2,0.2)  -- ++(0.2,0.2) -- ++(0.2,0)  -- ++(0.2,0.2) -- ++(0.2,-0.2)
-- ++(0.2,-0.2)  -- ++(0.2,0) -- ++(0.2,0.2)  -- ++(0.2,0.2) -- ++(0.2,-0.2)  -- ++(0.2,0.2) -- ++(0.2,-0)  -- ++(0.2,-0.2) -- ++(0.2,-0.2)  -- ++(0.2,0.2) -- ++(0.2,0.2)  -- ++(0.2,-0.2) -- ++(0.2,-0.2)  -- ++(0.2,0.2) -- ++(0.2,-0.2) -- ++(0.2,-0.2) -- ++(0.2,0)  -- ++(0.2,0) -- ++(0.2,-0.2) -- ++(0.2,0)  -- ++(0.2,0) -- ++(0.2,-0.2)  -- ++(0.2,0) -- ++(0.2,0)  -- ++(0.2,-0.2) -- ++(0.2,0)  -- ++(0.2,0) -- ++(0.2,0)  -- ++(0.2,-0.2) -- ++(0.2,0)-- ++(0.2,0)  -- ++(0.2,0)  -- ++(0.2,0) -- ++(0.2,-0.2) --++(0.2,0) --++(0.2,0) -- ++(0.2,0)--++(0.2,0) --++(0.2,-0.2)  ;

\draw [line width=1pt,red] (0,0) -- ++(0.2,0.2) -- ++(0.2,0.2) -- ++(0.2,0.2) -- ++(0.2,0.2) -- ++(0.2,0) -- ++(0.2,0) -- ++(0.2,0)  -- ++(0.2,0.2) -- ++(0.2,0)  -- ++(0.2,0.2) -- ++(0.2,0.2)  -- ++(0.2,0) -- ++(0.2,0.2) -- ++(0.2,-0.2) -- ++(0.2,0.2) -- ++(0.2,0) -- ++(0.2,0.2)  -- ++(0.2,0) -- ++(0.2,0.2)  -- ++(0.2,0) -- ++(0.2,0.2)  -- ++(0.2,0.2) -- ++(0.2,0.2)  -- ++(0.2,0.2) -- ++(0.2,0.2)  -- ++(0.2,0) -- ++(0.2,0.2) -- ++(0.2,-0.2)  -- ++(0.2,0) -- ++(0.2,0)  -- ++(0.2,-0.2) -- ++(0.2,0.2)
-- ++(0.2,0.2)  -- ++(0.2,-0.2) -- ++(0.2,-0.2)  -- ++(0.2,-0.2) -- ++(0.2,-0.2)  -- ++(0.2,0.2) -- ++(0.2,0)  -- ++(0.2,0.2) -- ++(0.2,-0.2)  -- ++(0.2,0.2) -- ++(0.2,-0.2)  -- ++(0.2,-0.2) -- ++(0.2,0)  -- ++(0.2,0.2) -- ++(0.2,-0.2) -- ++(0.2,0) -- ++(0.2,0)  -- ++(0.2,-0.2) -- ++(0.2,-0.2) -- ++(0.2,0)  -- ++(0.2,0) -- ++(0.2,-0.2)  -- ++(0.2,-0.2)  -- ++(0.2,0.2) -- ++(0.2,0)  -- ++(0.2,0) -- ++(0.2,0)  -- ++(0.2,-0.2) -- ++(0.2,-0.2)  -- ++(0.2,0.2)  -- ++(0.2,-0.2) -- ++(0.2,0) --++(0.2,-0.2) --++(0.2,0) --++(0.2,-0.2) --++(0.2,-0.2) --++(0.2,-0.2) -- ++(0.2,-0.2) --++(0.2,-0.2) --++(0.2,-0.2) ;	
	 
 \draw [line width=1pt] (0,0) -- ++(0.2,0.2) -- ++(0.2,0.2) -- ++(0.2,-0.2) -- ++(0.2,0.2) -- ++(0.2,0.2) -- ++(0.2,0.2) -- ++(0.2,-0.2)  -- ++(0.2,0.2) -- ++(0.2,0)  -- ++(0.2,0.2) -- ++(0.2,0)  -- ++(0.2,0.2) -- ++(0.2,0.2) -- ++(0.2,0) -- ++(0.2,0.2) -- ++(0.2,-0.2) -- ++(0.2,0.2)  -- ++(0.2,0.2) -- ++(0.2,0.2)  -- ++(0.2,-0.2) -- ++(0.2,0)  -- ++(0.2,0.2) -- ++(0.2,0.2)  -- ++(0.2,-0.2) -- ++(0.2,0)  -- ++(0.2,0.2) -- ++(0.2,-0.2) -- ++(0.2,0.2)  -- ++(0.2,0.2) -- ++(0.2,0)  -- ++(0.2,0.2) -- ++(0.2,-0.2)
-- ++(0.2,-0.2)  -- ++(0.2,0) -- ++(0.2,0.2)  -- ++(0.2,0.2) -- ++(0.2,-0.2)  -- ++(0.2,0.2) -- ++(0.2,-0)  -- ++(0.2,-0.2) -- ++(0.2,-0.2)  -- ++(0.2,0.2) -- ++(0.2,0.2)  -- ++(0.2,-0.2) -- ++(0.2,-0.2)  -- ++(0.2,0.2) -- ++(0.2,-0.2) -- ++(0.2,0) -- ++(0.2,-0.2)  -- ++(0.2,0) -- ++(0.2,-0.2) -- ++(0.2,-0.2)  -- ++(0.2,0.2) -- ++(0.2,-0.2)  -- ++(0.2,-0.2)  -- ++(0.2,-0.2) -- ++(0.2,0)  -- ++(0.2,0) -- ++(0.2,0)  -- ++(0.2,-0.2) -- ++(0.2,0)  -- ++(0.2,0.2)  -- ++(0.2,-0.2) -- ++(0.2,0.2) --++(0.2,0) --++(0.2,-0.2) --++(0.2,-0.2) --++(0.2,-0.2) --++(0.2,-0.2) -- ++(0.2,0) --++(0.2,-0.2) --++(0.2,-0.2) ;	

 \node (H1) at (0.5,1.7) {$N^{\varepsilon}$}; 	
 \node (H2) at (4.3,-0.5) {$N^{3\varepsilon}\log(N)$}; 
 \node (H3) at (14.4,-0.5) {$N$};
  
  \draw [line width=1pt] (4.3,0) to  (4.3,-0.2); 
  \draw [line width=1pt] (14.4,0) to  (14.4,-0.2);  
  \draw [line width=1pt] (0,1.4) to  (0.2,1.4);  

\end{tikzpicture}
    \caption{Visualization of the different delocalization strategies in Propositions \ref{pro:DelocalizationHalfspaceDomination} and \ref{pro:DelocalizationHalfspaceMartingale}. The black curve corresponds to the trajectory sampled according to $\tilde{\mathbf{P}}_{\lsem 0,N \rsem}^{N}$. The red path is a sample for $\mathbf{P}^{\textup{c}}$ when $u+v \leq 0$, and dominates the black curve by Proposition \ref{pro:DelocalizationHalfspaceDomination} after removing all horizontal steps. The blue path illustrates a sample of $\mathbf{P}^{\textup{c}}$ when $u+v > 0$  and $\varepsilon>0$. For the colored paths, heuristically, the number of south moves is negligible before reaching a height of order $N^{\varepsilon}$. }
    \label{fig:DelocalizationApproaches}
\end{figure}
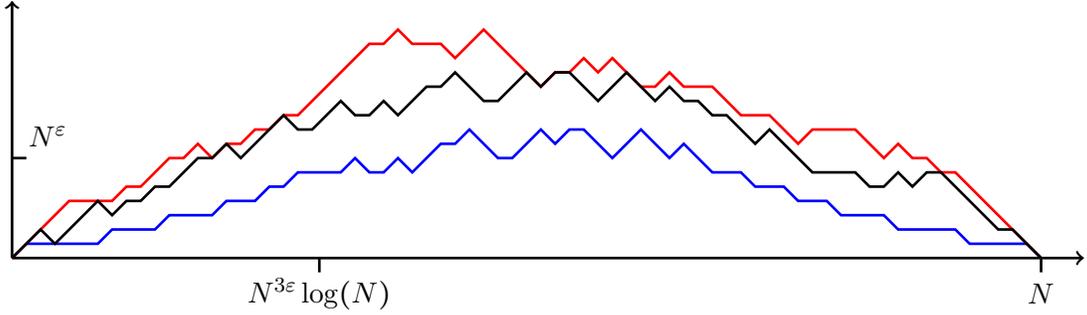

\subsection{Delocalization of a free random polymer with a hard wall}\label{sec:DelocalizationHardWall}

For the second approach, under stronger assumptions on the location and size of the target segment, we use a suitable supermartingale in order to compare the occurrence of events $A^{i,j}_N$ under the random polymer measure and the law $\Pb_N^{\mathsf{h}}$ of a lazy simple random walk conditioned to stay non-negative until time $N$; see Figure \ref{fig:DelocalizationApproaches} for the different strategies for delocalization. 
We start with a result on the delocalization for a free random polymer measure, which is then transferred to the constraint random polymer measure in Section \ref{sec:DelocalizationPinning}. \\

Recall the measure $\Pb^{\mathsf{f}}=\Pb_{N,V}^{\mathsf{f}}$, and assume in the following that $q$ satisfies \eqref{eq:qAssumption} for some $\varepsilon \in (0,\frac{1}{3})$, as well as that $u,v$ satisfy $\max(u,v) < 1$. The next proposition states a delocalization result on the corresponding free random polymer measure. 

\begin{proposition}\label{pro:DelocalizationHalfspaceMartingale}
 For all constants $C>0$, and for any  $(k_N)_{N \in \N}$ with $k_N \gg  N^{3\varepsilon}\log(N)$
\begin{equation}
 \Pb^{\mathsf{f}}(A^{i,j}_N \text{ holds for some } i \leq k_N \text{ and } j \geq C N^{\varepsilon}\log(N)) \geq  1 - \exp(-N^{\varepsilon})
\end{equation} 
for all $N$ sufficiently large.
\end{proposition}

In order to show Proposition \ref{pro:DelocalizationHalfspaceMartingale}, we require some setup.  Let
\begin{equation}\label{def:LatticePaths}
\begin{split}\Lambda_N := \Big\{ (v_0,v_1,\dots,v_N) \in (\N_0 \times \Z)^{N} \colon \,  v_i-v_{i-1} \in \{ (1,1),(1,0),(1,-1)\} \, \forall \, i\in  \lsem N \rsem \Big\}
\end{split}
\end{equation} 
be the space of all lattice paths of length $N$. For fixed $x\in \N_0$, let $(S^{\mathsf{h}}_n)_{n \in \lsem 0,N\rsem}$ be a sample on the space $\Lambda_N$ according to a lazy simple random walk started from $x$ and conditioned to be non-negative until time $N$. Note that $(S^{\mathsf{h}}_n)_{n \in \lsem 0,N\rsem}$ can be interpreted as a time-inhomogeneous Markov chain with $S^{\mathsf{h}}_0=x$ and
\begin{equation*}
\P( S^{\mathsf{h}}_n = y \, | \, S^{\mathsf{h}}_{n-1} = x ) = 
\begin{cases}
\frac{1}{4} P_{x+1}(\tau_0 > N-n) P_{x}(\tau_0 > N-n+1)^{-1}  & \text{ if } y=x+1 \\
\frac{1}{2} P_x(\tau_0 > N-n) P_{x}(\tau_0 > N-n+1)^{-1}  & \text{ if } y=x \\
\frac{1}{4} P_{x-1}(\tau_0 > N-n) P_{x}(\tau_0 > N-n+1)^{-1} & \text{ if } y=x-1 \text{ and } y\geq 0 \\
0 & \text{ otherwise,}
\end{cases}
\end{equation*}
where $P_x(\tau_0 > n-1)$ is the probability of a lazy simple random walk started from $x$ to not return to $0$ within the first $n-1$ steps; see also Chapter 12 in \cite{LL:RWIntroduction}. Let $(S_i)_{i \in \N_0}$ denote a lazy simple random walk. Using the reflection principle, we see that for all $x \in \N$
\begin{equation*}
P_{x}(\tau_0 > j) = P_{0}\Big( \max_{i \in \lsem j \rsem}S_i < x  \Big) =  P_{0}\big( S_j  \in (-x,x) \big) . 
\end{equation*} This implies that for all $\delta>0$, there exists some $J=J(\delta)$ such that for all $j \geq J$ 
\begin{equation}\label{eq:ApproximateTransitions}
1 = \inf_{x > 0} \frac{P_{x}(\tau_0 > j) }{P_{x}(\tau_0 > j+1)} \leq \sup_{x > 0} \frac{P_{x}(\tau_0 > j) }{P_{x}(\tau_0 > j+1)} \leq 1 + \delta
\end{equation}
using a local central limit theorem for $(S_i)_{i \in \N_0}$; see Chapter 2 in \cite{LL:RWIntroduction}. 
Next, we define an auxiliary process $(X_n)_{n \in \N_0}$ to express the partition function $\mathcal{Z}_{N,V}^{\mathsf{f}}$; see \cite{B:Martingales} for a related martingale technique for random polymers. 
Recall  $V$ from \eqref{eq:PolymerFunctionMotzkin} and set for all $n\in \lsem 0,N \rsem$
\begin{equation}
X_n = \prod_{i=1}^{n} \exp(V(S_i^{\mathsf{h}},S^{\mathsf{h}}_{i}-S^{\mathsf{h}}_{i-1})) . 
\end{equation}  Let $(\mathcal{F}_{n})_{n \in \lsem 0,N\rsem}$ denote the natural filtration with respect to $(S^{\mathsf{h}}_i)_{i \in \lsem 0,N\rsem}$, and write $\mathbf{P}_x=\mathbf{P}^{(N)}_{x}$ and $\mathbf{E}_x=\mathbf{E}^{(N)}_{x}$ for the law and expectation under the lazy simple random walk started from $S^{\mathsf{h}}_0=x$ and conditioned to be non-negative until time $N$, respectively. 

\begin{lemma}\label{lem:SupermartingaleComparison1}
Let $u,v \in (-1,1)$, and  $q$ satisfy \eqref{eq:qAssumption} for some $\varepsilon>0$ and $c>0$. Then there exist some absolute constant $J_0=J_0(u,v,\varepsilon,c)$ such that for all $N$ sufficiently large, $(X_n)_{n \in \lsem 0,N-J_0\rsem}$ is a supermartingale with respect to $(\mathcal{F}_n)_{n \in \lsem 0,N\rsem}$. Moreover, we have that
 $\mathbf{E}_0[X_N] \leq 2^{J_0}$ and
\begin{equation}\label{eq:WeightIdentity}
\mathcal{Z}_{N,V}^{\mathsf{f}} = \mathbf{E}_0[X_N] \cdot \left| \left\{ \omega \in  \mathcal{A}^{N} \colon h_{y}(\omega) \geq 0 \text{ for all } y \in \lsem N \rsem \right\}\right| . 
\end{equation} 
\end{lemma}
\begin{proof} Observe that there exists some $\delta>0$ and $N_0 \in \N$ such that for all $N \geq N_0$
\begin{equation}\label{eq:DeltaExtra}
(1+\delta)(u+v) \leq (1-\delta)(1+uv)q^{2} .
\end{equation}
Recall the constant $J=J(\delta)$ from \eqref{eq:ApproximateTransitions} for this choice of $\delta$, and set $J_0=\max(J,N_0)$.
Using \eqref{eq:ApproximateTransitions} and the fact that 
$x \mapsto P_x(\tau_0>j)$ is increasing 
for the first step, and the AM-GM inequality for the second step, a computation shows that for all $n \in \lsem N-J_0 \rsem$
\begin{align}
\begin{split}\label{eq:SupermartingaleCheck}
\mathbf{E}_x[X_{n+1} | \mathcal{F}_{n} ] &\leq  \bigg( \frac{1+\delta}{4}\Big(2+ (u+v)q^{h_n} \Big)+ \frac{1-\delta}{2}\sqrt{\Big(1-q^{(h_n+2)}\Big)\Big(1-uvq^{(h_n+1)}\Big)} \bigg) X_n \\
&\leq  \bigg( \frac{1+\delta}{4}\Big(2+ (u+v)q^{h_n} \Big)+ \frac{1-\delta}{4}\Big(2-(1+uv)q^{(h_n+2)}\Big) \bigg) X_n \\
 & \leq \left( 1 + \frac{1}{4}q^{h_n} \left( (1+\delta)(u+v) - (1-\delta)(1+uv)q^2 \right) \right)  X_{n}  \\
 & \leq X_{n} , 
 \end{split}
\end{align} which ensures that $(X_n)_{n \in \lsem 0,N-J_0\rsem}$ is a supermartingale with $X_0=1$. Since all weights  $W$ are bounded from above by $1+\max(|u|,|v|)< 2$, this gives the desired bound on $\mathbf{E}_0[X_M]$.
For \eqref{eq:WeightIdentity}, apply the same arguments as in Lemma \ref{lem:WeightToPolymer}, but with respect to the free random polymer measure. 
\end{proof}

\begin{remark}\label{rem:HarnessComparison}
As pointed out in the introduction, Bryc et al.\  investigate the stationary distribution of the open ASEP using Askey--Wilson processes \cite{BW:AskeyWilsonProcess,BW:QuadraticHarnesses,BW:Density}. These processes exhibit a martingale structure, and it remains an open question whether this fact can be related to the above described supermartingale $(X_n,\mathcal{F}_n)$ for the open WASEP. 
\end{remark}

Next, for fixed $C>0$, we define the \textbf{exit time} $\tau_{C,N}$ from level $CN^{\varepsilon}\log(N)$ as 
\begin{equation*}
\tau_{C,N} = \inf\left\{ n \geq 0 \, \colon \,  S_n^{\mathsf{h}} \geq CN^{\varepsilon}\log(N)  \right\}  .
\end{equation*}
For fixed $M\leq N$, we write $(\tilde{S}_n^{\mathsf{h}})_{n \geq 0}$ for the simple random walk which is non-negative until time $N$ and conditioned to not hit level $CN^{\varepsilon}\log(N)$ until time $M$. Let $p_{x,y}^{\mathsf{h}}$ be given by
\begin{equation*}
p_{x,y}^{\mathsf{h}}  = \P( S^{\mathsf{h}}_n = y \, | \, S^{\mathsf{h}}_{n-1} = x )
\end{equation*} for all $x,y\in \N_0$. Similar to $(S_n^{\mathsf{h}})_{n \geq 0}$, we notice that $(\tilde{S}_n^{\mathsf{h}})_{n \in \lsem 0,N \rsem}$ can be  written for all $n\leq M$ as a time-inhomogeneous Markov chain with transition probabilities
\begin{equation}\label{eq:TransitionsAfterConditioned}
\P\big( \tilde{S}^{\mathsf{h}}_n = y \, \big| \, \tilde{S}^{\mathsf{h}}_{n-1} = x \big) = 
p_{x,y}^{\mathsf{h}} \frac{\textbf{P}_{y}(\tau_{C,N} > M-n)}{\textbf{P}_{x}(\tau_{C,N} > M-n+1)}  
\end{equation} with $x,y \in \lsem CN^{\varepsilon}\log(N)-1\rsem$, where we recall that $\textbf{P}_{x}$ denotes the law of a lazy simple random walk started from $x$ conditioned to stay non-negative until time $N$. Further, recall  $V$ from \eqref{eq:PolymerFunctionMotzkin} and set for all $n \in \lsem 0, M \rsem$
\begin{equation*}
\tilde{X}_n = \prod_{i=1}^{n} \exp(V(\tilde{S}_i^{\mathsf{h}},\tilde{S}^{\mathsf{h}}_{i}-\tilde{S}^{\mathsf{h}}_{i-1})) . 
\end{equation*}
Similar to Lemma \ref{lem:SupermartingaleComparison1}, the following lemma justifies that  $(\tilde{X}_n)_{n  \in \lsem 0,N\rsem}$ gives rise to a supermartingale with respect to the natural filtration $(\tilde{\mathcal{F}}_{n})_{n \in \lsem 0,N \rsem}$ of $(\tilde{S}_n^{\mathsf{h}})_{n \geq 0}$. 
%
%
\begin{lemma}\label{lem:SupermartingaleComparison2}
Let $u,v \in (-1,1)$, and  $q$ satisfy \eqref{eq:qAssumption} for some $\varepsilon>0$ and $c>0$. Recall $M=M(N)$ in the definition of $(\tilde{S}_n^{\mathsf{h}})_{n \in \lsem 0,N\rsem}$. 
Then there exist some $\tilde{J}_0=\tilde{J}_0(u,v,\varepsilon,c)$ such that for all $N$ sufficiently large, $(\tilde{X}_n)_{n \in \lsem 0,M-\tilde{J}_0\rsem}$ is a supermartingale with respect to $(\tilde{\mathcal{F}}_n)_{n \geq 0}$. Moreover, we have that
 $\mathbf{E}_0[\tilde{X}_M] \leq 2^{\tilde{J}_0}$.
%
\end{lemma}
\begin{proof}
A similar computation as for \eqref{eq:ApproximateTransitions} using the reflection principle and a local central limit theorem for the standard lazy simple random walk on $\Z$ ensures that for every $\delta>0$, there exists some constant $\tilde{J}_0>0$ such that 
\begin{equation}\label{eq:modifiedSuper2}
1-\delta \leq \frac{\mathbf{P}_x(\tau_{C,N}>j)}{\mathbf{P}_y(\tau_{C,N}>j+1)} \leq 1 + \delta
\end{equation} uniformly in the choice of $j\geq \tilde{J}_0$ and $x,y \in \lsem CN^{\varepsilon}\log(N) \rsem$ with $|x-y|\leq 1$. 
Together with  \eqref{eq:ApproximateTransitions} and \eqref{eq:TransitionsAfterConditioned}, taking $\delta>0$ sufficiently small, a similar computation as for \eqref{eq:SupermartingaleCheck} yields that $(\tilde{X}_n)_{n \in \lsem 0,M-\tilde{J}_0\rsem}$  is a supermartingale with $\tilde{X}_0=1$, and thus $\mathbf{E}_0[\tilde{X}_M] \leq 2^{\tilde{J}_0}$.
\end{proof}

\begin{proof}[Proof of Proposition \ref{pro:DelocalizationHalfspaceMartingale}]

By a change of measure, it suffices to show that there exists some $C^{\prime}=C^{\prime}(u,v,\varepsilon)>0$ such that for $M=C^{\prime}N^{3\varepsilon}\log(N)$
\begin{align}\label{eq:ConditionalDecomposition}
 \Pb^{\mathsf{f}}(\tau_{C,N} > M) = \frac{1}{\mathbf{E}_0[X_M]}\mathbf{E}_0[X_M \mathds{1}_{\{\tau_{C,N} > M\}}] =  \frac{\mathbf{E}_0[\tilde{X}_M]}{\mathbf{E}_0[X_M]} \mathbf{P}_0 ( \tau_{C,N} > M )
\end{align} converges to $0$ as $N \rightarrow \infty$.
We claim that there exists some $c^{\prime}=c^{\prime}(u,v)>0$ such that
\begin{equation}\label{eq:LowerBoundPartitionFunction}
\lim_{N \rightarrow \infty} \frac{\mathcal{Z}_{N,V}^{\mathsf{f}}}{4^{N}} \exp(-c^{\prime} N^{\varepsilon}\log(N)) = \infty . 
\end{equation} To see this, consider the set of trajectories
\begin{equation*}
G = \big\{ \omega \in \mathcal{A}^{M} \, \colon \, h_{x}(\omega) \geq \min(x,\tilde{c}N^{\varepsilon}\log(N)) \text{ for all } x\in \lsem M \rsem \big\} , 
\end{equation*} with $\tilde{c}>0$. Choosing $\tilde{c}$ sufficiently large, there exists some $c^{\prime}=c^{\prime}(\tilde{c})>0$ such that each path in $G$ has weight at least $\exp(-c^{\prime}N^{\varepsilon})$. At the same time,  $|G| \geq 4^M\exp(-c^{\prime}N^{\varepsilon}\log(N))$ for all $N$ sufficiently large. Hence, recalling \eqref{eq:WeightIdentity} for $\mathbf{E}_0[X_M]$, Lemma \ref{lem:SupermartingaleComparison1} ensures that
\begin{equation*}
\mathbf{E}_0[X_M] \geq \exp(-2c^{\prime} N^{\varepsilon}\log(N))
\end{equation*}
for all $N$ sufficiently large. Choosing now the constant $C^{\prime}=C^{\prime}(c^{\prime},\tilde{c},C)>0$ sufficiently large, 
\begin{equation*}
\mathbf{P}_0(\tau_{C,N} > M ) \leq \exp(-3c^{\prime} N^{\varepsilon}\log(N)) .
\end{equation*}
Thus, using \eqref{eq:ConditionalDecomposition}, and Lemma \ref{lem:SupermartingaleComparison2} to bound $\mathbf{E}_0[\tilde{X}_M]$, this finishes the proof.
\end{proof}
\begin{remark}\label{rem:StrongRepulsion}
The same arguments as for Proposition \ref{pro:DelocalizationHalfspaceMartingale} 
ensure that for some constants $C,c^{\prime}>0$, a lazy simple random walk trajectory of length $M \gg N^{3\varepsilon}\log(N)$, weighted according to \eqref{def:FreePolymerMeasure} with $V$ from \eqref{eq:PolymerFunctionMotzkin}, reaches with probability at least $\exp(-2c^{\prime} N^{\varepsilon}\log(N))$ a height of at least $CN^{\varepsilon}\log(N)$ until time $M$,  uniformly in its  starting position.
\end{remark}

\subsection{Delocalization for the open WASEP in the maximal current phase}\label{sec:DelocalizationPinning}

In this section, we transfer the delocalization result in Proposition \ref{pro:DelocalizationHalfspaceMartingale} for the free polymer to a delocalization result for the constraint polymer. Recall that we assume $u,v \in (-1,1)$.

\begin{proposition}\label{pro:DelocalizationHardWallPinning}
Consider $q$ from \eqref{eq:qAssumption} with $\varepsilon < \frac{1}{3}$, and recall $\Pb^{\textup{\textsf{c}}}_{N,V}$ from \eqref{def:ConstraintPolymerMeasure}. Then
\begin{equation}\label{eq:DelocConstraintMidPoint}
\lim_{ N\rightarrow \infty} \Pb^{\textup{\textsf{c}}}_{N,V}\left( A_N^{i,j} \right)  = 1
\end{equation} for all $j=j_N$ and $i=i_N \leq N/2$ such that $j_N \ll \sqrt{i_N}$ and $\min(i_N,N-i_N)\gg N^{3\varepsilon}\log(N)$. Moreover, for every interval $I=\lsem a,b\rsem$ satisfying the assumptions \eqref{eq:AssumptionsMaxCurrentWASEP}, 
\begin{equation}\label{eq:DelocConstraintGeneral}
\lim_{ N\rightarrow \infty} \Pb^{\textsf{c}}\left( h_{x}(\zeta)^{2} \geq \sqrt{ |I| \min(a,N-b) } \text{ for all } x\in I \right)  = 1 . 
\end{equation}
\end{proposition}

\subsubsection{Improved bounds on the delocalization under the free polymer measure}\label{sec:PreliminariesConstraintPolymer}

In order to show Proposition \ref{pro:DelocalizationHardWallPinning} in Section \ref{sec:DelocalizationPinningLocalLimit}, we require an improved bound on the delocalization, stated below as Lemma \ref{lem:ReachHighLevel}. 
To do so, we start with a basic observation on the simple random walk $(S^{\mathsf{h}}_{n})_{n \geq 0}$ conditioned to stay positive until time $N$. Recall the law $\Pb^{\mathsf{h}}_N$ and that we write $P_x$ for the law of a lazy simple random walk $(S_{n})_{n \in \N_0}$ on $\Z$ started from~$x$.

\begin{lemma}\label{lem:GeometricWaitingTimes}
For any $(a_N)_{N \in \N}$ and $(b_N)_{N \in \N}$ with $a_N > b_N \gg 1$, there exist some absolute constant $C>0$ such that for all $N$ large enough
\begin{equation}\label{eq:GeometricWaitingCounting}
\Pb^{\mathsf{h}}_N\big( S^{\mathsf{h}}_{n} \geq b_N \text{ for all } n\geq 0 \, \big| \, S^{\mathsf{h}}_0=a_N \big) \geq 1 - C \frac{b_N}{a_N} . 
\end{equation}
\end{lemma}
\begin{proof}
We show \eqref{eq:GeometricWaitingCounting} by a comparison to a lazy simple random walk on $\Z$. Note that
\begin{equation*}
\Pb^{\mathsf{h}}_N\big( S^{\mathsf{h}}_{n} \geq b_N \text{ for all } n\geq 0 \, \big| \, S^{\mathsf{h}}_0=a_N \big) = \frac{P_{a_N}( S_n \geq b_N \text{ for all } n \geq 0 ) }{P_{a_N}( S_n \geq 0 \text{ for all } n \geq 0 )} \, .
\end{equation*} Using the reflection principle for the simple random walk, we see that for all $x \in \N$, 
\begin{equation*}
P_{x}( S_n \geq 0 \text{ for all } n \in \lsem N \rsem )  =  P_{0}( S_N  \in [-x,x] ) . 
\end{equation*}
A local central limit theorem for the simple random walk -- see for example Chapter 2 in~\cite{LL:RWIntroduction} --  now gives the desired bound. 
\end{proof}
\begin{remark}
Alternatively, Lemma \ref{lem:GeometricWaitingTimes} can be shown by interpreting $(S^{\mathsf{h}}_{n})$ as an $h$-transformed lazy simple random walk, i.e.\ as a random walk on an electrical network $(c_e)$ on $\N_0$ with respect to conductances $
c_{\{x,x+1\}} = x(x+1)$. By a standard argument on network reduction  the effective resistance between $a_N$ and $b_N$  is of order $a_N$ while the effective resistance between $b_N$ and infinity is of order $b_N$, giving the desired result.
\end{remark}

It will be convenient to work with a measure $\mathbf{P}^{\mathsf{f},N}_x$ on the space  $\Lambda_N$ from \eqref{def:LatticePaths} when we shift the weight function $V$ from \eqref{eq:PolymerFunctionMotzkin} to $V_x(h,y) := V(h + x, y)$, i.e.\ for all $\zeta \in \Lambda_N$
\begin{equation}\label{def:ShiftedMeasure}
\mathbf{P}^{\mathsf{f},N}_x ( \zeta ) := \mathbf{P}^{\mathsf{f}}_{N,V_x} ( \zeta ) . 
\end{equation} 
A key observation is that the measure $\mathbf{P}_x^{\mathsf{f},N}$ satisfies a spatial Markov property.
\begin{lemma}\label{lem:SpatialMarkov}
For all $M \in \lsem N \rsem $, and all $\zeta=(\zeta_1,\zeta_2) \in \Lambda_N$ with $\zeta_1 \in \Lambda_M$ and $\zeta_2 \in \Lambda_{N-M}$, 
\begin{equation}
\mathbf{P}^{\mathsf{f}}( \zeta ) = \mathbf{P}^{\mathsf{f},N}_0 ( \zeta_1 ) \mathbf{P}^{\mathsf{f},N-M}_{h_M(\zeta_1)} ( \zeta_2 ) . 
\end{equation}
\end{lemma}
\begin{proof}
This follows from the product form of the weight function $V$, and the Markov property of the simple random walk on $\Z$ conditioned to stay non-negative until time~$N$. 
\end{proof}
With a slight abuse of notation, we will write $\mathbf{P}^{\mathsf{f}}=\mathbf{P}^{\mathsf{f}}_{N,V}=\mathbf{P}^{\mathsf{f},N}_0$ when $N$ and $V$ are clear from the context. Note that similar to Lemma \ref{lem:WeightToPolymer}, we can represent the probability of a path $\zeta$ under $\mathbf{P}^{\mathsf{f},N}_{x}$
using a suitable weight function. More precisely, for all $\omega \in \mathcal{A}^{N}$, we define the weight $W^{(x)}(\omega) := \prod_{i=1}^{N} W^{(x)}_i(\omega)$ by 
\begin{equation}\label{def:WeightsIndividuallyX}
W^{(x)}_i(\omega) := \begin{cases}
(1-q)\big(1-q^{h_i(\omega)+1-x}\big) &\text{ if } \omega(i)=\Nc \\
(1-q)\big(1+uq^{h_i(\omega)-x}\big) &\text{ if } \omega(i)=\Eb \\
(1-q)\big(1+vq^{h_i(\omega)-x}\big) &\text{ if } \omega(i)=\Ew \\
(1-q)\big(1-uvq^{h_i(\omega)-1-x}\big) &\text{ if } \omega(i)=\Sc ,
\end{cases}
\end{equation} whenever $h_{i}(\omega)\geq -x$ for all $i \in \lsem N \rsem$. Furthermore, note that for all $\omega \in \mathcal{A}^{N}$, we can associate a lattice path $v_\omega \in \Lambda_N$ via the relation \eqref{eq:BiColorToPath}. The next lemma states a correspondence between the path weights in \eqref{def:WeightsIndividuallyX} and the measure $\mathbf{P}^{\mathsf{f},N}_x$.
 Since the arguments are one-to-one to Lemma \ref{lem:WeightToPolymer}, we omit the proof.
\begin{lemma}\label{lem:WeightUnderPolymerShifted} 
Let $uv< 1$ and $N\in \N$, and fix some $x\in \N$. Then for all $\zeta \in \Lambda_N$, 
\begin{equation}\label{eq:EqualityOfMeasuresShifted}
\mathbf{P}^{\mathsf{f},N}_x (\zeta)= \frac{1}{\mathcal{Z}_{N}^{(x)}}\sum_{\omega \, \colon \, v_{\omega}=\zeta} W_{x}(\omega) , 
\end{equation} where $\mathcal{Z}_{N}^{(x)}$ is a suitable normalization constant. 
\end{lemma}

Next, recall from \eqref{eq:qAssumption} that $q$ takes the form $q=\exp(-N^{\varepsilon}c_q)$ for some constant $c_q>0$. To simplify notation, we will set from now on $M=M(N,q)=3c_q^{-1}N^{\varepsilon}\log(N)$ for all $N\in \N$. The following lemma extends the delocalization result from Proposition \ref{pro:DelocalizationHalfspaceMartingale}. 

\begin{lemma}\label{lem:ReachHighLevel}
Consider $q$ from \eqref{eq:qAssumption} with $\varepsilon < \frac{1}{3}$. Then for all $x \gg N^{3\varepsilon/2}\log(N)$, and for all $N$ large enough, uniformly in $y\in \N$
\begin{equation}\label{eq:ReachLevelStatement}
\Pb_y^{\textup{\textsf{f}}}(\exists z \in \lsem x^2 \rsem \text{ such that } h_z(\zeta) > x\log^{-1}(N) ) \geq 1- N^{-3}  .
\end{equation} 
\end{lemma}
\begin{proof}
Set $I_{N,x}= \lsem M,x\log^{-1}(N) \rsem$ and define the time $\tau^{\prime}$ as the first exit time from the interval $I_{N,x}$, that is
\begin{equation*}
\tau^{\prime}= \inf\left\{ n \geq 0 \, \colon \,  h_n(\zeta) \notin I_{N,x} \right\} . 
\end{equation*}
In the following, our goal is to show that for all $N$ sufficiently large,
\begin{equation}\label{eq:HitHigh1}
 \Pb^{\textsf{f}}_z\left( \tau^{\prime} \leq \frac{1}{20}x^2\log^{-1}(N) \text{ and } h_{\tau^{\prime}}(\zeta) = \lfloor x\log^{-1}(N) \rfloor \right)  \geq \frac{3}{4}
\end{equation} uniformly in the starting position $z \geq C M$, for some suitable constant $C>0$ specified later on.
Provided that \eqref{eq:HitHigh1} holds, we have the following strategy to conclude. For all $x \in \N$, we define the family of events
\begin{equation*}
B^{x,i}_1 := \left\{ \zeta \, \colon \, h_{y}(\zeta) \geq CM \text{ for some } y \in \left\lsem \frac{(i-1)}{20}x^2\log^{-1}(N), \frac{i}{20}x^2\log^{-1}(N) \right\rsem \right\} .
\end{equation*} Since $x^2\log^{-1}(N) \gg N^{3\varepsilon}\log(N)$ by our assumptions,  Proposition \ref{pro:DelocalizationHalfspaceMartingale} and Remark \ref{rem:StrongRepulsion} ensure that there exist constant a constant $c=c(C)>0$ such that 
\begin{equation*}
 \Pb^{\textsf{f}}_z( B^{x,i}_1 ) \geq 1-\exp(-2cN^{\varepsilon}\log(N))
\end{equation*}
uniformly in the choice $i \in \N$ and $z \in \N$, provided that $N$ is large enough. Let now
\begin{equation*}
B^{x,i}_2 := \left\{ \zeta \, \colon \, h_{y}(\zeta) \geq x\log^{-1}(N) \text{ for some } y \in \left\lsem \frac{(i-1)}{20}x^2\log^{-1}(N), \frac{(i+1)}{20}x^2\log^{-1}(N) \right\rsem \right\} .
\end{equation*}
Note that conditioning on the event $B^{x,i}_1$ we see by \eqref{eq:HitHigh1} and Lemma \ref{lem:SpatialMarkov} that 
\begin{equation}\label{eq:HitHighEventProb}
\Pb_y^{\textsf{f}}( B^{x,i}_2  \, | \, B^{x,i}_1 ) \geq \frac{1}{2}
\end{equation} uniformly in the starting position $y\in \N$ and $i\in \N$. Using again the spatial Markov property from Lemma \ref{lem:SpatialMarkov}, we iterate \eqref{eq:HitHighEventProb} along all odd $i$ to obtain
\begin{align*}
\Pb_y^{\textup{\textsf{f}}}\big(\exists z \in  \lsem x^2 \rsem  \, \colon \, h_z(\zeta) > x\log^{-1}(N) \big) \geq \Pb_y^{\textup{\textsf{f}}}  \left( \bigcup_{i \in  \lsem 10\log(N) \rsem }  B^{x,2i-1}_2\right) \geq 1 - \frac{1}{2^{10\log(N)}} \geq 1 - N^{-3}
\end{align*}
for all $N$ sufficiently large, and all $y\in \N$. This gives the desired bound in \eqref{eq:ReachLevelStatement}. 
It remains to show that \eqref{eq:HitHigh1} holds. To do so, fix some $z \geq CM$ and partition the set $\Lambda_N$ of lattice paths of length $N$ into two sets $\mathcal{B}_1$ and $\mathcal{B}_2$ as follows. For every $\zeta\in \Lambda_N$, recall
\begin{equation}\label{def:ReturnTimes}
\tau_{M-z} = \inf\left\{ n \geq 1 \, \colon \, h_{\zeta}(n)=M-z \right\} 
\end{equation} as the first intersection point with the level $M-z$, and say that $\zeta \in \mathcal{B}_1$ if $\tau(\zeta)= \infty$, and $\zeta \in \mathcal{B}_2$ otherwise.  We claim that there exists some $C>0$ such that for all $N$ large enough
\begin{equation}\label{eq:PathsInB1}
\Pb^{\textup{\textsf{f}}}_z( \zeta \in \mathcal{B}_1 ) \geq \frac{9}{10} 
\end{equation} whenever $z \geq CM$. To show this, notice that by Lemma \ref{lem:WeightUnderPolymerShifted}, it suffices to prove
\begin{equation*}
\sum_{\omega \, \colon \, v_{\omega}\in \mathcal{B}_2} W_x(\omega) \leq \frac{1}{10} \mathcal{Z}_N^{(x)} .
\end{equation*} 
From the definition of the set $\mathcal{B}_1$, and assuming $z \geq CM$, we see that for all $\omega \in \mathcal{B}_1$ 
\begin{equation}\label{eq:WeightBound}
W_{z}(\omega) \in \left[1-\frac{1}{N},1+\frac{1}{N}\right] 
\end{equation} when $N$ is sufficiently large, taking a suitably large constant  $C>0$. In particular, 
\begin{equation}\label{eq:WeightComparison1}
\sum_{\omega \, \colon \, v_{\omega}\in \mathcal{B}_1} W_z(\omega) \geq  \left(1-\frac{1}{N}\right) |\mathcal{B}_1| . 
\end{equation}
Splitting now the paths in $\mathcal{B}_2$ according to the value of $\tau_{M-z}$, we get that 
\begin{equation}\label{eq:WeightComparison2}
\sum_{\omega \, \colon \, v_{\omega}\in \mathcal{B}_2} W_z(\omega) \leq \left(1+\frac{1}{N}\right)\sum_{i=1}^{N} \Pb^{\textup{\textsf{f}}}_z(\tau_{M-z}=i) \frac{|\mathcal{B}_2|}{\big|\Lambda_{N-i}^{(M)}\big|} \sum_{\omega \in \Lambda^{(M)}_{N-i}} W_{M}(\omega) \leq \left(1+\frac{1}{N}\right)|\mathcal{B}_2| 2^{J}
\end{equation} 
where we use the definition of $W_z$ and the spatial Markov property from Lemma \ref{lem:SpatialMarkov} for the first step, and the supermartingale property in Lemma \ref{lem:SupermartingaleComparison1} with constant $J$ for the second step. Combining now \eqref{eq:WeightComparison1} and \eqref{eq:WeightComparison2},  and bounding $|\mathcal{B}_2|/|\mathcal{B}_1|$ by Lemma \ref{lem:GeometricWaitingTimes}, we obtain \eqref{eq:PathsInB1}. Since all paths in $\mathcal{B}_1$ satisfy \eqref{eq:WeightBound}, the bound in \eqref{eq:PathsInB1}  together with a standard estimate on the exit time of a simple random walk when restricting to the paths in $\mathcal{B}_1$ yields \eqref{eq:HitHigh1}. This finishes the proof. 
\end{proof}

\subsubsection{Delocalization with pinning via a path decomposition}\label{sec:DelocalizationPinningLocalLimit}

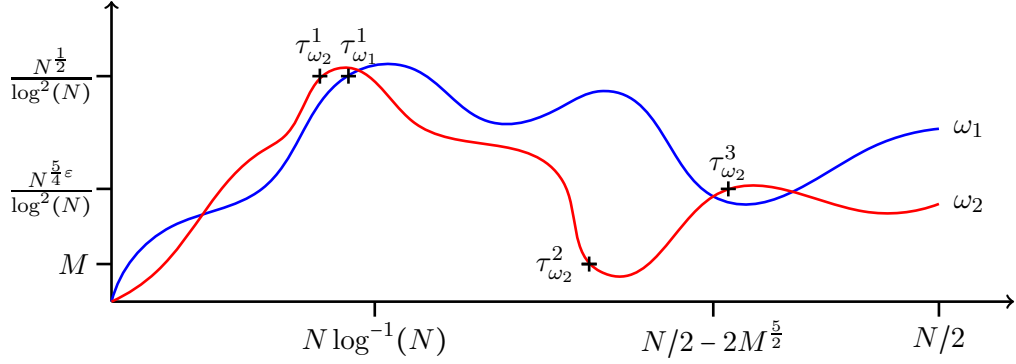
\begin{figure}
    \centering
\begin{tikzpicture}[scale=1]
\draw [->,line width=1pt] (0,0) to (12,0); 	
\draw [->,line width=1pt] (0,0) to (0,4); 	

\draw [line width=1pt] (0,0.5) to  (-0.2,0.5); 	
\draw [line width=1pt] (0,1.5) to  (-0.2,1.5); 
\draw [line width=1pt] (0,3) to  (-0.2,3); 

\draw [line width=1pt] (3.5,-0.2) to  (3.5,0); 
\draw [line width=1pt] (8,-0.2) to  (8,0); 
\draw [line width=1pt] (11,-0.2) to  (11,0); 
		
 
 \node (H1) at (-0.5,0.5) {$M$};
 \node (H2) at (-0.8,1.5) {$\frac{N^{\frac{5}{4}\varepsilon}}{\log^2(N)}$};
 \node (H3) at (-0.8,3) {$\frac{N^{\frac{1}{2}}}{\log^2(N)}$};
  
 \node (H4) at (3.5,-0.5) {$N\log^{-1}(N)$} ; 
 \node (H5) at (8,-0.5) {$N/2 - 2M^{\frac{5}{2}}$} ; 
  \node (H5) at (11,-0.5) {$N/2$} ;

  \node (H5) at (2.7,3.4) {$\tau^{1}_{\omega_2}$} ; 
  \node (H5) at (3.3,3.4) {$\tau^{1}_{\omega_1}$} ; 
  
\draw [line width=1pt] (6.25,0.5) to  (6.45,0.5);

  \node (H5) at (5.9,0.5) {$\tau^{2}_{\omega_2}$} ;
  \node (H5) at (8.2,1.85) {$\tau^{3}_{\omega_2}$} ;

  \draw[blue,line width=1 pt] (0,0) to[curve through={(0.8,1)..(2,1.5)..(3,2.9)..(4,3.1)..(5,2.4)..(6,2.6)..(6.5,2.8)..(8,1.4)..(10,2)}] (11,2.3);

  \draw[red,line width=1 pt] (0,0) to[curve through={(0.6,0.4)..(2,2)..(2.3,2.2)..(3,3.1)..(4,2.4)..(6,1.6)..(6.5,0.4)..(8,1.4)..(10,1.2)}] (11,1.3);	 
 
  \node (H5) at (11.4,2.3) {$\omega_1$} ;  
  \node (H5) at (11.4,1.3) {$\omega_2$} ; 

 

\draw [line width=1pt] (2.77,3.1) to  (2.77,2.9); 
\draw [line width=1pt] (3.15,3.1) to  (3.15,2.9); 
\draw [line width=1pt] (2.67,3) to  (2.87,3); 
\draw [line width=1pt] (3.05,3) to  (3.25,3);

\draw [line width=1pt] (6.25,0.5) to  (6.45,0.5);   
\draw [line width=1pt] (6.35,0.6) to  (6.35,0.4);

\draw [line width=1pt] (8.1,1.5) to  (8.3,1.5);   
\draw [line width=1pt] (8.2,1.4) to  (8.2,1.6);

\end{tikzpicture}
    \caption{Visualization of the different stopping times in \eqref{eq:PathSplittingTimes} for the trajectories of the free random polymer. The path $\omega_1$ drawn in blue belongs to the event $\mathcal{C}_1$, the path $\omega_2$ drawn in red to the event $\mathcal{C}_2$. 
    }
    \label{fig:Events}
\end{figure}
We have now all tools for the delocalization for general values of  $u,v \in (-1,1)$ in Proposition \ref{pro:DelocalizationHardWallPinning}. Without loss of generality, we assume in the following that $N$ is even. We start with a bound on the height at position $N/2$ in the constraint random polymer. The next result is our key lemma, which also serves as an outline for the proof of Proposition \ref{pro:DelocalizationHardWallPinning}. 
\begin{lemma}\label{lem:IntersectionPoint} Assume that $q$ satisfies \eqref{eq:qAssumption} with some $\varepsilon < \frac{1}{3}$. Then for any  $j=j_N \ll \sqrt{N}$,
\begin{equation}\label{eq:IntersectionPointStatement}
\lim_{ N\rightarrow \infty} \Pb^{\mathsf{c}}_{N,V}\left( h_{N/2}(\zeta) \geq j_N \right)  = 1 . 
\end{equation}
\end{lemma}
In order to show Lemma \ref{lem:IntersectionPoint}, we rely on a path decomposition similar to the proof of Lemma \ref{lem:ReachHighLevel}. Let $\Psi^{(0)}_N \subseteq \mathcal{A}^{N}$ be the set of all bi-colored lattice paths $\omega$ with $h_i(\omega)\geq 0$ for all $i\in \lsem N \rsem$. For $\omega \in \Psi^{(0)}_N$, we define the times $(\tau^{i}_{\omega})_{i \in \lsem 4 \rsem}$ by
\begin{align}
\begin{split}\label{eq:PathSplittingTimes}
\tau^{1}_{\omega} &:= \inf\left\{ n \geq 0 \,  \colon \,  h_{z}(\omega) = \frac{N^{\frac{1}{2}}}{\log^2(N)} \right\} \\
\tau^{2}_{\omega} &:= \inf\left\{ n \geq \tau^{1}_{\omega}  \,  \colon \,  h_{z}(\omega) = M \right\} \\
\tau^{3}_{\omega} &:= \inf\left\{ n \geq \tau^{2}_{\omega} \,  \colon \,  h_{z}(\omega) = \frac{N^{\frac{5}{4}\varepsilon}}{\log^2(N)} \right\} \\
\tau^{4}_{\omega} &:= \inf\left\{ n \geq \tau^{3}_{\omega}  \,  \colon \,  h_{z}(\omega) = M \right\} ,
\end{split} 
\end{align} where we recall $M=3c_q N^{\varepsilon}\log(N)$ for $c_q>0$ from \eqref{eq:qAssumption}. Further, we partition the set $\Psi^{(0)}_N$ into subsets $(\mathcal{C}_i)_{i \in \lsem 3 \rsem}$, where  $\mathcal{C}_3 := \Psi^{(0)}_N  \backslash (\mathcal{C}_1 \cup \mathcal{C}_2)$  for 
\begin{align}
\begin{split}\label{eq:PathDecompositionGlobal}
\mathcal{C}_1 &:= \left\{ \omega \in  \Psi^{(0)}_N \, \colon \, \tau^{1}_{\omega} \leq \frac{N}{\log(N)} \ \wedge \  \tau^{2}_{\omega} = \infty \right\} \\
\mathcal{C}_2 &:= \left\{ \omega \in  \Psi^{(0)}_N \, \colon \, \tau^{1}_{\omega} \leq \frac{N}{\log(N)} \ \wedge \   \tau^{2}_{\omega} < \frac{N}{2} - 2M^{\frac{5}{2}}  \ \wedge \ \tau^{3}_{\omega}-\tau^{2}_{\omega} < M^{\frac{5}{2}} \ \wedge \ \tau^{4}_{\omega}=\infty \right\} .
\end{split}
\end{align}
We refer to Figure \ref{fig:Events} for a visualization. 
With a slight abuse of notation, we treat $\mathcal{C}_i$ as a subsets of $\Lambda_N^{(0)}$, i.e.\ for $i\in \lsem 3 \rsem$, we say that $\zeta \in \mathcal{C}_i$ with $\zeta \in \Lambda_N^{(0)}$ if $\zeta = v_{\omega}$ for some $\omega \in \mathcal{C}_i$. 

\begin{remark}
The choice of the times $(\tau^i_{\omega})_{i \in [4]}$ in \eqref{eq:PathSplittingTimes} and the assumption $\varepsilon<\frac{1}{3}$ in Lemma \ref{lem:IntersectionPoint} may seem unnatural at first glance, but have the following intuitive explanation. Proposition \ref{pro:DelocalizationHalfspaceMartingale} guarantees that after order $N^{3\varepsilon}\log(N)$ steps, the path has reached a height of at least $N^{\varepsilon}\log(N)$ with high probability. Thus, in order to give bounds on the height of the path after $N/2$ steps, we need to assume that $\varepsilon<\frac{1}{3}$. The times $(\tau^i_{\omega})_{i \in [4]}$ then allow for a fine allocation of the height after $N/2$ steps, depending on whether the path has crossed level $M$.
\end{remark}

We argue in the following that a lattice path according to $\Pb^{\mathsf{c}}_{N,V}$, restricted to the first $N/2$ positions, will with high probability belong to the set $\mathcal{C}_1$, while the probability to see a path in $\mathcal{C}_3$ is of order $o(N^{-1/4})$. To do so, we first consider the measure $\Pb^{\mathsf{f}}_{N,V}$ and $\mathcal{C}_3$.
\begin{lemma}\label{lem:PathsSmall1}
Recall the partition function $\mathcal{Z}_{N/2,V}^{\mathsf{f}}$ and let $\varepsilon< \frac{1}{3}$ for $q$ in \eqref{eq:qAssumption}. Then
\begin{equation}\label{eq:PathSmall1}
\Pb^{\mathsf{f}}_{N/2,V}( \omega \in \mathcal{C}_3) = \frac{1}{\mathcal{Z}_{N/2,V}^{\mathsf{f}}(1-q)^{N/2}}\sum_{\omega \in \mathcal{C}_3 } W(\omega) = o(N^{-1/4}) . 
\end{equation}
\end{lemma}
\begin{proof}
Using Lemma \ref{lem:ReachHighLevel}, it suffices to 
show \eqref{eq:PathSmall1} with respect to the configurations in $\mathcal{C}_3$ with $\tau^{1}_{\omega} \leq \frac{N}{\log(N)}$. To do so, consider for all $i\in  \lsem N\log^{-1}(N) \rsem$ and $j \in \lsem N-i \rsem $ the sets
\begin{align}\label{eq:SetsC3ij}
\mathcal{C}_3^{i,j} &:=  \left\{ \omega \in \Psi^{(0)}_{N/2} \, \colon \, \tau^{1}_{\omega} = i \text{ and }  \tau^{2}_{\omega}  = \frac{N}{2}-j   \right\}  .
\end{align} We start by showing that there exist some constant $c>0$ such that for $N$ large enough
\begin{equation}\label{eq:CountingC3Bound}
\frac{\big|\mathcal{C}_3^{i,j}\big|}{ \left| \left\{\omega \in \Psi^{(0)}_{N/2} \, \colon \, \tau^{1}_{\omega} = i  \right\} \right| } \leq \frac{c}{N} \min\left(1, \frac{M}{\sqrt{j}}  \right) \log^{2}(N)
\end{equation} uniformly in $i\in  \lsem N\log^{-1}(N) \rsem$ and $j \in \lsem 2M^{5/2} \rsem $.
%
%
%
To see this, recall that $P_{z}$ denotes the law of a lazy simple random walk $(S_n)_{n \geq 0}$ on $\Z$ started from $z$, and note that expressing the cardinality of the sets in \eqref{eq:CountingC3Bound} as probabilities of a simple random walk yields
\begin{align*}
\frac{\big|\mathcal{C}_3^{i,j}\big|}{ \left| \left\{\omega \in \Psi^{(0)}_{N/2} \, \colon \, \tau^{1}_{\omega} = i  \right\} \right| } = \frac{P_{K}\Big( S_{n} \geq 0 \, \forall\, n \in \lsem \frac{N}{2}-i-j\rsem \wedge S_{\frac{N}{2}-j-i}=0\Big)P_{M}(S_{n} \geq 0 \,\forall \, n \in \lsem j \rsem)}{P_{K}( S_{n} \geq 0 \, \forall \, n \in \lsem \frac{N}{2}-i\rsem)} ,
\end{align*}
where we set $K:=N^{1/2}\log^{-2}(N)$. A computation using the reflection principle shows that
\begin{equation*}
P_{K}\left( S_{n} \geq 0 \, \forall\, n \in \Big\lsem \frac{N}{2}-i-j\Big\rsem \, \wedge \, S_{\frac{N}{2}-j-i}=0\right) =   P_0 ( S_{\frac{N}{2}-i-j} = K ) - P_0 ( S_{\frac{N}{2}-i-j} = K +2 ) .
\end{equation*} Together with a local central limit theorem, there exists some  $c_1>0$ such that
\begin{equation*}
P_{K}\Big( S_{n} \geq 0 \, \forall \,  n \in \Big\lsem \frac{N}{2}-i-j\Big\rsem \, \wedge \, S_{\frac{N}{2}-j-i}=0\Big)  \leq \frac{c_1}{N }
\end{equation*} uniformly in $i\in  \lsem N\log^{-1}(N) \rsem$ and $j \in \lsem 2M^{5/2} \rsem $. Furthermore, using again the local central limit theorem and the reflection principle, there exist constants $c_2,c_3>0$ such that for any $x,y \in \N$
\begin{equation*}
c_2 \min\left(1, \frac{x}{\sqrt{y}}\right) \leq P_{x}(S_{n} \geq 0 \,\forall \, n \in \lsem y \rsem) \leq  c_3 \min\left(1, \frac{x}{\sqrt{y}}\right) . 
\end{equation*} Combining the above observations, this yields \eqref{eq:CountingC3Bound}. Summing over all $j \in \lsem 2M^{5/2} \rsem $ and using that $\varepsilon<\frac{1}{3}$ by our assumptions, we see that there exists some $\delta=\delta(\varepsilon)>0$ with
\begin{equation}\label{eq:RatioC3}
\left| \left\{\omega \in \Psi^{(0)}_{N/2} \, \colon \, \tau^{1}_{\omega} = i  \right\} \right|^{-1}\sum_{j=1}^{2M^{5/2}} \big|\mathcal{C}_3^{i,j}\big| = o(N^{-1/4-\delta}) . 
\end{equation}
Next, we use \eqref{eq:CountingC3Bound} and \eqref{eq:RatioC3} in order to argue that
\begin{equation*}
\Pb^{\mathsf{f}}_{N/2,V}\left(  \exists i\in \Big\lsem \frac{N}{\log(N)} \Big\rsem , j \in \Big\lsem 2M^{\frac{5}{2}} \Big\rsem \, \colon \, \omega \in \mathcal{C}^{i,j}_3 \right) =\frac{1}{\mathcal{Z}_{N/2,V}^{\mathsf{f}}} \sum_{i=1}^{\frac{N}{\log(N)}} \sum_{j=1}^{2M^{\frac{5}{2}}} \sum_{\omega \in \mathcal{C}_3^{i,j}} W(\omega) = o(N^{-1/4}) . 
\end{equation*} 
To do so, we will follow a similar reasoning as in Lemma \ref{lem:ReachHighLevel} using the spatial Markov property and supermartingale arguments. Recall from Lemma \ref{lem:ReachHighLevel} that 
\begin{equation}\label{eq:HitHighDirectly}
\Pb^{\mathsf{f}}_{N/2,V}\Big( \tau_{\omega}^{1} \leq \frac{N}{\log(N)} \Big) = \frac{1}{\mathcal{Z}_{N/2,V}^{\mathsf{f}}} \sum_{i=1}^{N/\log(N)} \sum_{\omega \in \Psi_{N/2}^{(0)}} W(\omega) \mathds{1}_{\{ \tau_{\omega}^{1} = i \}} \geq 1 - N^{-3} . 
\end{equation} Thus, by Lemma \ref{lem:SpatialMarkov}, it suffices to show that uniformly in $i\in  \lsem N\log^{-1}(N) \rsem$, 
\begin{equation}\label{eq:ResampleReduction}
\mathbf{P}^{\mathsf{f},N-i}_K\Big(\tau_{\omega}^{2} \in \Big\lsem   \frac{N}{2}-i - 2 M^{\frac{5}{2}}, \frac{N}{2} - i \Big\rsem \Big)  
= o(N^{-1/4}) , 
\end{equation} where we recall the measure $\mathbf{P}^{\mathsf{f},N}_K$ from Lemma \ref{lem:WeightUnderPolymerShifted}. By Lemma \ref{lem:GeometricWaitingTimes}, bounding the weights $W^{(k)}(\omega)$ as in \eqref{eq:WeightBound}, we notice that the partition function $\mathcal{Z}_{N-i}^{(K)}$ for  $\mathbf{P}^{\mathsf{f},N-i}_K$ satisfies
\begin{equation}\label{eq:AlternativeRatioBound1}
\mathcal{Z}_{N-i}^{(K)} \geq  c_4  \Big| \Big\{  \omega \in \Lambda_{\frac{N}{2}-i}^{(K)} \Big\} \Big| \geq  \log^{-4}(N) 4^{N/2-i}
\end{equation}
for some constant $c_4>0$ and all $N$ sufficiently large. Using Lemma \ref{lem:SupermartingaleComparison1} in order to bound the weight of paths after position $\tau_{\omega}^{2}$, together with \eqref{eq:RatioC3} and \eqref{eq:HitHighDirectly}, we see that
\begin{equation}\label{eq:AlternativeRatioBound2}
\frac{1}{4^{N/2-i}} \sum_{j=1}^{2M^{5/2}} \sum_{\omega \in  \Lambda_{N/2-i}^{(K)}}W(\omega)\mathds{1}_{\{ \tau_{\omega}^{2}=N/2-i-j \}}  = o(N^{-1/4-\delta^{\prime}})  
\end{equation} for some $\delta^{\prime}=\delta^{\prime}(\varepsilon)>0$. Combining now \eqref{eq:AlternativeRatioBound1} and \eqref{eq:AlternativeRatioBound2} yields \eqref{eq:ResampleReduction}. Next, we consider
\begin{equation*}
\mathcal{D}:= \left\{  \tau_{\omega}^{1} \leq N\log^{-1}(N) \text{ and } \tau_{\omega}^{2} \notin \lsem N/2 - 2M^{5/2}, N/2 \rsem \text{ and } \tau_{\omega}^{3}-\tau_{\omega}^{2} < M^{5/2} \text{ when } \tau_{\omega}^{2}<\infty \right\} . 
\end{equation*} By Lemma~\ref{lem:ReachHighLevel} in order to bound the probability of the event $\{ \tau_{\omega}^{3}-\tau_{\omega}^{2} > M^{5/2} \}$ under $\Pb^{\mathsf{f}}_{N/2,V}$, using the spatial Markov property together with equation \eqref{eq:ResampleReduction}, we get that
\begin{equation}
\Pb^{\mathsf{f}}_{N/2,V}(\mathcal{D} )=1-o(N^{-1/4}) . 
\end{equation} Hence, it remains to bound the weight of paths $\omega \in \mathcal{D}$ with $\tau_{\omega}^{4}< \infty$. Partition the set $\mathcal{D}$ by
\begin{equation*}
\mathcal{D}_{k} := \left\{ \omega \in \mathcal{D} \, \colon \,  \tau^{4}_{\omega}= \frac{N}{2} - k  \right\} 
\end{equation*} for some $k \in \lsem N/2 \rsem$, and set $k=-\infty$ when $\tau^{4}_{\omega}=\infty$. Using Lemma \ref{lem:GeometricWaitingTimes}, a similar computation as for \eqref{eq:CountingC3Bound} yields that there exists some constant $c_5>0$ such that
\begin{align*}
\big|\Psi^{(0)}_{N/2}\big|^{-1} \sum_{k=1}^{N} |\mathcal{D}_{k}| &\leq \mathbf{P}^{\mathsf{h}}(\exists \, n\in \N  \, \colon \, S^{\mathsf{h}}_n = M  \, | \, S^{\mathsf{h}}_0=K ) \mathbf{P}^{\mathsf{h}}(\exists \, n\in \N  \, \colon \, S^{\mathsf{h}}_n = M \, | \, S^{\mathsf{h}}_0= N^{5\varepsilon/4} \log^{-2}(N)) \\ &\leq c_5 \frac{M N^{\varepsilon} }{K N^{5\varepsilon/4}} = o(N^{-1/4}) . 
\end{align*}
%
%
Using now \eqref{eq:WeightBound} in order to bound the path weight between $\tau_{\omega}^{3}$ and $\tau_{\omega}^{4}$, and Lemma \ref{lem:SupermartingaleComparison1} for the path weight after $\tau_{\omega}^{4}$, this gives the remaining bound on the event $\{ \tau^{4}_{\omega}=\infty\}$.
\end{proof}

In a similar way, we show that the contribution of paths in $\mathcal{C}_2$ is of lower order as well. 

\begin{lemma}\label{lem:PathsSmall2}
There exists some constant $c_1>0$ such that for all $N$ sufficiently large
\begin{equation}\label{eq:PathSmall2}
\Pb^{\mathsf{f}}_{N/2,V}( \omega \in \mathcal{C}_2) = \frac{1}{\mathcal{Z}_{N/2,V}^{\mathsf{f}}(1-q)^{N/2}}\sum_{\omega \in \mathcal{C}_2} W(\omega) \leq c_1 N^{\varepsilon-\frac{1}{2}}\log^{2}(N) . 
\end{equation}
Further, there exists a constant $c_2>0$ such that for all $z \in \lsem N/2 \rsem$, and large enough $N$, 
\begin{equation*}
\Pb^{\mathsf{f}}_{N/2,V}( \omega \in \mathcal{C}_2 \, \wedge \, h_{N/2}(\omega)=z)= \frac{1}{\mathcal{Z}_{N/2,V}^{\mathsf{f}}(1-q)^{N/2}}\sum_{\omega \in \mathcal{C}_2 \, \colon \, h_{N/2}(\omega)=z} W(\omega) \leq  c_2 N^{-\frac{1}{2}-\frac{\varepsilon}{4}}\log^{2}(N) . 
\end{equation*}
\end{lemma}
\begin{proof} 
The first statement follows from the same arguments as \eqref{eq:ResampleReduction} in Lemma~\ref{lem:PathsSmall1}, noting that by Lemma \ref{lem:GeometricWaitingTimes}, for all $i\in  \lsem N\log^{-1}(N) \rsem$
\begin{equation*}
\left| \left\{\omega \in \Psi^{(0)}_{N/2} \, \colon \, \tau^{1}_{\omega} = i  \right\} \right|^{-1}\sum_{j=i}^{N} \big|\mathcal{C}_3^{i,j}\big| = \mathcal{O}(N^{\varepsilon-1/2}\log^{2}(N))  , 
\end{equation*} where we recall the sets $\mathcal{C}_3^{i,j}$ from \eqref{eq:SetsC3ij}.
For the second statement, recall that all paths $\omega \in \mathcal{C}_2$ satisfy $\tau_{\omega}^{3}<N/2-M^{5/2}$ and $\tau_{\omega}^{4}=\infty$. Using the path weight bound \eqref{eq:WeightBound},  we see that for all $i\in \lsem M^{5/2},N/2 \rsem$ and $z \geq M$, and $N$ sufficiently large
\begin{equation}\label{eq:LocalCLTPart1}
\frac{\Pb^{\mathsf{f}}_{N/2,V}(  h_{N/2}(\zeta)=z \, | \, \tau_{\omega}^{3}= N/2-i  \, \wedge \, \tau_{\omega}^{4}=\infty )}{\mathbf{P}^{\mathsf{h}}( S^{\mathsf{h}}_{i}= z-M \,  | \,  S^{\mathsf{h}}_0 = M^{5/2}-M )} \in \left[ 1 - \frac{1}{N}, 1 + \frac{1}{N} \right] .
\end{equation}
Using a local central limit theorem, we get that for all $i \geq M^{5/2}$
\begin{equation}\label{eq:LocalCLTPart2}
c_1M^{-5/4} \mathds{1}_{\{ z \in \lsem c_2 M^{5/4},c_3 M^{5/4} \rsem \} } \leq \mathbf{P}^{\mathsf{h}}( S^{\mathsf{h}}_{N-i}= z-M \,  | \,  S^{\mathsf{h}}_0 = M^{5/2}-M ) \leq c_4 M^{-5/4}
\end{equation} for some positive constants $(c_i)_{i \in \lsem 4 \rsem}$, and by \eqref{eq:PathSmall2}, we conclude.  
\end{proof}

We have now all tools to show Lemma \ref{lem:IntersectionPoint}.

\begin{proof}[Proof of Lemma  \ref{lem:IntersectionPoint}]
Recall that we assume without loss of generality that $N$ is even, as the argument is similar for odd $N$. For $\zeta=(\zeta_1,\zeta_2,\dots,\zeta_N) \in \textup{MP}_N$, we write
\begin{equation}\label{def:PathSplit}
\zeta^{L}:= (\zeta_1,\zeta_2,\dots,\zeta_{N/2-1}) \quad \text{ and } \quad \zeta^{R}:= (\zeta_{N},\zeta_{N-1},\dots,\zeta_{N/2})   
\end{equation}
with $\zeta^{L},\zeta^{R} \in \Lambda^{(0)}_{N/2}$. Our key observation is that we can write for fixed $\tilde{\zeta} \in \textup{MP}_N$
\begin{align*}
\Pb^{\mathsf{c}}_{N,V}\Big( \zeta = \tilde{\zeta} \Big) = \frac{(\mathcal{Z}^{\textsf{f}}_{N/2,V})^2}{\mathcal{Z}^{\textsf{c}}_{N,V}}\Pb^{\textsf{f}}_{N/2,V}\Big( \zeta= \tilde{\zeta}^{L} \Big) \Pb^{\textsf{f}}_{N/2,V}\Big( \zeta=\tilde{\zeta}^{R} \Big)  .
\end{align*} 
Summing over all paths of height $x$ at position $N/2$, we see that
\begin{equation}\label{eq:FromcTof}
\Pb^{\mathsf{c}}_{N,V}\Big( h_{N/2}(\zeta)=x \Big) \sim \Pb^{\mathsf{f}}_{N/2,V}\Big( h_{N/2}(\zeta)=x \Big)^2 . 
\end{equation}
We argue in the following that with probability tending to $1$, the two paths $\zeta^{L},\zeta^{R}$ for some $\zeta$ chosen according to $\Pb^{\mathsf{c}}_{N,V}$ will belong to $\mathcal{C}_1$. To do so, we start with the claim that there exist positive constants $(c_i)_{i \in \lsem 4 \rsem}$ such that for all $x\geq 0$, and $N$ sufficiently large
\begin{equation}\label{eq:LowerMeetingProbability}
\frac{c_1}{\sqrt{N}} \mathds{1}_{x \in \lsem c_2 \sqrt{N},c_3\sqrt{N} \rsem}  \leq \Pb^{\textsf{f}}_{N/2,V}( h_{N/2}(\zeta)=x \text{ and } \zeta \in \mathcal{C}_1 ) \leq \frac{c_4}{\sqrt{N}}  . 
\end{equation}
To show \eqref{eq:LowerMeetingProbability}, first note that by Lemma~\ref{lem:PathsSmall1} and Lemma~\ref{lem:PathsSmall2}, we see that 
\begin{equation*}
\lim_{N \rightarrow \infty}\Pb^{\mathsf{f}}_{N/2,V}( \zeta \in \mathcal{C}_1) = 1 . 
\end{equation*}
The claim \eqref{eq:LowerMeetingProbability} follows by the same arguments as \eqref{eq:LocalCLTPart1} and \eqref{eq:LocalCLTPart2} in the proof of Lemma~\ref{lem:PathsSmall2}, i.e.\ we compare the law of a path $\zeta$, chosen according $\Pb^{\mathsf{f}}_{N,V}( \zeta \in \, \cdot \, | \, \zeta \in \mathcal{C}_1)$, after time $\tau_{\omega}^{1}$ to a simple random walk conditioned to stay above level $M$ until time $N/2$, and apply a local central limit theorem. 
%
Next, consider the product measure $\Pb^{\textsf{f},2}_{N/2,V}:=\Pb^{\textsf{f}}_{N/2,V} \times \Pb^{\textsf{f}}_{N/2,V}$ under which we sample a pair of lattice paths $(\zeta,\zeta^{\prime})$. The lower bound in \eqref{eq:LowerMeetingProbability} ensures that
\begin{equation*}
\Pb^{\textsf{f},2}_{N/2,V}\left(  h_{N/2}(\zeta) = h_{N/2}(\zeta^{\prime}) \right) \geq \frac{c_5}{\sqrt{N}}
\end{equation*} for some constant $c_5>0$, provided that $N$ is sufficiently large. Using  again Lemma~\ref{lem:PathsSmall1} and Lemma~\ref{lem:PathsSmall2}, together with \eqref{eq:LowerMeetingProbability}, we notice that for some constant $c_6>0$
\begin{align*}
\Pb^{\textsf{f},2}_{N/2,V}&\Big( h_{N/2}(\zeta)= h_{N/2}(\zeta^{\prime}) \wedge \big( \zeta \in \mathcal{C}_2 \cup \mathcal{C}_3 \, \vee \,  \zeta^{\prime} \in \mathcal{C}_2 \cup \mathcal{C}_3 \big) \Big)  \\
&\leq  \Pb^{\textsf{f},2}_{N/2,V}\left(  \zeta, \zeta^{\prime} \in \mathcal{C}_2 \cup \mathcal{C}_3 \right) + \frac{c_6}{\sqrt{N}}\sum_{x \geq 0} \Pb^{\textsf{f},2}_{N/2,V}\left( \zeta \in \mathcal{C}_2 \cup \mathcal{C}_3  \wedge h_{N/2}(\zeta)=x\right) , 
\end{align*}
and so both right-hand side terms are $o(N^{-1/2})$. Together with \eqref{eq:FromcTof} and \eqref{eq:LowerMeetingProbability}, we get
\begin{equation}\label{eq:PathsInC1}
\lim_{N \rightarrow \infty} \Pb^{\textsf{c}}_{N/2,V}\left( \zeta^{L},\zeta^{R} \in \mathcal{C}_1 \right) = 1 . 
\end{equation} 
Now \eqref{eq:IntersectionPointStatement} follows from \eqref{eq:FromcTof} and a central limit theorem for the paths $\omega \in \mathcal{C}_1$ after $\tau^{1}_{\omega}$. 
\end{proof}

We apply a similar idea as in Lemma \ref{lem:IntersectionPoint} to obtain the delocalization in Proposition~\ref{pro:DelocalizationHardWallPinning} for general locations $(i_N)$ with $N^{3\varepsilon}\log(N) \ll i_N \ll N/2$. 

\begin{proof}[Proof of Proposition \ref{pro:DelocalizationHardWallPinning}] 
From \eqref{eq:PathsInC1} in the Lemma \ref{lem:IntersectionPoint}, and a central limit theorem for the paths in $\mathcal{C}_1$, we note that for all $\delta>0$, there exists some $c=c(\delta)>0$ such that
\begin{equation}\label{eq:StayInC1}
\liminf_{N \rightarrow \infty} \Pb^{\textsf{c}}_{N,V}(h_{N/2}(\zeta)\geq c \sqrt{N} \text{ and } \zeta^{L},\zeta^{R} \in \mathcal{C}_1) \geq 1- \frac{\delta}{3} . 
\end{equation} 
Let $(m_N)$ be a sequence with $\sqrt{i_N} \gg m_N \gg j_N \gg M$, where we recall that $M=C N^{\varepsilon} \log(N)$ for a suitable constant $C>0$. Further, recall for $\zeta \in \Lambda_{N/2}^{(0)}$ the times
\begin{equation*}
\tau_m(\zeta) := \inf\{ n \geq 0 \, \colon \, h_{n}(\zeta) \geq {m} \} , 
\end{equation*}
and define for all $k,z\in \lsem N/2 \rsem$ the events
\begin{align*}
\mathcal{C}^{k,z}_1 &:= \{ \tau_{m_N}(\zeta) = k  \} \cap \{ h_{N/2}(\zeta) = z\} \cap \{ h_{x}(\zeta)>M \text{ for all } x \geq k \}  \\
\mathcal{C}^{k,z}_2 &:= \{ \tau_{m_N}(\zeta) = k  \} \cap \{ h_{N/2}(\zeta) = z\} \cap \{ h_{x}(\zeta)=M \text{ for some } x \geq k \} .
\end{align*}  We apply the same arguments as in Lemma \ref{lem:ReachHighLevel} and for \eqref{eq:ResampleReduction} in Lemma \ref{lem:PathsSmall1} to see that there exists some $\tilde{c}>0$ such that for all $z\geq c^{\prime}\sqrt{N}$ with some suitable $c^{\prime}=c^{\prime}(c)>0$
\begin{align*}
\frac{1}{\mathcal{Z}_{N/2,V}^{\mathsf{f}}}\sum_{k \in \lsem i_N \rsem }\sum_{\omega \in \mathcal{C}^{k,z}_1} W(\omega) &\geq \tilde{c}N^{-1/2} 
\end{align*}
as well as that
\begin{align*}
\frac{1}{\mathcal{Z}_{N/2,V}^{\mathsf{f}}}\sum_{k \in \lsem i_N \rsem }\sum_{\omega \in \mathcal{C}^{k,z}_2} W(\omega) &=o(N^{-1/2}) . 
\end{align*}
In particular, with \eqref{eq:StayInC1} and summing over all $z\geq c^{\prime}\sqrt{N}$, and using Lemma \ref{lem:WeightToPolymer}, we get
\begin{equation}\label{eq:PathsInCzk}
\liminf_{N \rightarrow \infty}\Pb^{\textsf{c}}_{N,V}\Big(\zeta^{L} \in \mathcal{C}^{k,z}_1 \text{ for some } k \leq i_N \text{ and } z \geq c^{\prime}\sqrt{N}\Big) \geq 1- \frac{2\delta}{3} . 
\end{equation}
Together with the weight bound \eqref{eq:WeightBound} for paths in $\mathcal{C}^{k,z}_1$ after position $k$, we see that 
\begin{equation*}
\Pb^{\textup{\textsf{c}}}_{N,V}\left( A_N^{i,j} \right)  \geq \left( 1 - \frac{1}{N}\right)\left( \sum_{z \geq c^{\prime}\sqrt{N}}\sum_{k \in \lsem i_N \rsem} \mathbf{P}^{(N/2)}_{m_N}\big( S^{\mathsf{h}}_{i_N-k } \geq j \, | \, S^{\mathsf{h}}_{\frac{N}{2}-k} = z \big) \Pb^{\textsf{c}}_{N,V}\big( \mathcal{C}^{k,z}_1\big) \right) \geq 1- \delta 
\end{equation*} for all $N$ large enough, where we recall the measure $\mathbf{P}^{(N/2)}_{x}$ of the lazy simple random started from $x$ conditioned to stay non-negative until time $N/2$. Here, we use \eqref{eq:WeightBound} and Lemma~\ref{lem:SpatialMarkov} for the first step,  and \eqref{eq:PathsInCzk}  together Lemma~\ref{lem:GeometricWaitingTimes} for the second inequality. Since $\delta>0$ was arbitrary, we get \eqref{eq:DelocConstraintMidPoint}. The second claim \eqref{eq:DelocConstraintGeneral} follows analogously.  
\end{proof}

\begin{remark}\label{rem:PartitionFunctionRestricted}
 We showed that when $u,v \in (-1,1)$ and $\varepsilon<\frac{1}{3}$, the partition function of the free and the constraint random polymer is of order $4^{N(1+o(1))}$. Indeed, recalling the definition of the free energy in \eqref{def:FreeEnergyPartition}, this yields that the polymer measure associated with the open WASEP in the maximal current phase is delocalized. Let us remark that for $q \in (0,1)$, it was shown in \cite{S:DensityProfilePASEP} that the partition function of the constraint random polymer is of order $4^{N}N^{-3/2}$, and hence the corresponding random polymer is delocalized. Our approach yields delocalization for more general values of $q$ depending on~$N$. 
\end{remark}

\subsection{From delocalization to approximation by a product measure}\label{sec:ComparisionMeasures}

In this section, we prove Theorem \ref{thm:MaxCurrent} and Theorem \ref{thm:MaxCurrentWASEP} on approximating the stationary distribution of the open ASEP and open WASEP in the maximal current phase. 
In the following, fix for each $N$ an interval $I = I(N)=\lsem a,b\rsem$ with $a=a(N)$ and $b=b(N)$, and recall that $P_x$ denotes the law of a lazy simple random walk trajectory when starting from position $x$.

\begin{lemma}\label{lem:ASEPLocalCLTSRW} Consider the open ASEP for some $q\in (0,1)$, and parameters $u,v \in (-1,1)$. Let $I$ satisfy $\min(a,N-b) \gg \max(|I|,\log^2(N))$. 
Then 
\begin{equation}
\lim_{N \rightarrow \infty} \TV{  \Pb^{\textup{\textsf{c}}}_{N,V}\left( (h_{x}(\zeta)-h_{a}(\zeta))_{x \in I} \in \, \cdot \, \right)  -   P_{0}((S_{x})_{x \in \lsem b-a \rsem } \in \, \cdot \,  ) }  = 0 . 
\end{equation}
\end{lemma}
\begin{proof}
We consider in the following only the case $|I| \ll a\leq N/3$ as the remaining cases are similar. Let 
\begin{equation*}
\mathcal{D}_{k,z} := \left\{ h_a(\zeta)=k \right\} \cap \left\{ h_{\frac{N}{2}}(\zeta)=z \right\}   . 
\end{equation*} for all $k, z \in \lsem N/2 \rsem$. Recall from the proof of Lemma \ref{lem:EquivalenceASEPhtransform} that the law $\Pb^{\textup{\textsf{c}}}_{N,V}$ is equivalent to the law of a lazy simple random walk conditioned to stay non-negative and to return to the origin after $N$ steps. Using Lemma \ref{lem:EquivalenceASEPhtransform}, recalling the weights $W$ from \eqref{def:WeightsIndividually}, we get
\begin{equation}\label{eq:FirstSplit}
\lim_{N \rightarrow \infty}\TV{ \Pb^{\textup{\textsf{c}}}_{N,V}\left( (h_{x}(\zeta))_{x \in I} \in \, \cdot \, | \, \mathcal{D}_{k,z}  \right)  -  \mathbf{P}^{(N-a)}_{k}( (S_{x})_{x \in \lsem b-a \rsem } \in \, \cdot \,  | \, S_{N/2-a}=z )  } = 0 , 
\end{equation}
provided that $k,z \gg \log(N)$. As $|I| \ll a$, a similar argument as for Lemma \ref{lem:GeometricWaitingTimes} yields 
\begin{equation}\label{eq:SecondSplit}
\lim_{N \rightarrow \infty}\TV{ \mathbf{P}^{(N-a)}_{k}( (S_{x})_{x \in \lsem b-a \rsem } \in \, \cdot \,  | \, S_{N/2-a}=z ) - P_{k}((S_{x})_{x \in \lsem b-a \rsem } \in \, \cdot \,  ) } = 0 
\end{equation}
whenever $k \gg \max(\sqrt{|I|},\log(N))$ and $z$ is of order $\sqrt{N}$.  Hence, using \eqref{eq:FirstSplit} and \eqref{eq:SecondSplit} with the triangle inequality, and summing over all suitable values of $k$ and $z$, we conclude.
\end{proof}
\begin{lemma}\label{lem:WASEPLocalCLTSRW}
Consider the open WASEP for some $q\in (0,1)$ with $\varepsilon>0$, and $u,v \in (-1,1)$. Assume that $I$ satisfies the assumptions in Theorem \ref{thm:MaxCurrentWASEP}. 
Then we have that 
\begin{equation}
\lim_{N \rightarrow \infty} \TV{  \Pb^{\textup{\textsf{c}}}_{N,V}\left( (h_{x}(\zeta)-h_{a}(\zeta))_{x \in I} \in \, \cdot \, \right)  -   P_{0}((S_{x})_{x \in I } \in \, \cdot \,  ) }  = 0 . 
\end{equation}
\end{lemma}
\begin{proof}
As for Lemma \ref{lem:ASEPLocalCLTSRW}, we consider only the case $a\leq N/3$ as the remaining cases are similar. 
Note that by Proposition~\ref{pro:DelocalizationHalfspaceMartingale} and Proposition~\ref{pro:DelocalizationHardWallPinning}, we find $(k_N)$ and $(m_N)$ with 
\begin{equation}\label{eq:SequenceAssumptions}
\min(a_N,N-b) \gg \sqrt{k_N} \gg \sqrt{m_N} \gg |I|
\end{equation} 
such that for every $\delta>0$, the events
\begin{equation*}
\tilde{\mathcal{D}}_{\ell,z} := \left\{ h_{a}(\zeta) = \ell \right\} \cap \left\{  h_{x}(\zeta)\geq m_N \, \forall x\in I \right\} \wedge \left\{  h_{N/2}(\zeta) =z \right\}  
\end{equation*}
for $z \in \lsem c_1\sqrt{N},c_2 \sqrt{N} \rsem$ and constants $c_1,c_2>0$, and $\ell \geq k_N$ satisfy for all $N$ large enough 
\begin{equation*}
\Pb^{\textsf{c}}_{N,V}\left( \tilde{\mathcal{D}}_{\ell,z} \text{ holds for some } z \in \lsem c_1\sqrt{N},c_2 \sqrt{N} \rsem \text{ and } \ell \geq k_N \right) \geq 1 - \delta . 
\end{equation*}
 By Lemma~\ref{lem:EquivalenceASEPhtransform} and the choice of the weights $W$ in \eqref{def:WeightsIndividually}, we see that for all $\ell \geq k_N$, and uniformly in $z \in \lsem c_1\sqrt{N},c_2 \sqrt{N} \rsem$, 
\begin{equation*}
\lim_{N \rightarrow \infty}\TV{ \Pb^{\textup{\textsf{c}}}_{N,V}\left( (h_{x}(\zeta))_{x \in I} \in \, \cdot \, | \, \tilde{\mathcal{D}}_{\ell,z}  \right)  -  \mathbf{P}^{(N/2-a)}_{\ell}( (S_{x})_{x \in \lsem b-a \rsem } \in \, \cdot \,  | \, S_{N/2-a}=z )  } = 0 .
\end{equation*}  By a similar comparison to the lazy simple random walk as in Lemma \ref{lem:ASEPLocalCLTSRW}, using the assumptions   \eqref{eq:SequenceAssumptions} and a local central limit theorem in order to remove the conditioning on the event $\{ S_{N/2-a}= z\}$, we conclude.
\end{proof}
Recall that $\P_N$ denotes the uniform measure on the space of bi-colored Motzkin paths $\Psi_{N}$, and that $(h_{i}(\omega))_{i \in \lsem 0,N \rsem}$ is the height function of a path $\omega=(\omega_1,\omega_2,\dots,\omega_N) \in \Psi_{N} \subseteq \mathcal{A}^{N}$, with $\mathcal{A}$ from \eqref{def:Acal}. 
Recall that $\mu_{I}$ is the stationary measure of the exclusion process  projected to  $I$.
We use the above results to compare $\mu_I$ to the uniform distribution on $\{0,1\}^{|I|}$.
\begin{proof}[Proof of Theorem \ref{thm:MaxCurrent} and Theorem \ref{thm:MaxCurrentWASEP}]
For an interval $I=\lsem
 a, b \rsem$ for $a=a_N$ and $b=b_N$, consider a lattice path $\zeta \in \Lambda_{N}$ for $\Lambda_N$ from \eqref{def:LatticePaths}, sampled according to $\Pb^{\textsf{c}}_{N,V}$, and restricted to the interval~$I$. Given $\zeta$, we obtain a sample $\omega$ according to $\P_N$ by assigning to each horizontal move a color proportionally to the weights $W$ from \eqref{def:WeightsIndividually}. We claim that under the assumptions on the interval $I$ in Theorem \ref{thm:MaxCurrent} and Theorem \ref{thm:MaxCurrentWASEP}, 
 \begin{equation}\label{eq:LatticePathToBiColored}
 \lim_{N \rightarrow \infty} \TV{ \P_N( (\omega_x)_{x \in I} \in \, \cdot \, ) - \textup{Unif}(\mathcal{A}^{|I|}) } = 0, 
 \end{equation}
 where $\textup{Unif}(\mathcal{A}^{|I|})$ denotes the uniform distribution on $\mathcal{A}^{|I|}$. To see this for the open ASEP, we use Lemma~\ref{lem:ASEPLocalCLTSRW} in order to bound the total variation distance between the law $\P_N$ projected onto the space $\textup{MP}_N$ of  Motzkin paths of length $N$ and a simple random walk trajectory, and Lemma \ref{lem:EquivalenceASEPhtransform} together with the definition of the weights $W$ in order to estimate the probability of assigning a given pattern of colors to the horizontal steps. Similarly, the statement \eqref{eq:LatticePathToBiColored} follows for the open WASEP by combining Lemma~\ref{lem:WASEPLocalCLTSRW} together with Proposition~\ref{pro:DelocalizationHalfspaceMartingale} and Proposition~\ref{pro:DelocalizationHardWallPinning}. Next, for all $N\in \N$, we fix a subset $A = A(N) \subseteq \{0,1\}^{|I|}$ of particle configurations. We claim that for every $\delta>0$, we can find a constant $C=C(\delta)>0$, and a family of subsets $(A_{\delta})$ such that $A_{\delta} \subseteq A$ and
 \begin{equation*}
 \limsup_{N \rightarrow \infty }\left| \frac{|A|}{2^{b-a}} - \frac{|A_{\delta}|}{2^{b-a}} \right| \leq \delta , 
 \end{equation*} where for all configurations $\eta \in A$, and all $x \in \lsem a, b \rsem$, we have that
 \begin{equation}\label{eq:TransversalBounds}
 \sum_{i=a}^{x} \left(\eta(x) - \frac{1}{2}\right) \leq C \sqrt{b-a} . 
 \end{equation} This follows as \eqref{eq:TransversalBounds} holds for $C>0$ sufficiently large with probability at least $1-\delta/2$ under the uniform distribution on $\{0,1\}^{b-a}$. Let $A^{\prime}_{\delta} \subseteq \Psi_N$ be the set of bi-colored Motzkin paths which agree with some element of $A_{\delta}$ on the interval $I$, i.e.\
\begin{equation*}
A^{\prime}_{\delta} := \left\{ \omega \in \Psi_N \, \colon \, (\omega_x)_{x \in I} = (\omega^{\prime}_x)_{x \in I} \text{ for some }\omega^{\prime} \in \mathcal{C}_\eta \text{ with } \eta \in A_{\delta} \right\} ,
\end{equation*} where we recall the set $\mathcal{C}_{\eta}$ from \eqref{def:SetOfPathsForConfiguration}. By our assumptions $ |I| \ll \min(a,N-b)$, the height function for a uniformly sampled paths in $\Psi_N$ is at position $a$ with probability tending to $1$ of order $\sqrt{a}$. Thus, we see that for all sets $A$, and all choices of $\delta>0$
\begin{equation}\label{eq:SubsetChoice2}
\limsup_{N \rightarrow \infty} \left| \frac{|A|}{2^{|I|}} - \frac{|A_{\delta}^{\prime}|}{|\Psi_N|} \right| \leq \limsup_{N \rightarrow \infty} \left| \frac{|A_{\delta}|}{2^{|I|}} - \frac{|A_{\delta}^{\prime}|}{|\Psi_N|} \right| + \delta \leq 2 \delta . 
\end{equation}
Hence, by Lemma~\ref{lem:Paths} for the first step, and \eqref{eq:LatticePathToBiColored} together with \eqref{eq:SubsetChoice2} for the second step,
 \begin{equation*}
\limsup_{N \rightarrow \infty}\left| \mu_I(A) - \frac{|A|}{2^{|I|}} \right| = \limsup_{N \rightarrow \infty}\left| \P_N(A^{\prime}_{\delta}) - \frac{|A^{\prime}_{\delta}|}{|\Psi_N|} \right| + \limsup_{N \rightarrow \infty}\left| \frac{|A|}{2^{|I|}} - \frac{|A^{\prime}_{\delta}|}{|\Psi_N|} \right| \leq 3\delta .
 \end{equation*}
 Since the set $A$ and $\delta>0$ were arbitrary, this finishes the proof.
\end{proof}

\section{Approximation in the fan region of the high and low density phase}\label{sec:LocalizationPhase}

In this section, we approximate the stationary distribution of the open ASEP in the fan region of the high and low density phase by product measures. We establish localization in the associated polymer model and study the structure of the respective regeneration times. Thereafter, we couple the system to an infinite random polymer model corresponding to a stationary system; see also \cite{CGZ:PolymerRenewal,M:StationaryASEP,SS:MEC} for a similar approach in related models.

\subsection{Expected return times for a localized polymer}\label{sec:ExpectedReturnTimes}

For $N\in \N$, recall the measures $\Pb^{\textup{\textsf{f}}}_{N,V}$ and $\Pb^{\textup{\textsf{c}}}_{N,V}$ from  \eqref{def:FreePolymerMeasure} and \eqref{def:ConstraintPolymerMeasure}, respectively, with $V$ from \eqref{eq:PolymerFunctionMotzkin}. Moreover, for all $x \in \Z$, let $\Pb^{\textup{\textsf{f},N}}_{x}$ be defined as in \eqref{def:ShiftedMeasure}, and recall for all $m \in \N$ from \eqref{def:ReturnTimes} the quantities
\begin{equation}\label{def:ReturnTimes2}
\tau_m = \inf\left\{y \in \N \, \colon \, h_y(\zeta)=m \right\} .
\end{equation}
In the following, we argue that for the open ASEP in the fan region of the high and low density, the associate polymer has for every $m\geq 0$ a return time to level $m$ whose finite moments are bounded uniformly as $N \rightarrow \infty$. This is in contrast to the maximal current phase, where we saw in Lemma \ref{lem:EquivalenceASEPhtransform} that the random polymer is delocalized.
\begin{proposition}\label{pro:RegenerationPoint} 
Let $u,v$ be such that $u>\max(1,v)$ and $uv<1$. Then for every $x \in \N$, all $m\geq -x$ and $k\in\N$, there exists some $C,N_0 \in \N$, depending only on $x,m,k$, such that
\begin{equation}
\mathbf{E}^{\textup{\textsf{f},N}}_{x}[\min(\tau_m , N)^{k}] \leq C
\end{equation} for all $N\geq N_0$. The same statement holds under the expectation $\mathbf{E}^{\textup{\textsf{c}}}_{N,V}$ corresponding to the measure $\mathbf{P}^{\textup{\textsf{c}}}_{N,V}$, and in the high density phase, where $v>\max(1,u)$ and $uv<1$.
\end{proposition}
\begin{proof}
We consider only the case of $u>\max(1,v)$, and the measure $\mathbf{P}^{\textup{\textsf{f},N}}_{x}$, as the other cases are similar. 
For $q\in (0,1)$, equation (65) in \cite{BECE:ExactSolutionsPASEP} states that the partition function $Z_N$ from Lemma \ref{lem:Weights} satisfies
\begin{equation}
Z_N = \frac{(u^{-2};q)_{\infty}^2}{(uv,u/v;q)_{\infty}}\left( \frac{2+u+u^{-1}}{1-q}\right)^{N} (1 + o(1)) , 
\end{equation} where we recall the Pochhammer symbol from \eqref{def:Pochhammer};  see also \cite{USW:PASEPcurrent} for a more detailed derivation of the above formula using Askey--Wilson polynomials.
In combination with Lemma \ref{lem:Paths} and Lemma  \ref{lem:WeightToPolymer}, this yields that the partition function $\mathcal{Z}^{\mathsf{c}}_{N,V}$ for the stationary distribution of the open ASEP in the low density phase satisfies
\begin{equation}\label{eq:PartitionFunctionAsymptoticsHighLow}
\mathcal{Z}^{\mathsf{c}}_{N,V} = \frac{(u^{-2};q)_{\infty}^2}{(uv,u/v;q)_{\infty}}\left( 2+u+u^{-1}\right)^{N}(1 + o(1)) .
\end{equation}  Let $M$ be such that for all $m \geq M$
\begin{equation}\label{eq:ParameterHighLowRelation}
 1 + \max( -uv, u,v )q^{m} <  \frac{2+u^{-1} + u}{4}  . 
\end{equation} Recalling $W$ from \eqref{def:WeightsIndividually}, and choosing $M$ according to \eqref{eq:ParameterHighLowRelation}, together with the bound $\mathcal{Z}^{\mathsf{c}}_{N,V} \leq \mathcal{Z}^{\mathsf{f}}_{N,V}$, we see that
\begin{equation}
\sum_{\omega \in \Lambda_0} W_M(\omega) \leq c^{N}  \mathcal{Z}^{\mathsf{f}}_{N,V}
\end{equation} for some constant $c \in (0,1)$ and all $N$  sufficiently large. As a consequence, we get that the return time to level $M$ has exponential tails, i.e.\ there exists  $c^{\prime},t_0>0$ such that
\begin{equation}\label{eq:ReturnToM}
\mathbf{P}^{\textup{\textsf{f},N}}_{M}( \tau_0 > t ) \leq \exp(-c^{\prime}t)
\end{equation} for all $t \in \lsem t_0, N \rsem$.
Using the spatial Markov property in Lemma~\ref{lem:SpatialMarkov}, and a standard argument of using independent geometric tries, we get that for all $t \in \lsem N \rsem$ and $n \geq N$
\begin{equation}\label{eq:ReturnToMorZero}
\max_{x \in \lsem M \rsem } \mathbf{P}^{\textup{\textsf{f},n}}_{x}( \min(\tau_{-x},\tau_{M-x}) > t )  \leq \exp(-c_1 t)
\end{equation} 
\begin{equation}\label{eq:PositiveHit}
\min_{x \in \lsem M \rsem } \mathbf{P}^{\textup{\textsf{f},n}}_{x}( \tau_{-x} < \tau_{M-x} < n) \geq c_2
\end{equation}
for some constants $c_1,c_2,t_0^{\prime}>0$, and all $t \in \lsem  t_0^{\prime},n \rsem$. Decomposing  every trajectory according to its intersections with level $M$, and using the spatial Markov property from Lemma \ref{lem:SpatialMarkov}, we combine  \eqref{eq:ReturnToM}, \eqref{eq:ReturnToMorZero} and \eqref{eq:PositiveHit} to conclude.
\end{proof}

\subsection{Regeneration structure in the localization phase}\label{sec:RegenerationLocalized}

Using Proposition \ref{pro:RegenerationPoint}, we construct a regenerative process, as well as a bi-infinite stationary process $\mathbf{P}_{\textup{stat}}$ related to the polymer measures $\Pb^{\textup{\textsf{f}}}_{N,V}$ and $\Pb^{\textup{\textsf{c}}}_{N,V}$. To do so, we require the \textbf{lazy h-transformed simple random walk} $(\bar{S}^{\mathsf{h}}_n)_{n \geq 0}$, i.e.\ the Markov chain on $\N_0$ with transition probabilities
\begin{equation}\label{def:htransformed}
\bar{p}_{\textup{h}}(x,y) = \begin{cases}
\frac{x+2}{4(x+1)}& \text{ if } y=x+1\\
\frac{x}{4(x+1)}& \text{ if } y=x-1\\
\frac{1}{2}& \text{ if } y=x\\
0 & \text{ otherwise,}
\end{cases}
\end{equation}  and increments $\bar{X}^{\mathsf{h}}_i=\bar{S}^{\mathsf{h}}_i-\bar{S}^{\mathsf{h}}_{i-1}$. 
Let $\bar{\Pb}_N$ be the corresponding law on the space of trajectories $\Lambda_N^{(0)}$ of length $N$, defined in~\eqref{def:LatticePathsNonNegative}.
 For all $N\in \N$, we define the polymer measure $\bar{\mathbf{P}}^{\mathsf{f}}_{N,V}$ by
\begin{equation}\label{def:DerivativehTransform}
\frac{\dif \bar{\mathbf{P}}^{\mathsf{f}}_{N,V}}{\dif \bar{\Pb}_N} = \frac{1}{\bar{\mathcal{Z}}^{\mathsf{f}}_{N}} \exp\left( \sum_{i=1}^{N} V(\bar{S}^{\mathsf{h}}_i,\bar{X}^{\mathsf{h}}_i) \right), 
\end{equation} for the function $V$ from \eqref{eq:PolymerFunctionMotzkin}, and a normalization constant $\bar{\mathcal{Z}}^{\mathsf{f}}_{N}$. 
It is a classical result that for all $M\in \N$, the law of a (lazy)  simple random conditioned to stay non-negative until time $N$ converges on any fixed finite interval to the law of the (lazy) $h$-transformed simple random walk as $T\rightarrow \infty$; see \cite{BD:hTransform}. We have the following consequence of this observation.
\begin{lemma}\label{lem:ConvergenceBeginningEnd}
Assume that either $u>\max(1,v)$ or $v>\max(1,u)$ holds. Let $A,B \subseteq \Lambda_{m}^{(0)}$ for some fixed $m \in \N$. Then we have that
\begin{align*}\begin{split}
\lim_{N \rightarrow \infty} \Big|\mathbf{P}^{\mathsf{c}}_{N,V}\left( (h_{x}(\omega))_{x \in \lsem 0,m \rsem} \in A \, \wedge \,  (h_{N-x}(\omega))_{x \in \lsem 0,m \rsem} \in B \right) - \bar{\mathbf{P}}^{\mathsf{f}}_{m,V}\left(  A \right)\bar{\mathbf{P}}^{\mathsf{f}}_{m,V}\left(  B \right) \Big| = 0 . \end{split}
\end{align*} 
\end{lemma}
\begin{proof}
Let $\bar{P}_N$ be the law on the space $\Psi_N$ given by a lazy simple random walk conditioned to stay non-negative, and to return to $0$ after $N$ steps. It is a well-known result -- see for example \cite{BD:hTransform} -- that for any fixed $m\in \N$, and any subset $A \subseteq \Lambda_{m}^{(0)}$
\begin{equation*}
\lim_{N \rightarrow \infty} |\bar{P}_N( (S_x)_{x \in \lsem 0,m \rsem} \in A) - \bar{\mathbf{P}}_m(A) | = 0 .
\end{equation*}  Together with a local central limit theorem for the lazy simple random walk, we see that for any pair of sets $A,B \subseteq \Lambda_{m}^{(0)}$
\begin{equation}\label{eq:ConvergenceOfWalks}
\lim_{N \rightarrow \infty} \left|\bar{P}_N((S^{\mathsf{h}}_x)_{x\in \lsem 0,m\rsem} \in  A \text{ and } (S^{\mathsf{h}}_{N-x})_{x\in \lsem 0,m\rsem} \in  B) - \bar{\mathbf{P}}_m(A) \bar{\mathbf{P}}_m(B) \right| = 0 . 
\end{equation}
Recalling the construction of the measures $\mathbf{P}^{\mathsf{c}}_{N,V}$ and $\bar{\mathbf{P}}^{\mathsf{f}}_{m,V}$ in \eqref{def:ConstraintPolymerMeasure} and \eqref{def:DerivativehTransform}, we observe that both measures are defined with respect to the same Radon--Nikodym derivative. As $A$ and $B$ only depend on finitely many coordinates, we conclude by \eqref{eq:ConvergenceOfWalks}. 
\end{proof}

Using the polymer measures $(\bar{\mathbf{P}}^{\mathsf{f}}_{N,V})_{N \in \N}$, we construct a bi--infinite and shift invariant measure $\mathbf{P}_{\textup{stat}}$ as follows. Recall the return times $\tau_0$ from \eqref{def:ReturnTimes} and define
\begin{equation}\label{def:GeneralPathSpace}
\Psi_{\ast} := \bigcup_{n \in \N} \left\{ \omega \in \Psi_{n} \, \colon \,  h_{x}(\omega)>0 \text{ for all } x \in \lsem n-1 \rsem \right\},
\end{equation} where we recall $\Psi_n$ from \eqref{def:BiColoredMotzkin}. 
Intuitively, $\Psi_{\ast}$ corresponds to set of all lattice paths which return to $0$ after $n$ steps for some $n\in \N$, and are positive for all $x \in \lsem n-1\rsem$.
Assume that either $u>\max(1,v)$ or $v>\max(1,u)$ holds. Then combining Proposition~\ref{pro:RegenerationPoint} and Lemma~\ref{lem:ConvergenceBeginningEnd}, there exists a unique measure $\Q$ on $\Psi_{\ast}$ such that 
\begin{equation}\label{eq:ConstructionOfQ}
\Q(\xi)  :=\lim_{N \rightarrow \infty}\mathbf{P}^{\textup{\textsf{c}}}_{N,V}\left( (h_x(\omega))_{x \in \lsem 0,\tau_0 \wedge N \rsem} = \xi \right) = \lim_{N \rightarrow \infty}\bar{\mathbf{P}}^{\mathsf{f}}_{N,V}\left( (h_x(\omega))_{x \in \lsem 0,\tau_0 \wedge N\rsem} = \xi \right)
\end{equation} for all $\xi \in \Psi_{\ast}$. In words, the measure $\Q$ corresponds to the law of the path under the measure $\mathbf{P}^{\textup{\textsf{c}}}_{N,V}$ until the first return to the x-axis when taking $N$ going to infinity. 
%
%
%
 Furthermore, note that by Proposition \ref{pro:RegenerationPoint} and Lemma \ref{lem:ConvergenceBeginningEnd}, $\Q$ has exponential tails, i.e.\
\begin{equation}\label{eq:ExponentialTailsIncrements}
\Q(|\xi| > t ) \leq \exp(-c t)
\end{equation} for some constant $c>0$ and all $t>0$ sufficiently large, where $|\xi|$ denotes the length of the path $\xi$.
Note that by Kolmogorov's extension theorem, we can extend the measure $\Q$ to a bi--infinite measure $\bar{\Q}$ on the space of bi--infinite lazy simple random walk trajectories going through the origin by sampling a bi-infinite i.i.d.\ sequence according to $\Q$. To obtain the stationary process $\mathbf{P}_{\textup{stat}}$, we use a standard construction for stationary point processes with independent increments according to $\Q$: consider a Markov chain $(X_t)_{t \geq 0}$ on $\N_0$ as follows. When $X_t=0$ for some $t\geq 0$,  let $X_{t+1}=|\xi|$, where $\xi \sim \Q$. Otherwise, let $X_{t+1}=X_t-1$. From  \eqref{eq:ExponentialTailsIncrements}, we get that the Markov chain $(X_t)_{t \geq 0}$ is positive recurrent, and has a unique stationary distribution $\pi$. The measure $\mathbf{P}_{\textup{stat}}$ is now defined on the space 
\begin{equation*}
\bar{\Lambda} := \bigcup_{n \in \Z} \Big\{ (\dots,v_{-1},v_0,v_1,\dots)\in (\Z \times \N_0)^{\Z} \, \colon \, v_0 = (n,0) \, \wedge \,  v_i-v_{i-1} \in \{ (1,1),(1,0),(1,-1)\} \, \forall i\in \Z \Big\}
\end{equation*}
equipped with the sigma-algebra generated by all cylinder functions, and where
\begin{equation}
\mathbf{P}_{\textup{stat}}(\zeta \in \, \cdot \, ) := \bar{\Q}( \theta_s \zeta \in \, \cdot \,  ) . 
\end{equation}
Here, $s$ is chosen according to $\pi$, and  $\theta_s$ denotes the horizontal shift operator on $\bar{\Lambda}$ by $s$. In other words, for a sample according to $\mathbf{P}_{\textup{stat}}$, we first choose a horizontal starting point $s$ according to $\pi$, and then sample a bi-infinite lazy simple random walk trajectory starting from $(s,0)$ with   increments according to $\Q$.

\subsection{Coupling of regenerative processes}\label{sec:CouplingRegenerative}

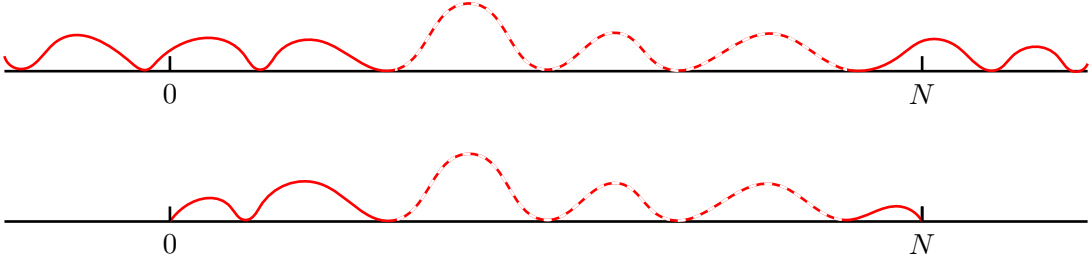
\begin{figure}
    \centering
\begin{tikzpicture}[scale=1]
\draw [line width=1pt] (-2.2,0) to (12.2,0); 	
\draw [line width=1pt] (-2.2,-2) to (12.2,-2); 	
\draw [line width=1pt] (0,0) to  (0,0.2); 	
\draw [line width=1pt] (10,0) to  (10,0.2); 	
		
\draw [line width=1pt] (10,0.2-2) to  (10,-2); 		
 
 \node (H4) at (0,-2.3) {$0$};
 \node (H4) at (10,-2.3) {$N$};  

 \node (H4) at (0,-0.3) {$0$};
 \node (H4) at (10,-0.3) {$N$};
  
   \draw[red,line width=1 pt] (-2.2,0.2) to[curve through={(-2,0.03)..(-1.5,0.4)..(-0.5,0.1)..(-0.3,0.02)..(-0.2,0.1)..(1,0.2)..(1.2,0.02)..(1.4,0.2)..(2,0.4)..(3,0.02)..(4,0.9)..(5,0.02)..(6,0.5)..(6.5,0.08)..(8,0.5)..(9,0.02)..(9.7,0.2)..(10,0.4)..(10.5,0.3)..(11,0.02)..(11.2,0.2)..(11.8,0.2)..(12,0)}] (12.2,0.1);

   \draw[red,line width=1 pt] (0,-2) to[curve through={(0.2,0.2-2)..(0.8,0.2-2)..(1,0.02-2)..(1.2,0.2-2)..(2,0.5-2)..(3,0.02-2)..(4,0.9-2)..(5,0.02-2)..(6,0.5-2)..(6.5,0.08-2)..(8,0.5-2)..(9,0.02-2)..(9.7,0.2-2)}] (10,0-2);	 
 
 \begin{scope}
\clip(3,-2) rectangle (9,2);

\draw[white,line width=1 pt,dashed] (-2.2,0.2) to[curve through={(-2,0.03)..(-1.5,0.4)..(-0.5,0.1)..(-0.3,0)..(-0.2,0.1)..(1,0.2)..(1.2,0.02)..(1.4,0.2)..(2,0.4)..(3,0.02)..(4,0.9)..(5,0.02)..(6,0.5)..(6.5,0.08)..(8,0.5)..(9,0.02)..(9.7,0.2)..(10,0.4)..(10.5,0.3)..(11,0.02)..(11.2,0.2)..(11.8,0.2)..(12,0)}] (12.2,0.1);	

  \draw[white,line width=1 pt,dashed] (0,-2) to[curve through={(0.2,0.2-2)..(0.8,0.2-2)..(1,0.02-2)..(1.2,0.2-2)..(2,0.5-2)..(3,0.02-2)..(4,0.9-2)..(5,0.02-2)..(6,0.5-2)..(6.5,0.08-2)..(8,0.5-2)..(9,0.02-2)..(9.7,0.2-2)}] (10,0-2);	
 \end{scope}

\draw [line width=1pt] (0,-2) to  (0,0.2-2);  
\draw [line width=1pt] (10,0.2-2) to  (10,-2); 		 
 

\end{tikzpicture}
    \caption{Coupling of the stationary regeneration process from $\mathbf{P}_{\textup{stat}}$ with a sample according to the constraint random polymer measure $\mathbf{P}^{\textup{\textsf{c}}}_{N,V}$. Note the polymers are coupled so that they agree on the dashed parts of the lines.}
    \label{fig:Regeneration}
\end{figure}

In this section, we establish a coupling between $\mathbf{P}_{\textup{stat}}$ and the random polymer measure $\mathbf{P}^{\textup{\textsf{c}}}_{N,V}$; see also Figure \ref{fig:Regeneration} for a visualization. To do so, we start with the following lemma collecting some basic observations about the measure $\mathbf{P}_{\textup{stat}}$ on the space $\bar{\Lambda}$.

\begin{lemma}\label{lem:StationaryProcess}
Assume that either $u>\max(1,v)$ or $v>\max(1,u)$ holds. Then the process $\mathbf{P}_{\textup{stat}}$ is invariant under spatial shifts, i.e.\ for all $x\in \Z$ and all measurable sets $A$ on $\bar{\Lambda}$, 
\begin{equation}\label{eq:ShiftInvariantStationary}
\mathbf{P}_{\textup{stat}}(\zeta \in A) = \mathbf{P}_{\textup{stat}}( \theta_x\zeta \in A ) .
\end{equation} Moreover, for any pair of positions $a<b$
\begin{equation}\label{eq:PositivePinningProbability}
\mathbf{P}_{\textup{stat}}(h_{a}(\zeta)=0 \text{ and } h_{b}(\zeta)=0) >0,
\end{equation}
and  for any $A \subseteq \Psi_{n}$ with $n=b-a$,
\begin{equation}\label{eq:ConditionedLawStationary}
\mathbf{P}_{\textup{stat}}\big(  
(h_{x}(\zeta))_{x \in \lsem a,b\rsem} \in A \, \big| \,  h_{a}(\zeta) =0
\text{ and } h_{b}(\zeta)=0\big) = \mathbf{P}^{\textup{\textsf{c}}}_{n,V}(A) . 
\end{equation}
\end{lemma}
\begin{proof} 
For the shift invariance property \eqref{eq:ShiftInvariantStationary}, note that applying the shift operator $\theta_x$ to a configuration $\zeta$ according to $\mathbf{P}_{\textup{stat}}$ corresponds to a shift in the underlying Markov chain $(X_t)_{t \geq 0}$ by $x$ used in the construction of $\mathbf{P}_{\textup{stat}}$. The claim follows as the initial shift is chosen according to the stationary distribution $\pi$ of $(X_t)_{t \geq 0}$. The second statement \eqref{eq:PositivePinningProbability} is immediate from \eqref{eq:ShiftInvariantStationary} and the facts that $\pi(0)>0$ and $\Q( |\xi|=1)>0$. For the last claim, note that the measure $(\bar{\mathbf{P}}^{\mathsf{f}}_{N,V})_{N \in \N}$ satisfy the spatial Markov property; see also  Lemma~\ref{lem:SpatialMarkov}. As the underlying $h$-transformed lazy simple random walk for $(\bar{\mathbf{P}}^{\mathsf{f}}_{N,V})_{N \in \N}$ is a time-homogeneous Markov chain, we can extend $(\bar{\mathbf{P}}^{\mathsf{f}}_{N,V})_{N \in \N}$ 
to a measure $\bar{\mathbf{P}}^{\mathsf{f}}_{\infty}$ on 
\begin{equation*}
\Lambda^{(0)}_{\infty} := \Big\{ (v_0,v_1,\dots)\in (\N_0 \times \N_0)^{\N_0} \, \colon \, v_0 = (0,0) \, \wedge \,  v_i-v_{i-1} \in \{ (1,1),(1,0),(1,-1)\} \, \forall i\in \N \Big\}
\end{equation*}
 such that for all subsets $A \subseteq \Psi_m$ for some fixed $m\in \N$
\begin{align}\label{eq:ConditionalMeasure}
\begin{split}
\mathbf{P}^{\mathsf{f}}_{m,V}(A) &= \bar{\mathbf{P}}^{\mathsf{f}}_{\infty}( (\zeta_x)_{x \in \lsem 0,m \rsem} \in A)  \\ \mathbf{P}^{\mathsf{c}}_{m,V}(A) &= \bar{\mathbf{P}}^{\mathsf{f}}_{\infty}( (\zeta_x)_{x \in \lsem 0,m \rsem} \in A \, | \, h_m(\zeta)=0 )  .
 \end{split}
\end{align}
Moreover, using Proposition \ref{pro:RegenerationPoint} and Lemma \ref{lem:ConvergenceBeginningEnd}, we see that for all $\xi \in \Psi_{\ast}$
\begin{equation}\label{eq:LimitPolymerMeasure}
\Q(\xi)=\lim_{N \rightarrow \infty}\bar{\mathbf{P}}^{\mathsf{f}}_{N}\left(\tau_0\leq N \text{ and } (h_x(\zeta))_{x \in \lsem 0,\tau_0 \rsem} = \xi \right)= \bar{\mathbf{P}}^{\mathsf{f}}_{\infty}\left( (h_x(\zeta))_{x \in \lsem 0,\tau_0\rsem} = \xi \right) . 
\end{equation} 
Combining now \eqref{eq:ConditionalMeasure} and \eqref{eq:LimitPolymerMeasure}, together with the spatial Markov property for the measure $\bar{\mathbf{P}}^{\mathsf{f}}_{\infty}$ and the shift invariance property \eqref{eq:ShiftInvariantStationary} for the measure $\mathbf{P}_{\textup{stat}}$, we get \eqref{eq:ConditionedLawStationary}.
\end{proof}

We have the following relation between the measures $\mathbf{P}^{\mathsf{c}}_{N,V}$ and $\mathbf{P}_{\textup{stat}}$.

\begin{lemma}\label{lem:CouplingRegenerativeProcesses}

Let $\delta>0$ and assume that either $u>\max(1,v)$ or $v>\max(1,u)$ holds. Then there exist $M_0,N_0 \in \N$ and a coupling $\mathbf{P}^{\ast}$ of $\mathbf{P}^{\mathsf{c}}_{N,V}$ and $\mathbf{P}_{\textup{stat}}$ such that for all $N\geq N_0$, 
\begin{equation}
\mathbf{P}^{\ast}\left( h_x(\zeta)=h_{x}(\xi) \text{ for all } x\in \lsem M_0, N-M_0 \rsem  \right) \geq 1-\delta
\end{equation} where we let $\zeta \sim \mathbf{P}^{\mathsf{c}}_{N,V}$ and $\xi \sim \mathbf{P}_{\textup{stat}}$ according to $\mathbf{P}^{\ast}$.
\end{lemma}
\begin{proof}
Let $(\zeta^{\prime},\xi) \sim \mathbf{P}^{\mathsf{c}}_{N,V} \times \mathbf{P}_{\textup{stat}}$ be chosen independently. 
Since the measure $\Q$ has full support on $\N$, we see that for all $\delta>0$, we find some $M_0=M_0(\delta)$ and $N_0=N_0(\delta)$ such that
\begin{equation*}
\mathcal{B}_{M_0} := \left\{ \exists x\in \lsem 0,M_0 \rsem  \text{ and } y \in \lsem N-M_0,N \rsem \, \colon \, h_{x}(\zeta^{\prime})=h_{x}(\xi)=h_{y}(\zeta^{\prime})=h_{y}(\xi) =0 \right\} 
\end{equation*} satisfies
\begin{equation*}
(\mathbf{P}^{\mathsf{c}}_{N,V} \times \mathbf{P}_{\textup{stat}}) ( \mathcal{B}_{M_0} ) \geq 1 - \delta
\end{equation*} for all $N \geq N_0$. Under the event $\mathcal{B}_{M_0}$, let $a^{\prime} \in \lsem 0,M_0 \rsem $ and $b^{\prime} \in \lsem N-M_0,N \rsem$ denote the smallest, respectively the largest points such that $h_{a^{\prime}}(\zeta^{\prime})=h_{b^{\prime}}(\zeta^{\prime})=h_{a^{\prime}}(\xi)=h_{b^{\prime}}(\xi)=0$ holds. By \eqref{eq:ConditionedLawStationary} in Lemma \ref{lem:StationaryProcess}, note that for all $A^{\prime} \subseteq \Psi_{b^{\prime}-a^{\prime}}$, and all choices of $a^{\prime}$ and $b^{\prime}$,
\begin{equation*}
\mathbf{P}^{\mathsf{c}}_{N,V}\big(  
(h_{x}(\zeta^{\prime}))_{x \in \lsem a^{\prime},b^{\prime}\rsem} \in A^{\prime} \, \big| \,  h_{a^{\prime}}(\zeta^{\prime}) =0
\text{ and } h_{b^{\prime}}(\zeta^{\prime})=0\big) = \mathbf{P}_{b^{\prime}-a^{\prime},V}^{\mathsf{c}}(A) .
\end{equation*} 
Now suppose that the event $\mathcal{B}_{M_0}$ holds. We condition on the value of $a^{\prime}$ and $b^{\prime}$, respectively, and choose $\zeta=(\zeta_x)_{x \in \lsem N \rsem} \in \Psi_N$ with
\begin{equation*}
\zeta_x = \begin{cases}
\xi_{x} & \text{ if } x \in \lsem a^{\prime}, b^{\prime} \rsem \\
\zeta^{\prime}_x & \text{ otherwise }
\end{cases}
\end{equation*} for the coupling of $(\zeta,\xi) \sim \mathbf{P}^{\ast}$, and $\zeta=\zeta^{\prime}$ on the complement of $\mathcal{B}_{M_0}$, to conclude.
\end{proof}
%
%
%

\subsection{From regenerative processes to approximation by a product measure}\label{sec:CouplingPinnedPolymer}

We now show that the measure $\mathbf{P}_{\textup{stat}}$ constructed in Section \ref{sec:RegenerationLocalized} implies a product structure in the corresponding stationary distribution of the open ASEP. To do so, we first describe how a configuration according to $\mathbf{P}_{\textup{stat}}$ yields a measure on the space of particle configurations $\{0,1\}^{n}$ for some $n \in \N$. Let $\xi \sim \mathbf{P}_{\textup{stat}}$. From $\xi$, we get a random configuration $\omega^{(\xi)} \in \mathcal{A}^{n}$ by
\begin{equation}\label{eq:MotzkinFromPath1}
\omega^{(\xi)}_x = \begin{cases} \Nc &\text{ if } h_{x}(\xi)-h_{x-1}(\xi)=1 \\
\Sc &\text{ if } h_{x}(\xi)-h_{x-1}(\xi)=-1 , 
\end{cases}
\end{equation}
and when $h_{x}(\xi)=h_{x-1}(\xi)$ for some $x\in \lsem n \rsem$ by independently assigning
\begin{equation}\label{eq:MotzkinFromPath2}
\P\left( \omega^{(\xi)}_x = \Eb \right) = 1 - \P\left( \omega^{(\xi)}_x = \Ew \right) = \frac{1+uq^{h_x(\xi)}}{2+(u+v)q^{h_x(\xi)}} . 
\end{equation} We denote the corresponding measure on $\mathcal{A}^{n}$ by $\P^{(n)}_{\textup{stat}}$. For $\eta \in \{0,1\}^{n}$, recall from \eqref{def:SetOfPathsForConfiguration} the set $\mathcal{C}_{\eta}$ of all bi-colored Motzkin paths which map to $\eta$. Then let $\mu_{\textup{stat}}^{(n)}$ be defined on $\{0,1\}^{n}$ by
\begin{equation*}
\mu_{\textup{stat}}^{(n)}(\eta) = \sum_{\omega \in \mathcal{C}_{\eta}} \P^{(n)}_{\textup{stat}}(\omega) .
\end{equation*} 
We make the following observation about the measure $\mu_{\textup{stat}}^{(n)}$.

\begin{lemma}\label{lem:BernoulliPathStructure}
Assume that $u>\max(1,v)$ holds and let $a,b \in \N$ with $n=b-a \in \N$. Then the measure $\mu_{\textup{stat}}^{(n)}$ is a Bernoulli-$\rho$-product measure on $\{ 0, 1\}^{n}$ for $\rho=\alpha(1-q)^{-1}$. Similarly, for $v>\max(1,u)$, $\mu_{\textup{stat}}^{(n)}$ is a Bernoulli-$\rho$-product measure on $\{ 0, 1\}^{n}$ for $\rho=1-\beta(1-q)^{-1}$.
\end{lemma}
\begin{proof}[Proof of Lemma \ref{lem:BernoulliPathStructure}]
We will only consider $u>\max(1,v)$ as the arguments for $v>\max(1,u)$ are similar. 
Let $\delta>0$ and recall the constant $M_0=M_0(\delta)$ from Lemma \ref{lem:CouplingRegenerativeProcesses}. 
By the shift invariance property of $\mathbf{P}_{\textup{stat}}$, and assuming without loss of generality that $n$ is even, we write $a=m/2 - n/2$ and $b=m/2 + n/2-1$ for some $m \in \N$ with $m\geq M_0+n$. Set $I_{m,n}:= \lsem m/2 - n/2, m/2 + n/2-1\rsem$ and recall that we write $\mu_{I_{m,n}}=\mu^{m,q,\alpha,\beta}_{I_{m,n}}$ for the invariant measure of an open ASEP on the segment $\lsem m \rsem$ with parameters $q,\alpha,\beta$, projected to the interval $I_{m,n}$; see also \eqref{def:ProjectedMu}. Recall from Lemma \ref{lem:WeightToPolymer} and the relations \eqref{eq:MotzkinFromPath1} and \eqref{eq:MotzkinFromPath2} that we can write the invariant measure $\mu_{I_{m,n}}$ of in terms of the measure $\mathbf{P}^{\mathsf{c}}_{m,V}$. From Proposition \ref{pro:Finite} and Lemma \ref{lem:CouplingRegenerativeProcesses}, choosing $m=m(\delta,n,\alpha,\beta,q)$ now sufficiently large, we get
\begin{equation}\label{eq:BernoulliToLocal}
\TV{\mu_{I_{m,n}} - \textup{Ber}_{I_{m,n}}\left(\frac{\alpha}{1-q}\right)} \leq 2\delta , 
\end{equation} 
where we recall that $\textup{Ber}_{I_{m,n}}(\rho)$ denotes the Bernoulli-$\rho$-product measure on $\{ 0,1 \}^{n}$ for $\rho \in [0,1]$.
Moreover, using the coupling in Lemma \ref{lem:CouplingRegenerativeProcesses} between $\mathbf{P}^{\mathsf{c}}_{m,V}$ and $\mathbf{P}_{\textup{stat}}$, we see that
\begin{equation}\label{eq:LocalToStationary}
\TV{ \mu_{I_{m,n}} - \mu_{\textup{stat}}^{(n)}} \leq \delta . 
\end{equation} 
Finally, using the triangle inequality for the total variation distance, and the fact that $\delta>0$ was arbitrary, we combine \eqref{eq:BernoulliToLocal} and \eqref{eq:LocalToStationary} to conclude. 
\end{proof}

\begin{proof}[Proof of Theorem~\ref{thm:HighLowFan}] Using Lemma \ref{lem:Weights} and Lemma \ref{lem:WeightToPolymer} to express the stationary distribution $\mu$ of the open ASEP by the measure $ \mathbf{P}^{\mathsf{c}}_{N,V}$, the approximation of the stationary distribution in \eqref{eq:HighLowFan} follows by combining Lemma \ref{lem:CouplingRegenerativeProcesses} and  Lemma \ref{lem:BernoulliPathStructure}.
\end{proof}

\section{Approximation in the shock region of the high and low density phase}\label{sec:ShockPhase}

In this section, we establish Theorem \ref{thm:HighLowShock} and Theorem \ref{thm:HighLowShockWASEP} on approximating the stationary distribution of the open ASEP and open WASEP in the shock region of the high and low density phase. Since the total path weights defined in \eqref{def:TotalWeight} may be negative in the shock region, the techniques presented in Sections \ref{sec:DelocalizationMaxCurrent} and \ref{sec:LocalizationPhase} do no longer apply. However, as remarked in Section \ref{sec:MPAFiniteRep}, for special choices of the boundary parameters, the invariant of the open ASEP has a simple representation as a convex combination of Bernoulli shock measures. This was first observed by Jafarpour and Masharian in \cite{JM:Finite}. 

\subsection{Bernoulli shock measures as invariant measures of the open ASEP}\label{sec:ShockMeasuresSchutz}

Recall the definition of the parameters $u$ and $v$ from \eqref{def:uAndv} for the open ASEP, and from \eqref{def:uAndvWASEP} for the open WASEP. 
Let $\rho_0= \frac{1}{1+u}$ and $\rho_{k}=\frac{v}{1+v}$ denote the \textbf{effective density} at the left end and right end of the segment. A simple computation shows that whenever the condition $uvq^{k}=1$ from \eqref{eq:FiniteRep} holds for some $k\in \N_0$, and $uv>1$, there exist some $(\rho_i)_{i \in \lsem k-1\rsem}$ with $\rho_0 \leq \rho_1 \leq \cdots \leq \rho_k$ such that
\begin{equation}\label{eq:FugacityRelation}
\frac{\rho_i}{1-\rho_i}= q^{-1} \frac{\rho_{i-1}}{1-\rho_{i-1}}
\end{equation} for all $i \in \lsem k \rsem$. In the following, fix for all $i\in \lsem k \rsem$ some $\rho^{\ast}_i \in [0,1]$. We refer to $\rho_i$  as \textbf{bulk densities} and $\rho_i^{\star}$ as \textbf{shock densities}. For $n \in \lsem k \rsem \cup \{0 \}$ and $k\in [N]$, let $\Omega_{N,n}$ with
\begin{equation}\label{def:StateSpaceKParticles}
\Omega_{N,n} := \left\{ \mathbf{x}=(x_1,x_2,\dots,x_n) \in \lsem N \rsem^{n}  \text{ with } x_1 < x_2 < \cdots < x_n \right\}
\end{equation} be the space of locations for $n$ shocks, and note that $\Omega_{N,n}$ can be identified with the state space of an $n$-particle exclusion process, where the particles identify the shock locations. For a given vector $\mathbf{x} \in \Omega_{N,n}$, and a shift parameter $y\in \lsem 0, k-n \rsem$, the corresponding \textbf{Bernoulli shock measure} $\mu^{\mathbf{x},y}$ is defined as the product measure on the space $\{0,1\}^{N}$ with
\begin{equation}\label{def:shock}
\mu^{\textbf{x},y}(\eta)= \prod_{j=0}^{n+1} p^{\textbf{x},y}_{\eta(j)}
\end{equation}
for all $\eta \in \{0,1\}^{N}$ and marginals
\begin{equation}\label{eq:shockMarginals}
\begin{split}p^{\textbf{x},y}_{\eta(j)} := \begin{cases}
    (1-\rho_{i+y}^{\star} )(1-\eta(j))+ \rho_{i+y}^{\star} \eta(j), & \mbox{ if }j=x_i \text{ for some } i \in \lsem n \rsem  \\
    (1-\rho_{i+y} )(1-\eta(j))+ \rho_{i+y} \eta(j), & \mbox{ if  } x_i<j < x_{i+1} \text{ for some } i \in \lsem 0,n \rsem . 
   \end{cases} \end{split}
\end{equation} Here,  we use the conventions that $x_0=0$ and $x_{k+1}=N+1$.  In order to simplify notation, let $j_i =(1-q) \rho_i(1-\rho_i)$ for $i \in \lsem 0,k \rsem$, and set $d_i = j_i / j_{i-1}$. The following characterization of the invariant measure is similar to Theorem 3.9 in \cite{S:ASEPReverse}; see also \cite{S:Schutz2} when $k=1$.
\begin{proposition}\label{pro:Schutz}
Consider the shock region of the high density phase, i.e.\ $uv>1$ and $v>\max(1,u)$, and assume that $uvq^{k}=1$ holds for some $k \leq N$. Let $\rho_i^{\ast} \equiv 0$. Then the unique stationary measure $\mu=\mu^{N,q,\alpha,\beta}$ of the open ASEP can be written as
\begin{equation}\label{eq:ShockMeasureHigh}
\mu= \frac{1}{Z^{u,v}_{N,k}}\sum_{n=0}^{k} \frac{1}{Z_{n,u}} \sum_{\textup{\textbf{x}}\in \Omega_{N,n}} \left( \prod_{i=1}^{n} d_{i+k-n}^{x_i} \right) \mu^{\textup{\textbf{x}},k-n},
\end{equation}
where we set $Z_{n,u} := \prod_{j=0}^{n} (u (1-q^{n-j}))$, and let $Z^{u,v}_{N,k}$ be a suitable renormalization constant. Similarly, consider  the shock regime of the low density phase, i.e.\ where $uv>1$ and $u>\max(1,v)$, and assume that  $uvq^{k}=1$ as well as $\rho_i^{\ast} \equiv 1$holds. Then $\mu=\mu^{N,q,\alpha,\beta}$  satisfies
\begin{equation}\label{eq:ShockMeasureLow}
\mu= \frac{1}{\tilde{Z}^{u,v}_{N,k}}\sum_{n=0}^{k} \frac{1}{Z_{k,v}} \sum_{\textup{\textbf{x}}\in \Omega_{N,n}} \left( \prod_{i=1}^n d_i^{x_i} \right) \mu^{\textup{\textbf{x}},0},
\end{equation} with $Z_{n,v} := \prod_{j=0}^{n} (v (1-q^{n-j}))$, and a suitable renormalization constant $\tilde{Z}^{u,v}_{N,k}$.
\end{proposition}

The proof of Proposition \ref{pro:Schutz} is deferred to the appendix. Let us remark that Theorem~3.9 in \cite{S:ASEPReverse} provides a similar characterization of the invariant measure for the five parameter version of the open ASEP, but for a specific subset of parameter $u$ and $v$ with $uvq^{k}=1$.

\subsection{Concentration of shocks in the reverse dual}

We will only consider the low density phase in the shock region of the open ASEP and open WASEP. All statements extend to the high density phase using the symmetry between particles and holes. Using Proposition \ref{pro:Schutz}, we express the invariant measure of the open ASEP as a convex combination of invariant measures of asymmetric simple exclusion processes on a closed segment with particle depending hopping rates. We refer to these processes as \textbf{duals}. More precisely, let $n\leq k$, and let $(\eta^{\ast}_t)_{t \geq 0}$ be the simple exclusion process with $n$ particles, where the $i^{\textup{th}}$ particle (counted from the left) jumps to the right at rate $j_{i}(\rho_{i}-\rho_{i-1})^{-1}$, and to the left at rate $j_{i-1}(\rho_{i}-\rho_{i-1})^{-1}$ under the exclusion constraint.
Let $\mu_{N,n}^{\ast}$ denote the stationary distribution of the process $(\eta^{\ast}_t)_{t \geq 0}$, and note that we indeed have for all $\mathbf{x} =(x_1,\dots,x_n)\in \Omega_{N,n}$
\begin{equation}\label{eq:ShockInvariantDistribution}
\mu_{N,n}^{\ast}(\mathbf{x}) = \frac{1}{Z}\prod_{i=1}^n d_{i}^{x_i} 
\end{equation}
with some normalization constant $Z$. This can be seen by verifying the detailed balance equations; see  Proposition~3.1 in \cite{S:ASEPReverse}. In the low density phase, we are interested in the position of the left-most particle under the stationary distribution $\mu_{N,n}^{\ast}$ for all $n\leq k$. The following proposition summarizes our results. 

\begin{proposition}\label{pro:ShockConcentration}
Consider the low density phase of the open ASEP where $u>\max(v,1)$ and $uvq^{k}=1$ for some fixed $k\in\N$. Then for all $\delta>0$, there exists some  $C>0$ such that 
\begin{equation}
\mu^{\ast}_{N,n}( x_1  \geq N-C  ) \geq 1- \delta  
\end{equation} for all $n\leq k$, and all $N$ sufficiently large. For the open WASEP, suppose that $q$ satisfies \eqref{eq:qAssumption} for some $\varepsilon \in (0,1)$, and that $u,v$ from \eqref{def:uAndvWASEP} satisfy $uvq^{k}=1$ with some $k=k(N)\in \N$. Then for all $\delta>0$, there exists some constant $C^{\prime}>0$ such that  
\begin{equation}
\mu^{\ast}_{N,n}( x_1  \geq N-C^{\prime}N^{\varepsilon}  ) \geq 1- \delta .   
\end{equation}
\end{proposition}
In order to simplify notation, we will only show the case $n=k$ in Proposition~\ref{pro:ShockConcentration}, as the arguments are analogous for $n<k$. First, we argue that the extremal shock locations $x_1$ and $x_k$ must be close together. We then argue that in the low density phase, all shocks concentrate at the right end of the segment. 

\begin{lemma}\label{lem:ConcentrationBias}
Suppose that $q \in (0,1)$ and that $u>\max(1,v)$ as well as that $uvq^{k}=1$ holds for some fixed $k\in \N$. Then  there exists some $M=M(k)$ and constants $c_1,c_2>0$ such that 
\begin{equation}
\mu^{\ast}_{N,k}( | x_k - x_1 | \geq y M ) \leq c_1\exp(-c_2 y) . 
\end{equation} for all $y>0$, and all $N$ large enough. Similarly, take $q$ from \eqref{eq:qAssumption} with $\varepsilon \in (0,1)$ and $c>0$, and assume that $u>\max(1,v)$  and $uvq^{k}=1$ for some $k=k(N)\in \lsem N \rsem$, and $u,v$ from \eqref{def:uAndvWASEP}. Then  there exists some $M^{\prime}=M^{\prime}(\varepsilon,c)$ and  $c_3,c_4>0$ such that
\begin{equation}\label{eq:ConcentrationBiasWASEP}
\mu^{\ast}_{N,k}( | x_k - x_1 | \geq y M^{\prime} N^{\varepsilon} ) \leq c_3\exp(-c_4 y)
\end{equation} for all $y>0$, and all $N$ large enough.
\end{lemma}

In order to show Lemma \ref{lem:ConcentrationBias}, it will be convenient to consider as an auxiliary process a simple exclusion process $(\xi_t)_{t \geq 0}$ on $\lsem -\infty, N \rsem$ with $k_1$ many asymmetric, and $k_2$ many symmetric particles. More precisely, fix some $\gamma \in (0,1/2)$, $k_1,k_2 \in \N$, and rates $(r_i)_{i \in \lsem k_1+k_2 \rsem}$. For $k,N \in \N$ with $k=k_1+k_2$, let $(\xi_t)_{t \geq 0}$ denote a simple exclusion process on 
\begin{equation*}
\Omega^{-}_{N,k} := \left\{ \xi = (\xi_1,\xi_2,\xi_3,\dots, \xi_k) \in \lsem -\infty , N \rsem ^{k} \, \colon \,  \xi_{1} < \xi_2 < \cdots < \xi_{k_1+k_2-1} <  \xi_{k_1+k_2} \leq N\right\} . 
\end{equation*}
In the exclusion process $(\xi_t)_{t \geq 0}$, the $i^{\textup{th}}$ particle, counted from left to right, jumps at rate $r_i>0$. The rightmost $k_2$ particles perform symmetric random walks, while the leftmost $k_1$ particles perform asymmetric random walks with bias $\gamma \in (0,\frac{1}{2})$, i.e.\ when the clock rings, they attempt a jump to the right with probability $\frac{1}{2}+\gamma$, and to the left with probability $\frac{1}{2}-\gamma$. Both types of moves are subject to reflection at the boundaries. We have the following result on the stationary distribution $\mu_{\N}$ of $(\xi_t)_{t \geq 0}$.
\begin{lemma}\label{lem:AuxiliaryASEP} Consider the exclusion process $(\xi_t)_{t \geq 0}$ on $\Omega_{N,k}^{-}$ with $\gamma \in (0,1/2)$ and uniformly bounded jump rates $(r_i)_{i \in \lsem k \rsem}$ for $k=k_1+k_2$. Assume there exist constants $c_1,c_2>0$ such that for all $n \in \N$
\begin{equation}
 k_1,k_2 \in \lsem c_1 n , c_2 n \rsem 
\end{equation} with $k_1=k_1(n)$ and $k_2=k_2(n)$. 
Then for $N\in \N$ fixed, $\mu_{\N}$ is unique, and there exist $c>0$ and $c_3,c_4>0$, depending only on $c_1,c_2>0$, such that for all $n \in \N$ and $y>0$
\begin{equation}
\mu_{\N}(\xi_{1} \leq  N - y c n  ) \leq   c_3 \exp(-c_4 y) . 
\end{equation}
\end{lemma}
\begin{proof}
Note that verifying the detailed balance equations, we see that the stationary distribution $\mu_{\N}$ is given by
\begin{equation*}
\mu_{\N}(  \xi ) = \frac{1}{Z}\prod_{i=1}^{k_1} \left( \frac{\frac{1}{2}-\gamma}{\frac{1}{2}+\gamma}\right)^{N-\xi_{i}} , 
\end{equation*} where $Z$ is a suitable renormalization constant. For all $z \in \N$, let $\mathcal{A}_z$ be the event that $\xi_{k_1+1}=N-z$. We claim that there exists a $C=C(\gamma,c_1,c_2)>0$ such that for all $z \in \N$, 
\begin{equation*}
\mu_{\N}(  \mathcal{A}_z  ) \leq C \binom{z}{k_2} \left( \frac{\frac{1}{2}-\gamma}{\frac{1}{2}+\gamma}\right)^{z} . 
\end{equation*} This follows noting that there are $\binom{z}{k_2}$ possibilities to place the rightmost $k_2$ particles on positions $\lsem N-z, N \rsem$, and summing over all possibilities to distribute the remaining $k_1$ particles on the halfspace $\lsem -\infty, N-z -1\rsem$. As a consequence,
\begin{equation*}
\mu_{\N}\Big( \bigcup_{ z \geq y \tilde{c} n} \mathcal{A}_{z}   \Big) \leq   \tilde{c}_1 \exp(-\tilde{c}_2 y) 
\end{equation*}
 for some constants $\tilde{c},\tilde{c}_1,\tilde{c}_2>0$, and for all $y>0$ and $n\in \N$. Note that for all $z>0$, on the event $\mathcal{A}_z$, the law of the leftmost $k_1$ particles is given by the stationary distribution $\tilde{\mu}_{z}$ of an asymmetric simple exclusion process on the halfspace $\lsem -\infty, -z \rsem$. Proposition 4.2 in \cite{BN:CutoffASEP} states that there exist constants $\tilde{c}_3,\tilde{c}_4$ such that for all $x>0$
 \begin{equation}
 \mu_z( \xi_{1} < -z - k_1 -  x  ) \leq  \tilde{c}_3 \exp(- \tilde{c}_4 x) , 
 \end{equation} allowing us to conclude.
\end{proof}
Next, we consider the partial order $\succeq_{\textup{o}}$ between configurations induced by the ordering of particles, i.e.\ for $\mathbf{x}=(x_1,x_2,\dots,x_k) \in \Omega_{N,k}$ and $\mathbf{x}^{\prime}=(x^{\prime}_1,x^{\prime}_2,\dots,x^{\prime}_k) \in \Omega_{N,k}$, we set
\begin{equation}\label{def:ParticleOrder}
\mathbf{x} \succeq_{\textup{o}} \mathbf{x}^{\prime} \quad \Leftrightarrow \quad x_{i} \geq x_{i}^{\prime}  \text{ for all } i \in \lsem k \rsem . 
\end{equation}
Observe that the partial order $\succeq_{\textup{o}}$ naturally extends to the space $\Omega^{-}_{N,k}$, and allows to compare configurations on different state
spaces with $k$ particles, using the right-hand side of \eqref{def:ParticleOrder}. 
We define a natural coupling $\mathbf{P}_{\textup{o}}$ for exclusion processes with $k$ particles and the same jump rates $(r_i)_{i \in \lsem k \rsem}$, but a potentially different biases $(\gamma^{(1)}_i)$ and $(\gamma_i^{(2)})$. For two  exclusion processes $(\eta_t^{(1)})_{t \geq 0}$ and $(\eta_t^{(2)})_{t \geq 0}$ coupled according to $\mathbf{P}_{\textup{o}}$, we assign to the $i^{\textup{th}}$ particle a rate $r_i$ Poisson clock. Whenever the clock of particle $i$ rings, sample a Uniform-$[0,1]$-random variable $U$. The particle attempts in $(\eta_t^{(1)})_{t \geq 0}$ a jump to the right if $U < 1/2+\gamma^{(1)}_i $, respectively in $(\eta_t^{(2)})_{t \geq 0}$ if $U < 1/2+\gamma^{(2)}_i$, and to the left otherwise.
The next lemma states that $\mathbf{P}_{\textup{o}}$ preserves the partial order $\succeq_{\textup{o}}$, provide that the particle biases are ordered. Since this follows by a standard argument -- see for example Lemma 2.1 in \cite{GNS:MixingOpen} for a similar coupling -- we give only a sketch of proof.
\begin{lemma}\label{lem:ParticleCoupling}
Let $(\eta_t^{(1)})_{t \geq 0}$ and $(\eta_t^{(2)})_{t \geq 0}$ be two exclusion processes on $\Omega_{N,k}$ with common jump rates $(r_i)$, and assume that $\gamma^{(1)}_i \geq \gamma^{(2)}_i$ holds for all $i \in \lsem k \rsem$. Then we have that
\begin{equation}
\mathbf{P}_{\textup{o}}\left( \eta_t^{(1)} \succeq_{\textup{o}} \eta_t^{(2)} \text{ for all } t \geq 0 \, \mid \, \eta_0^{(1)} \succeq_{\textup{o}} \eta_0^{(2)} \right) = 1 . 
\end{equation}
The same holds true if $(\eta_t^{(1)})_{t \geq 0}$ is defined on the space $\Omega_{N,k}$ for some $N\in \N$, while $(\eta_t^{(2)})_{t \geq 0}$ is an exclusion process on $\Omega^{-}_{N,k}$. 
\end{lemma}
\begin{proof}[Sketch of proof]
The claim follows by induction over the jump times of particles. More precisely, by the ordering of jump probabilities, whenever the  $i^{\text{th}}$ particle in $(\eta_t^{(2)})_{t \geq 0}$ attempts a jump to the left at time $s$, so does the $i^{\text{th}}$ particle in $(\eta_t^{(2)})_{t \geq 0}$. Note that when the $i^{\text{th}}$ particle occupies the same position in both processes at time $s_{-}$, the partial ordering at time $s_{-}$ ensures that the move is performed for $(\eta_t^{(2)})_{t \geq 0}$ whenever it is performed in $(\eta_t^{(1)})_{t \geq 0}$. Hence, the partial order $\succeq_{\textup{o}}$ is preserved at time $s$.
A similar argument applies when the $i^{\text{th}}$ particle in $(\eta_t^{(2)})_{t \geq 0}$ attempts a jump to the right. 
\end{proof}


\begin{proof}[Proof of Lemma \ref{lem:ConcentrationBias}]
We will in the following only show \eqref{eq:ConcentrationBiasWASEP} for the open WASEP as the arguments for the open ASEP are analogues. We distinguish two cases. First, assume that $u> 1 \geq v$. Then the effective density $\rho_{k+1}$ satisfies $\rho_{k+1}\leq \frac{1}{2}$, and hence $j_i \leq j_{i-1}$ for all $i \in \lsem k \rsem$. Since $uvq^{k}=1$ for some $k$ of order $N^{\varepsilon}$, there exist constants $\delta_1,\delta_2>0$ such that
\begin{equation}\label{eq:JumpRateAssumptions}
  \frac{j_{i-1}}{j_i-j_{i-1}} \leq \frac{j_{i}}{j_i-j_{i-1}} - 2 \delta_1
\end{equation} for all $i \leq \delta_2 N^{\varepsilon}$. Let $(\xi_{t})_{t \geq 0}$ be the exclusion process on $\Omega_{N,k}^{-}$ with $r_i=(j_{i-1}+j_i)(\rho_i-\rho_{i-1})^{-1}$ for all $i \in \lsem k \rsem$, and bias parameters
\begin{align*}
\gamma_i &= \begin{cases} 0 &\text{ if } i > k_1  \\
\delta_1 &\text{ if } i \leq k_1 , 
\end{cases} 
\end{align*} where we set $k_1=\delta_2 N^{\varepsilon}$ and $k_2=k-k_1$. Consider now the coupling $\mathbf{P}_{\textup{o}}$ between the dual process $(\eta^{\ast}_t)_{t \geq 0}$ and the exclusion process $(\xi_{t})_{t \geq 0}$ on $\Omega_{N,k}^{-}$. Using Lemma \ref{lem:ParticleCoupling} together with the assumption \eqref{eq:JumpRateAssumptions} on the transition rates, we take $t \rightarrow \infty$ to see that for all $z\in \N$ and $ N\in \N$
\begin{equation*}
\mu_{N,k}^{\ast}(x_{1} \leq  N - z  ) \leq \mu_{\N}(\xi_{1} \leq  N - z  ) . 
\end{equation*} Lemma \ref{lem:AuxiliaryASEP} for  $(\xi_{t})_{t \geq 0}$ gives the desired result. 
Next, suppose that $u>v>1$ holds. Then there exists some index $i_{\ast}$ such $j_i>j_{i-1}$ for all $i \geq i_{\ast}$ and $j_i<j_{i-1}$ for all $i<i_{\ast}$. Moreover, observe that there exist some positive constants $(\delta_i)_{i \in \lsem 4 \rsem}$ such that
\begin{align}\label{eq:ConditionsWASEP}
\begin{split}
\frac{j_{i-1}}{j_i-j_{i-1}} &\leq \frac{j_{i}}{j_i-j_{i-1}} - 2 \delta_1  \ \ \text{ for all } i < \delta_{2}N^{\varepsilon} \\
\frac{j_{i-1}}{j_i-j_{i-1}} &\geq \frac{j_{i}}{j_i-j_{i-1}} + 2 \delta_3 \  \ \text{ for all } i > (1-\delta_{4})N^{\varepsilon}.
\end{split}
\end{align}
 Conditioning on the value of $x_{i_{\ast}}$ in $\mu_{N,k}^{\ast}$ and using \eqref{eq:ConditionsWASEP}, we apply the same arguments as in the case $u>1 \geq v$, but for the positions $(x_i)_{i < i_{\ast}}$ and $(x_i)_{i > i_{\ast}}$ separately, to conclude. 
\end{proof}

We now argue that in the low density phase, shocks are concentrated at the right end of the segment.

\begin{proof}[Proof of Proposition \ref{pro:ShockConcentration}]
As for the proof of Lemma \ref{lem:ConcentrationBias}, we will only consider the open WASEP as the arguments are similar for the open ASEP. For $\mathbf{x}=(x_1,x_2,\dots,x_k) \in \Omega_{N,k}$ and $z \in \N$,  let $\mathcal{B}_{z}$ be the event
\begin{equation}
\mathcal{B}_{z} :=  \big\{ x_{1}<N- 2 z N^{\varepsilon}  \big\} \cap \big\{ |x_k-x_1| < z N^{\varepsilon} \big\} . 
\end{equation}
We claim that there exist constants $c_1,c_2>0$ such that for all $z \in \N$,
\begin{equation}\label{eq:BeforeShift}
\mu_{N,k}^{\ast}( \mathcal{B}_z ) \leq c_1 \exp(-c_2 z) . 
\end{equation}
To see this, note that we can map every $w=(w_1,w_2,\dots,w_k) \in \mathcal{B}_z$ to a configuration $\bar{w}=(\bar{w}_1,\bar{w}_2,\dots,\bar{w}_k) \in \Omega_{N,k}$ by $\bar{w}_{i}:=w_{i}+z$ for all $i\in \lsem k \rsem$. Observe that for all $z\in \N$
\begin{equation}
\frac{\mu_N^{\ast}(w)}{\mu_N^{\ast}(\bar{w})}=\left(\prod_{i=1}^{k} d_i^{w_i} \right) \left(\prod_{i=1}^{k} d_i^{\bar{w}_i} \right)^{-1} = \left(\prod_{i=1}^{k} d_i \right)^{-z} = \left(\frac{\rho_{0}(1-\rho_0)}{\rho_{k+1}(1-\rho_{k+1})}\right)^{z} =  \left( \frac{(1+v)^2u}{(1+u)^2v} \right)^{z} . 
\end{equation} 
Since $(1+u)^2v>(1+v)^2u$ as we assume $u>\max(1,v)$, and the map $w \mapsto \bar{w}$ is injective, equation \eqref{eq:ShockInvariantDistribution}  now yields \eqref{eq:BeforeShift}. Using Lemma~\ref{lem:ConcentrationBias} to bound the probability of the event $\{ |x_k-x_1| \leq z N^{\varepsilon} \}$ from below uniformly in $z \in \N$, we conclude.
%
%
%
\end{proof}

\subsection{From shock measures to approximation by a product measure}\label{sec:ApproximationShock}

Note that from Proposition \ref{pro:ShockConcentration}, we immediately get that Theorem \ref{thm:HighLowShock} and Theorem \ref{thm:HighLowShockWASEP} hold when $uvq^{k}=1$ for some $k\in \N_0$. In order to extend this to a more general range of parameters $\alpha,\beta>0$, we use the \textbf{canonical coupling}, also called the \textbf{basic coupling}, for the open ASEP in order to compare asymmetric exclusion processes with different boundary parameters. For $\alpha^{\prime}>\alpha>0$ and $\beta>\beta^{\prime}>0$, let $(\eta_t)_{t \geq 0}$ and $(\eta^{\prime}_t)_{t \geq 0}$ be two open ASEPs with respect to boundary parameters $\alpha,\beta>0$, and $\alpha^{\prime},\beta^{\prime}>0$. Then under the basic coupling $\mathbf{P}_{\textup{b}}$, we assign independent rate $1$ and rate $q$ Poisson clocks to all edges. When the rate $1$ clock rings at time $s$ for an edge $\{x,x+1\}$, and $\eta_{s_-}(x)=1-\eta_{s_-}(x+1)=1$, we move the particle in $(\eta_t)_{t \geq 0}$ from $x$ to $x+1$, and similarly for $(\eta^{\prime}_t)_{t \geq 0}$ when $\eta^{\prime}_{s_-}(x)=1-\eta^{\prime}_{s_-}(x+1)=1$. The same construction applies for the rate $q$ Poisson clocks. In addition, we use  rate $\alpha$ and rate $\beta^{\prime}$ Poisson clocks to determine for both processes when to attempt to enter and exit a particle at the boundaries. Furthermore, we attempt to place a particle at site $1$ in $(\eta^{\prime}_t)_{t \geq 0}$ at rate $\alpha^{\prime}-\alpha$, and to remove a particle from $(\eta_t)_{t \geq 0}$ at site $N$ at rate $\beta-\beta^{\prime}$. Lemma 2.1 in \cite{GNS:MixingOpen} guarantees that the component-wise  partial ordering $\succeq_{\textup{c}}$ on the state space $\{ 0,1\}^{N}$ is preserved by $\mathbf{P}_{\textup{b}}$, i.e.\ under the above assumptions on $(\eta_t)_{t \geq 0}$ and $(\eta^{\prime}_t)_{t \geq 0}$, 
\begin{equation}\label{eq:CouplingShock}
\mathbf{P}_{\textup{b}} ( \eta^{\prime}_t \succeq_{\textup{c}} \eta_t \text{ for all } t \geq 0 \, | \, \eta^{\prime}_0 \succeq_{\textup{c}} \eta_0) =  1 . 
\end{equation}
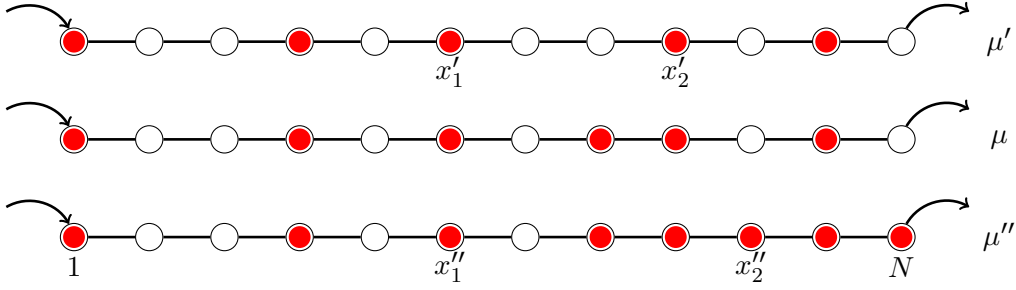
\begin{figure}
    \centering
\begin{tikzpicture}[scale=1]

	\node[shape=circle,scale=1,draw] (A1) at (0.5,-1){};
	\node[shape=circle,scale=1,draw] (A2) at (1.5,-1){};
	\node[shape=circle,scale=1,draw] (A3) at (2.5,-1){};
	\node[shape=circle,scale=1,draw] (A4) at (3.5,-1){};
	\node[shape=circle,scale=1,draw] (A5) at (4.5,-1){};
	\node[shape=circle,scale=1,draw] (A6) at (5.5,-1){};
	\node[shape=circle,scale=1,draw] (A7) at (6.5,-1){};
	\node[shape=circle,scale=1,draw] (A8) at (7.5,-1){};
	\node[shape=circle,scale=1,draw] (A9) at (8.5,-1){};
	\node[shape=circle,scale=1,draw] (A10) at (9.5,-1){};
	\node[shape=circle,scale=1,draw] (A11) at (10.5,-1){};
	\node[shape=circle,scale=1,draw] (A12) at (11.5,-1){};

 \draw [->,line width=1pt] (-0.4,-1+0.4) to [bend right,in=135,out=45] (A1);
 \draw [->,line width=1pt] (A12) to [bend right,in=135,out=45] (12.4,-1+0.4);	
	
	\node[shape=circle,scale=0.8,fill=red] (CB) at (0.5,-1) {}; 
	\node[shape=circle,scale=0.8,fill=red] (CB) at (3.5,-1) {}; 
	\node[shape=circle,scale=0.8,fill=red] (CB) at (5.5,-1) {}; 
	\node[shape=circle,scale=0.8,fill=red] (CB) at (8.5,-1) {}; 
	\node[shape=circle,scale=0.8,fill=red] (CB) at (10.5,-1) {}; 

	\draw [line width=1pt] (A1) to (A2); 	
	\draw [line width=1pt] (A2) to (A3); 	
	\draw [line width=1pt] (A3) to (A4); 	
	\draw [line width=1pt] (A4) to (A5); 	
	\draw [line width=1pt] (A5) to (A6); 	
	\draw [line width=1pt] (A6) to (A7); 	
	\draw [line width=1pt] (A7) to (A8); 	
	\draw [line width=1pt] (A8) to (A9); 	
	\draw [line width=1pt] (A9) to (A10); 	
	\draw [line width=1pt] (A10) to (A11); 	
	\draw [line width=1pt] (A11) to (A12);

		\node[shape=circle,scale=1,draw] (A1) at (0.5,-2.3){};
	\node[shape=circle,scale=1,draw] (A2) at (1.5,-2.3){};
	\node[shape=circle,scale=1,draw] (A3) at (2.5,-2.3){};
	\node[shape=circle,scale=1,draw] (A4) at (3.5,-2.3){};
	\node[shape=circle,scale=1,draw] (A5) at (4.5,-2.3){};
	\node[shape=circle,scale=1,draw] (A6) at (5.5,-2.3){};
	\node[shape=circle,scale=1,draw] (A7) at (6.5,-2.3){};
	\node[shape=circle,scale=1,draw] (A8) at (7.5,-2.3){};
	\node[shape=circle,scale=1,draw] (A9) at (8.5,-2.3){};
	\node[shape=circle,scale=1,draw] (A10) at (9.5,-2.3){};
	\node[shape=circle,scale=1,draw] (A11) at (10.5,-2.3){};
	\node[shape=circle,scale=1,draw] (A12) at (11.5,-2.3){};

 \draw [->,line width=1pt] (-0.4,-2.3+0.4) to [bend right,in=135,out=45] (A1);
 \draw [->,line width=1pt] (A12) to [bend right,in=135,out=45] (12.4,-2.3+0.4);	
	
	\node[shape=circle,scale=0.8,fill=red] (CB) at (0.5,-2.3) {}; 
	\node[shape=circle,scale=0.8,fill=red] (CB) at (3.5,-2.3) {}; 
	\node[shape=circle,scale=0.8,fill=red] (CB) at (5.5,-2.3) {}; 
	\node[shape=circle,scale=0.8,fill=red] (CB) at (7.5,-2.3) {}; 
	\node[shape=circle,scale=0.8,fill=red] (CB) at (8.5,-2.3) {}; 
	\node[shape=circle,scale=0.8,fill=red] (CB) at (10.5,-2.3) {}; 

	\draw [line width=1pt] (A1) to (A2); 	
	\draw [line width=1pt] (A2) to (A3); 	
	\draw [line width=1pt] (A3) to (A4); 	
	\draw [line width=1pt] (A4) to (A5); 	
	\draw [line width=1pt] (A5) to (A6); 	
	\draw [line width=1pt] (A6) to (A7); 	
	\draw [line width=1pt] (A7) to (A8); 	
	\draw [line width=1pt] (A8) to (A9); 	
	\draw [line width=1pt] (A9) to (A10); 	
	\draw [line width=1pt] (A10) to (A11); 	
	\draw [line width=1pt] (A11) to (A12);

		\node[shape=circle,scale=1,draw] (A1) at (0.5,-3.6){};
	\node[shape=circle,scale=1,draw] (A2) at (1.5,-3.6){};
	\node[shape=circle,scale=1,draw] (A3) at (2.5,-3.6){};
	\node[shape=circle,scale=1,draw] (A4) at (3.5,-3.6){};
	\node[shape=circle,scale=1,draw] (A5) at (4.5,-3.6){};
	\node[shape=circle,scale=1,draw] (A6) at (5.5,-3.6){};
	\node[shape=circle,scale=1,draw] (A7) at (6.5,-3.6){};
	\node[shape=circle,scale=1,draw] (A8) at (7.5,-3.6){};
	\node[shape=circle,scale=1,draw] (A9) at (8.5,-3.6){};
	\node[shape=circle,scale=1,draw] (A10) at (9.5,-3.6){};
	\node[shape=circle,scale=1,draw] (A11) at (10.5,-3.6){};
	\node[shape=circle,scale=1,draw] (A12) at (11.5,-3.6){};

 \draw [->,line width=1pt] (-0.4,-3.6+0.4) to [bend right,in=135,out=45] (A1);
 \draw [->,line width=1pt] (A12) to [bend right,in=135,out=45] (12.4,-3.6+0.4);	
	
	\node[shape=circle,scale=0.8,fill=red] (CB) at (0.5,-3.6) {}; 
	\node[shape=circle,scale=0.8,fill=red] (CB) at (3.5,-3.6) {}; 
	\node[shape=circle,scale=0.8,fill=red] (CB) at (5.5,-3.6) {}; 
	\node[shape=circle,scale=0.8,fill=red] (CB) at (7.5,-3.6) {}; 
	\node[shape=circle,scale=0.8,fill=red] (CB) at (8.5,-3.6) {}; 
	\node[shape=circle,scale=0.8,fill=red] (CB) at (9.5,-3.6) {}; 
	\node[shape=circle,scale=0.8,fill=red] (CB) at (10.5,-3.6) {}; 
	\node[shape=circle,scale=0.8,fill=red] (CB) at (11.5,-3.6) {}; 

	\draw [line width=1pt] (A1) to (A2); 	
	\draw [line width=1pt] (A2) to (A3); 	
	\draw [line width=1pt] (A3) to (A4); 	
	\draw [line width=1pt] (A4) to (A5); 	
	\draw [line width=1pt] (A5) to (A6); 	
	\draw [line width=1pt] (A6) to (A7); 	
	\draw [line width=1pt] (A7) to (A8); 	
	\draw [line width=1pt] (A8) to (A9); 	
	\draw [line width=1pt] (A9) to (A10); 	
	\draw [line width=1pt] (A10) to (A11); 	
	\draw [line width=1pt] (A11) to (A12); 
 
   \node (X4) at (12.8,-1) {$\mu^{\prime}$}; 
   \node (X4) at (12.8,-2.3) {$\mu$}; 
   \node (X4) at (12.8,-3.6) {$\mu^{\prime\prime}$};  

  \node (X4) at (0.5,-4) {$1$};  
  \node (X4) at (5.5,-4) {$x^{\prime\prime}_1$}; 
  \node (X4) at (9.5,-4) {$x^{\prime\prime}_2$}; 
  \node (X4) at (11.5,-4) {$N$};  
  
  \node (X4) at (5.5,-1.4) {$x^{\prime}_1$}; 
  \node (X4) at (8.5,-1.4) {$x^{\prime}_2$};

\end{tikzpicture}
    \caption{Consider $\beta^{\prime}$ and $\beta^{\prime \prime}$ according to Theorem \ref{thm:HighLowShock} with $k=2$ and shock locations $x_1^{\prime},x_1^{\prime\prime}$ and $x_2^{\prime},x_2^{\prime\prime}$, respectively, for the first and third line. The figure shows a sample according to the measures $\mu=\mu^{N,q,\alpha,\beta}$, respectively $\mu_N^{\prime}$ for $\beta^{\prime}$ and $\mu_N^{\prime\prime}$ for $\beta^{\prime\prime}$,  ordered according to the basic coupling $\mathbf{P}_{\textup{b}}$. }
    \label{fig:CouplingSchütz}
\end{figure}
%
%
%
%
%
In particular, the relation \eqref{eq:CouplingShock} holds also in stationarity. We have now all tools to finish the proof of Theorem~\ref{thm:HighLowShock} for the open ASEP and Theorem~\ref{thm:HighLowShockWASEP} for the open WASEP. 

\begin{proof}[Proof of Theorem \ref{thm:HighLowShock}] It suffices to consider the low density phase. 
For $\alpha>0$ and $\beta^{\prime\prime} > \beta >\beta^{\prime}$ as in \eqref{eq:ShockConditions}, let $\mu^{\prime\prime}_{N}$, $\mu$, and $\mu^{\prime}_{N}$ denote to the stationary distributions of the respective open ASEPs. Note that by the choice of $\beta^{\prime}$ and $\beta^{\prime\prime}$, the measures $\mu^{\prime}_{N}$ and $\mu^{\prime\prime}_{N}$ can be represented as Bernoulli shock measures, and write $x_1^{\prime}$, respectively $x_1^{\prime\prime}$, for the position of the left-most shock. 
Using \eqref{eq:CouplingShock} for $t \rightarrow \infty$, and Proposition~\ref{pro:Schutz}, we see that
\begin{equation}\label{eq:FinalTVBoundASEP}
\limsup_{N \rightarrow \infty}\TV{\mu_I - \textup{Ber}_I\left(\frac{\alpha}{1-q}\right)}  \leq \limsup_{N \rightarrow \infty} \P( x_1^{\prime}  \leq  b  ) +  \limsup_{N \rightarrow \infty} \P( x_1^{\prime\prime}  \leq b  ) 
\end{equation} for all $I=\lsem a,b \rsem$ with $N-b \gg 1$.
Since by Proposition \ref{pro:ShockConcentration}, the right-hand side in \eqref{eq:FinalTVBoundASEP} converges to $0$, we obtain the desired result in the shock region of the low density phase. 
\end{proof}


\begin{proof}[Proof of Theorem \ref{thm:HighLowShockWASEP}]
Again, we only consider the low density phase of the open WASEP, where $u > \max(1,v)$. Note that for $u$ and $v$ from \eqref{def:uAndvWASEP}, there exists some finite $N_0=N_0(u,v)$ and sequences $(v_{N}^{(1)},v_N^{(2)})_{N \geq N_0}$ and $(k_{N})_{N \geq N_0}$ such that $v_{N}^{(1)} \leq  v \leq v_N^{(2)}$ and 
\begin{equation}
u v_{N}^{(1)} q^{k_N} = u v_{N}^{(2)} q^{k_N+1} = 1 
\end{equation} for $k_N \in \N$, and all $N \geq N_0$. Let $\mu_{N}^{(1)}$ and $\mu_{N}^{(2)}$ denote the stationary distributions of the corresponding open WASEPs, and note that $\mu_{N}^{(1)}$ and $\mu_{N}^{(2)}$ are Bernoulli shock measures. Let $x_1^{(1)}$, respectively $x_1^{(2)}$, denote the position of the left-most shock. As in \eqref{eq:FinalTVBoundASEP}, 
\begin{equation}\label{eq:FinalTVBoundASEP2}
\limsup_{N \rightarrow \infty}\TV{\mu_I - \textup{Ber}_I\left(\frac{\alpha}{1-q}\right)}  \leq \limsup_{N \rightarrow \infty} \P( x_1^{(1)}  \leq  b  ) +  \limsup_{N \rightarrow \infty} \P( x_1^{(2)}  \leq b  ) . 
\end{equation}
holds for all $I=\lsem a,b \rsem$ with $N-b \gg N^{\varepsilon}$. Using now Proposition \ref{pro:ShockConcentration} for the measures $\mu_{N}^{(1)}$ and $\mu_{N}^{(2)}$, the right-hand side in \eqref{eq:FinalTVBoundASEP2} converges to $0$, allowing us to conclude.
\end{proof}
\section{Approximating the stationary distribution of the open TASEP}
\label{sec:LPP}

In this section, we prove Theorem \ref{thm:TASEP} on the stationary distribution of the open TASEP. We start by recalling some basic definitions on last passage percolation on the strip, and refer the interested reader to \cite{ES:HighLow,S:MixingTASEP} for a more thorough discussion.

\subsection{A brief introduction to last passage percolation}\label{sec:LPPIntroduction}
We define in the following directed last passage percolation on the slab 
\begin{equation}
    \mathcal S _N:=\big\{ (x,y)\in \mathbb Z ^2 :   y\le x \le y+N  \big\}
\end{equation} with upper boundary $\partial _1 (\mathcal S_N)$ and lower boundary $\partial _2 (\mathcal S_N)$ 
\begin{equation}\label{def:Boundaries12}
    \partial _1 (\mathcal S_N):=\big\{ (x,x) :  x\in \mathbb Z   \big\} ,\quad \text{and} \quad \partial _2 (\mathcal S _N):=\big\{ (x+N,x) :  x\in \mathbb Z   \big\} .
\end{equation} Fix $\alpha,\beta>0$ and let $(\omega_v)_{v \in \mathcal{S}_N}$ be a family of independent Exponential distributed random variables. For $v \in \partial _1 (\mathcal S_N)$,  $\omega_v$ has rate $\alpha$, and for $v \in \partial _2 (\mathcal S_N)$, $\omega_v$ has rate $\beta$. For the remaining $v \in \mathcal{S}_N$, we assign rate~$1$. Notice that there is a natural coupling when changing the boundary parameters $\alpha,\beta>0$ to $\alpha^{\prime},\beta^{\prime}>0$, multiplying the Exponential-$\alpha$-random variables along $\partial_1(\mathcal{S}_N)$ by $\frac{\alpha}{\alpha^{\prime}}$, respectively the Exponential-$\beta$-random variables along $\partial_2(\mathcal{S}_N)$ by $\frac{\beta}{\beta^{\prime}}$.  \\

Let $\succeq$ denote the component-wise ordering on $\Z^2$. For $w\succeq v$,  we say that $\pi(v,w)$ is a directed up-right \textbf{lattice path} from $v$ to $w$ if
\begin{equation*}
\pi(v,w) = \{ z^0=v, z^{1},\dots, z^{\lVert w-v \rVert_1}=w \, \colon \, z^{i+1}-z^{i} \in \{ \eone,\etwo\} \text{ for such } i \}  . 
\end{equation*} Here, we set $\eone:=(1,0)$ and $\etwo := (0,1)$, and recall that  $\lVert w-v \rVert_1$ is the $\ell_1$-distance between $v$ and $w$. For  $A \subseteq \Z^2$, let $\Pi_{A}^{v,w}$ denote the set of all lattice paths from $v$ to $w$, which do not leave the set $A$. We define
\begin{equation}\label{def:LastPassageTime}
T_{\alpha,\beta}(v,w):= \max_{\pi(u,v) \in \Pi_{\mathcal{S}_N}^{v,w}} \sum_{z \in \pi(v,w) \backslash \{w\}} \omega_z 
\end{equation} as the \textbf{last passage time} from $u$ to $v$ in the slab $\mathcal{S}_N$. We drop the subscript whenever the value of $\alpha$ and $\beta$ is clear from the context, and write $T(\gamma)$ for the passage time along a fixed path~$\gamma$. A path $\Gamma(v,w)$ maximizing the right-hand side in \eqref{def:LastPassageTime} is called a \textbf{geodesic}. 

Next, we relate last passage percolation on a strip to the open TASEP. Let $\eta \in \{0,1\}^{N}$ with $N\in \N$. We set $G_0=\{g_0^{i} \in \Z^2 \colon i \in \Z\}$ to be the  \textbf{initial growth interface}, where $g^0_0:= (0,0)$, and recursively
\begin{equation}\label{def:GrowthInterface}
g_0^{i} := \begin{cases} g_0^{i-1} + \eone & \text{ if } \eta(i)=0 \\
 g_0^{i-1} - \etwo & \text{ if } \eta(i)=1
\end{cases}
\end{equation} for all $i \geq 1$. For all $t \geq 0$, we define
\begin{equation}\label{def:GrowthInterface2}
G_t :=\Big\{ u \in \Z^2 \, \colon \, \max_{w\in G_0}T(w,u) \leq t \text{ and }  \max_{w\in G_0}T(w,u+(1,1))> t\Big\}  , 
\end{equation} and we write $G_t = \{g_t^{i-1} \in \Z^2 \colon i \in \lsem N +1\rsem \}$ such that $g_t^{0}=(0,0)$ for some $x\in \Z$, and
\begin{equation}
g_t^{i} - g_t^{i-1} \in \{ \eone, -\etwo \}
\end{equation} for all $i\in \lsem N \rsem$. The process $(G_t)_{t \geq 0}$ is called the \textbf{growth interface} for $(\omega_v)_{v \in \mathcal{S}_N}$.
The next statement, relating the open TASEP to the growth interface, is Lemma 3.1 in \cite{S:MixingTASEP}. Let us remark that Lemma 3.1 in \cite{S:MixingTASEP} only considers the case where $\alpha,\beta \geq 1/2$. However, the proof directly extends to general parameters $\alpha,\beta>0$. 
\begin{lemma}\label{lem:CornerGrowthRepresentation} Let $N\in \N$, and let $(\eta_t)_{t\geq0}$ be the open TASEP with respect to $\alpha,\beta>0$. There exists a coupling between $(\eta_t)_{t\geq0}$ and $(\omega_v)_{v \in \mathcal{S}_N}$ such that the respective growth interface $(G_t)_{t \geq 0}$ and the process $(\eta_t)_{t\geq0}$ satisfy almost surely for all $t\geq 0$ and $i\in \lsem N \rsem$
\begin{equation}
\{ \eta_t(i) = 0 \} = \{ g_t^{i} - g_t^{i-1}  = \eone \}  .
\end{equation}
\end{lemma}
In the following, we collect four preliminary results on last passage percolation. Since the results follow from well-known arguments, we only give a sketch of proof or provide a suitable reference. We start with the notion of a line $\mathbb{L}_n$ at height $n$ on the strip $\mathcal{S}_N$ as 
\begin{equation}
\mathbb{L}_n := \left\{ z \in \mathcal{S}_N \colon \left\lVert z \right\rVert_1 = n \right\} .
\end{equation}
 For all $n,k \geq N$, let the minimal and maximal last passage time connecting  $\mathbb{L}_n$ and $\mathbb{L}_{n+k}$ be
\begin{align*}
T_{\min}(n,n+k) &= T^{\alpha,\beta}_{\min}(n,n+k) := \min_{x \in \mathbb{L}_n, y \in \mathbb{L}_{n+k}} T_{\alpha,\beta}(x,y) \\ 
T_{\max}(n,n+k) &= T^{\alpha,\beta}_{\max}(n,n+k) := \max_{x \in \mathbb{L}_n, y \in \mathbb{L}_{n+k}} T_{\alpha,\beta}(x,y)  . 
\end{align*} The following result is given as Proposition 4.5 in \cite{S:MixingTASEP}. In words, it states that the minimal and maximal last passage times between two lines in the strip of distance of order $N^{3/2}$ have fluctuations of order $N^{1/2}$.
\begin{lemma}\label{lem:VarianceUpperBoundRephased} There exist constants  $c_1,c_2,\tilde{\theta}>0$, independently of $\alpha,\beta \geq \frac{1}{2}$ and $N^{\prime}_0 \in \N$ such that for all $\theta \geq \tilde{\theta}$, and all $n\geq N \geq N^{\prime}_0$, we have that
\begin{align}
\label{eq:UpperBoundFluctuations}\P\big( T_{\max}(n,n+\theta^{-1}N^{3/2}) - 2\theta^{-1}N^{\frac{3}{2}}\geq \theta \sqrt{N} \big) &\leq \exp(-c_1 \theta^{\frac{3}{2}}) \\
\label{eq:LowerBoundFluctuations}\P\big( T_{\min}(n,n+\theta^{-1}N^{3/2})  - 2\theta^{-1}N^{\frac{3}{2}} \leq  -  2\theta \sqrt{N} \big) &\leq  \exp(- c_2 \theta)  .
\end{align}
\end{lemma}

Next, consider the last passage times $T_{\Z}(v,w)$ and the geodesic $\Gamma_{\Z}(v,w)$ between two sites $v,w \in \Z^2$. Here, we take the same definition for sites on $\mathcal{S}_N$, but replace the environment on $\mathcal{S}_N$ by i.i.d.\ Exponential-$1$-distributed random variables on $\Z^2$; see \cite{S:ReviewCornerGrowth} for a survey on this model. The following statement is due to Ledoux and Rider \cite{LR:BetaEnsembles}.
\begin{lemma}\label{lem:PointToPointLaw}
There exist constants $c_1,c_2,\theta_0>0$ such that
\begin{equation}\label{eq:StatementSteepUpper}
\P \left( \left| T_{\Z}((0,0),v) -   (\sqrt{v_1}+\sqrt{v_2})^2 \right| \geq  \theta  v_1^{1/2}v_2^{-1/6}   \right) \leq c_1 \exp\left(-c_2 \theta \right) 
\end{equation} and for all $\theta>\theta_0$ and $(v_1,v_2)\in \N^2$. 
\end{lemma}
Let us stress that by shift invariance of the environment, Lemma \ref{lem:PointToPointLaw} provides a moderate deviation estimate for the last passage time between any ordered pair of sites in $\Z^2$. In particular, note that the last passage time between $(x,x)$ and $(y,y)$ for some $x,y\in \N$ has fluctuations of order $|x-y|^{1/3}$.
For a lattice path $\gamma$ from $v$ to $v+(n,m n)$ with some $m \in (0,\infty)$  and $n\in \N$, we define its \textbf{transversal fluctuations} for all $\ell\in \lsem n(1+m) \rsem$ as
\begin{equation}\label{def:TransversalFluctuations}
\TF(\gamma,\ell) := \lVert \gamma(\ell) - m\ell \rVert_1 \quad \text{ and } \quad \TF(\gamma) := \max_{\ell\in \lsem n(1+m)\rsem}\TF(\gamma,\ell) .
\end{equation} We have the following moderate deviation bound on the transversal fluctuations.
\begin{lemma}\label{lem:TransversalFluctuations} Let $\alpha,\beta \geq \frac{1}{2}$ and fix $\phi>1$ and $m_0\in [\phi^{-1},\phi]$. There exist constants $\theta_0,\ell_0,c>0$  such that for all $ m \in \big( \frac{m_0}{10},10 m_0 \big)$, $\ell \geq \ell_0$  and $\theta>\theta_0$
\begin{equation}\label{eq:LocalGeodesicFlutuations}
\P(\TF(\Gamma_{\Z}((0,0),(n,mn)),\ell) \geq \theta \ell^{2/3}) \leq \exp(-c \theta) 
\end{equation} as well as that
\begin{equation}\label{eq:GlobalGeodesicFlutuations}
\P(\TF(\Gamma((0,0),(n,mn))) \geq \theta n^{2/3}) \leq \exp(-c \theta) .
\end{equation} 
\end{lemma}
\begin{proof}[Sketch of proof] For the last passage times with respect to $\Z^2$, the first statement \eqref{eq:LocalGeodesicFlutuations} is the content of Theorem~3 in~\cite{BSS:Coalescence}. The second statement \eqref{eq:GlobalGeodesicFlutuations} for geodesics in $\Z^2$ was first shown in \cite{BSV:SlowBond} using a chaining argument, see Proposition C.9 in~\cite{BGZ:TemporalCorrelation} for a detailed proof. The proof of  Proposition~C.9 in~\cite{BGZ:TemporalCorrelation} applies one-to-one for geodesics in $\mathcal{S}_N$. However, as an input for the proof of Proposition~C.9 in~\cite{BGZ:TemporalCorrelation}, we need to replace the moderate deviation estimate from Lemma~\ref{lem:PointToPointLaw} for last passage times in $\Z^2$ by the moderate deviation estimates in Lemma \ref{lem:VarianceUpperBoundRephased} for minimal and maximal last passage times,  as well as Lemma 4.14 in \cite{S:MixingTASEP} for a moderate deviation bound on last passage times in $\mathcal{S}_N$ between any pair of boundary points of the strip.
\end{proof}

Our last preliminary result concerns the coalescence of geodesics in $\Z^2$. 

%
\begin{lemma}\label{lem:CoalesenceOnZ2} 
Let $L>0$ be fixed, and $k,n\in \N$. Consider the four sites $(a_i)_{i \in \lsem 4 \rsem}$ with
\begin{align*}
a_1 :=  (0, \lfloor L k^{2/3} \rfloor )   \quad a_2:=  (\lfloor L k^{2/3} \rfloor ,0)  \quad
a_3 := (n,n- \lfloor Ln^{2/3}\rfloor )   \quad a_4:= (n- \lfloor Ln^{2/3} \rfloor,n) . 
\end{align*}
For all $k=k(n)$ such that $n \gg k \gg 1$ as $n\rightarrow \infty$, we have that
\begin{equation}
\lim_{n \rightarrow \infty}\P( \Gamma_{\Z}(a_1,a_4) \cap   \Gamma_{\Z}(a_2,a_3)  \neq \emptyset ) = 1 . 
\end{equation}
\end{lemma}
\begin{proof}[Sketch of proof]
For the points $(a^{\prime}_i)_{i \in \lsem 4 \rsem}$ given as
\begin{equation*}
a^{\prime}_1 :=  (0, 3\lfloor L k^{2/3} \rfloor )   \quad a^{\prime}_2:=  (0,-3\lfloor L k^{2/3} \rfloor ,0)  \quad
a^{\prime}_3 := (n,n+3 \lfloor Ln^{2/3}\rfloor)   \quad a_4:= (n, n-3 \lfloor Ln^{2/3}\rfloor)
\end{equation*} the claim is Corollary 3.4 in \cite{BSS:Coalescence}. By Lemma \ref{lem:TransversalFluctuations}, as $k \gg 1$, we obtain that 
\begin{equation*}
\lim_{n \rightarrow \infty} \P\left( \text{there exist } \tilde{a} \in \Gamma_{\Z}(a^{\prime}_1,a^{\prime}_4) \text{ and } \tilde{b} \in  \Gamma_{\Z}(a^{\prime}_2,a^{\prime}_3) \text{ such that } \tilde{a} \succeq a_4 \text{ and } a_2 \succeq \tilde{b} \right) = 1 , 
\end{equation*} allowing us to conclude by the ordering of geodesics; see for example Lemma 11.2 in \cite{BSV:SlowBond}.
\end{proof}

\subsection{The TASEP in the maximal current phase}\label{sec:TASEPmaxcurrent}

Before giving the proof of Theorem~\ref{thm:TASEP} in the maximal current phase, let us outline our strategy. By Lemma~\ref{lem:CornerGrowthRepresentation}, the law of the open TASEP in an interval $I$ and at time $t$ depends only on the last passage times to a certain rectangle $R_{N,I,t}$ in $\mathcal{S}_N$. Consider now two open TASEPs, one with $\alpha,\beta \geq \frac{1}{2}$ and one where both parameters equal $\frac{1}{2}$, so that the invariant measure projected to $I$ is a Bernoulli-$\frac{1}{2}$-product measure. Using the above results on coalescence of geodesics, we couple the open TASEPs such that their last passage times to $R_{N,I,t}$ agree up to a time shift. In order to remove the time shift, and thus to conclude that the invariant measures of both processes projected to $I$ are close in total variation, we apply a strategy recently introduced in \cite{SS:TASEPcircle} in the context of mixing times for the TASEP on the circle. \\

For a segment $I=\lsem a,b \rsem \subseteq \lsem N \rsem$, and $t\geq 0$,  consider the finite segment $\mathbb{S}_N^{I}$ and the rectangle $R_{N,I,t}$ defined as
\begin{align}
\begin{split}
\mathbb{S}_N^{I} &:= \mathbb{L}_N \cap \{ (v_1,v_2) \in \Z^2 \colon v_1 - v_2 \in I \} \\ R_{N,I,t} &:= \{ u\in \mathbb{S}^{I}_n \text{ for some } n \in \lsem t/2-N^{3/4},t/2+N^{3/4} \rsem \} .
\end{split} 
\end{align}

\begin{lemma}\label{lem:ApproximationMaxTASEP} Let $G_\tau=(g_\tau^{i})_{i \in \lsem N +1\rsem}$ be the growth interface at time $\tau= N^{3/2}\log^{2}(N)$ when starting from the all empty initial configuration. Then for all $\alpha,\beta \geq \frac{1}{2}$, and $N$ sufficiently large,
\begin{equation}
\P\left( (g_{\tau}^{i})_{i \in \lsem a-1,b \rsem} \subseteq  R_{N,I,\tau} \right) \geq 1-N^{-3} .
\end{equation} 
\end{lemma}
\begin{proof}
This follows by iterating the bound on the last passage times in Lemma \ref{lem:VarianceUpperBoundRephased}.  
\end{proof}

As a consequence of Lemma \ref{lem:CornerGrowthRepresentation} and Lemma \ref{lem:ApproximationMaxTASEP}, it suffices to study the last passage times to sites in $R_{N,I,\tau}$ in order to investigate the law of the TASEP with open boundaries at time $\tau$ on the interval $I$. We fix some notation. For $N \in \N$, $I=\lsem a,b \rsem$, and $t \geq 0$, let
\begin{equation}
\mathbb{S}_{\textup{target}}^{L} := \mathbb{S}_{\lfloor t/2- (N(b-a))^{2/3} \rfloor}^{ \lsem a-L(N(b-a))^{1/2},b+L(N(b-a))^{1/2} \rsem} \quad \text{ and } \quad  \mathbb{S}_{\textup{end}}^{L} := \mathbb{S}_{\lfloor t/2 +N^{3/4} \rfloor}^{ \lsem a-L(b-a),b+L(b-a) \rsem} . 
\end{equation} The reason for the choice of the parameters will become clear in the sequel. Let $d_1$ and $d_2$ for the segment $\mathbb{S}_{\textup{target}}^{L}$ as well as $d_4$ and $d_3$ for the segment $\mathbb{S}_{\textup{end}}^{L}$ denote its upper left and down right endpoints, respectively. An illustration of  these quantities is given in Figure~\ref{fig:LPPcomponents}. 
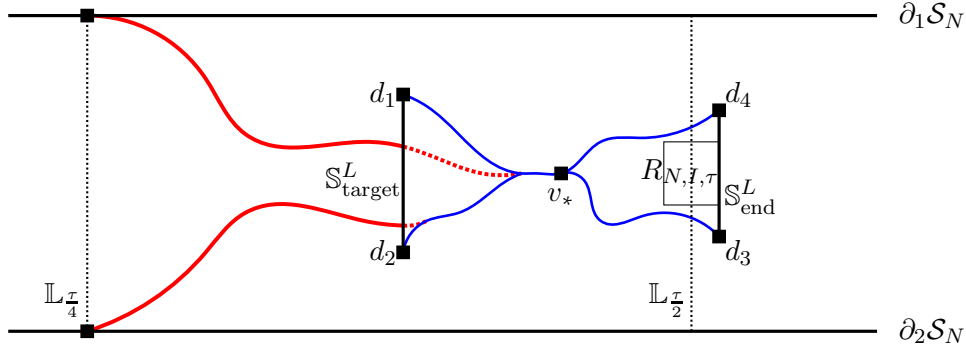
\begin{figure}
    \centering
\begin{tikzpicture}[scale=1.05]

%
%
%
%
%
%

   
	\node (x1) at (11.7,4){$\partial_1 \mathcal{S}_N$} ;      
 	\node (x2) at (11.7,0){$\partial_2 \mathcal{S}_N$} ;        

\draw[red,line width=1.5 pt] (1,0) to[curve through={(2.5,1)..(3,1.5)..(4.5,1.4)}] (5,1.34);

\draw[red,line width=1.5 pt] (1,4) to[curve through={(2.5,3.2)..(3,2.5)..(4.5,2.4)}] (5,2.34);

\draw[red,line width=1.5 pt,densely dotted] (5,2.34) to[curve through={(5.2,2.28)..(6,2)}] (6.5,2);	    

\draw[red,line width=1.5 pt,densely dotted] (5,1.34) to[curve through={(5.1,1.34)}] (5.3,1.4);	          
  \draw[blue,line width=1 pt] (5,3) to[curve through={(5.3,2.85)..(5.7,2.5)..(6.5,2)..(6.7,2)..(7,2)..(7.5,2.4)..(8,2.45)}] (9,2.8);	  
   
\draw[blue,line width=1 pt] (7,2) to[curve through={(7.3,1.9)..(7.5,1.4)..(8,1.45)}] (9,1.2);	     

\draw[blue,line width=1 pt] (5,1) to[curve through={(5.3,1.4)..(5.7,1.5)..(6.3,1.9)}] (6.5,2);

   \draw[line width=1.2pt] (0,0) -- (11,0);
   \draw[line width=1.2pt] (0,4) -- (11,4);

    \draw (9-0.7,2-0.4) rectangle (9,2+0.4); 
    
    \draw[densely dotted, line width=0.8pt] (9-0.35,0)--(9-0.35,4);

    \draw[densely dotted, line width=0.8pt] (1,0)--(1,4);
   
    \draw[line width=1.2pt] (9,2-0.8)--(9,2+0.8);   
   
	\draw[line width=1.2pt] (5,2-1)--(5,2+1);

 \filldraw [fill=black] (9-0.08,2-0.8-0.08) rectangle (9+0.08,2-0.8+0.08);        
 \filldraw [fill=black] (9-0.08,2+0.8-0.08) rectangle (9+0.08,2+0.8+0.08);  
 \filldraw [fill=black] (5-0.08,1-0.08) rectangle (5+0.08,1+0.08);        
 \filldraw [fill=black] (5-0.08,3-0.08) rectangle (5+0.08,3+0.08);  
	
 \filldraw [fill=black] (7-0.08,2-0.08) rectangle (7+0.08,2+0.08);  
 \filldraw [fill=black] (1-0.08,0-0.08) rectangle (1+0.08,0+0.08);  
 \filldraw [fill=black] (1-0.08,4-0.08) rectangle (1+0.08,4+0.08);

	\node (x1) at (4.75,3){$d_1$} ;   
	\node (x2) at (4.75,1){$d_2$} ;   
	\node (x3) at (9.25,1){$d_3$} ;   
	\node (x4) at (9.25,3){$d_4$} ;   
	\node (x5) at (8.5,2){$R_{N,I,\tau}$} ;      
	\node (x6) at (9.4,1.7){$\mathbb{S}_{\textup{end}}^{L}$} ;     
	\node (x7) at (4.5,1.9){$\mathbb{S}_{\textup{target}}^{L}$} ;   
	\node (x8) at (8.35,0.4){$\mathbb{L}_{\frac{\tau}{2}}$} ;   
	\node (x9) at (0.7,0.4){$\mathbb{L}_{\frac{\tau}{4}}$} ;  	
	\node (x10) at (7,1.7){$v_{\ast}$} ;  		
   
	\end{tikzpicture}	
    \caption{Visualization of the different segments, lines and geodesics on the strip $\mathcal{S}_N$, rotated by $\pi/4$, which are used in the proof of Theorem~\ref{thm:TASEP}. }
    \label{fig:LPPcomponents}
\end{figure}

\begin{lemma}\label{lem:CoalesenceMaxCurrent}
Let $\tau= N^{3/2}\log^{2}(N)$ and $\alpha,\beta \geq \frac{1}{2}$. Assume that $|I|=(b-a) \gg N^{3/4}$. Then for all $\delta>0$, 
\begin{equation}\label{eq:CoalescencePoint1}
\liminf_{N \rightarrow \infty}\P\left( \Gamma(u,w) \cap  \mathbb{S}_{\textup{target}}^{L}  \neq  \emptyset \text{ for all } u \in \mathbb{L}_{\tau/4} \text{ and } w\in  \mathbb{S}_{\textup{end}}^{L} 
   \right) \geq  1 -\frac{\delta}{2} 
\end{equation}
for some $L=L(\delta)>0$ sufficiently large. Moreover, we have that
\begin{equation}\label{eq:CoalescencePoint2}
\liminf_{N \rightarrow \infty}\P\left( \exists v_{\ast} \in \Z^{2} \, \colon \,  v_{\ast}\in \Gamma(u, w)  \text{ for all } u \in \mathbb{L}_{\tau/4} \text{ and } w\in  R_{N,I,\tau}   \right) \geq  1 - \delta . 
\end{equation}  \end{lemma}
 
\begin{proof} By symmetry and the ordering of geodesics, it suffices for \eqref{eq:CoalescencePoint1} to show that 
\begin{equation}\label{eq:ReductionGeo}
\liminf_{N \rightarrow \infty} \P\left( \Gamma\big((\tau/8,\tau/8),d_4\big) \cap \mathbb{S}_{\textup{target}}^{L} \neq \emptyset \right) \geq 1- \frac{\delta}{4} 
\end{equation} for some sufficiently large $L>0$. In order to show \eqref{eq:ReductionGeo}, we first argue that there exists some $\delta^{\prime}=\delta^{\prime}(\delta)>0$ sufficiently small such that
\begin{equation}\label{eq:RedRedGeo}
\liminf_{N \rightarrow \infty}  \P\left(\Gamma\big((\tau/8,\tau/8),d_4\big)  \cap \mathbb{S}^{\lsem\delta^{\prime}N,(1-\delta^{\prime})N\rsem}_{\tau/2-\delta^{\prime}N}\right) \geq 1- \frac{\delta}{5}  .
\end{equation} This follows from bounding the last passage time $T((\tau/8,\tau/8), d_4)$ by Lemma~\ref{lem:VarianceUpperBoundRephased} and Lemma~\ref{lem:PointToPointLaw}, when we restrict the space of lattice paths in the definition of the last passage time to contain only paths which pass through a site in $\mathbb{S}^{\lsem \delta^{\prime}N \rsem }_{\tau/2-\delta^{\prime}N}$ or $\mathbb{S}^{\lsem N,(1-\delta^{\prime})N \rsem }_{\tau/2-\delta^{\prime}N}$. Using \eqref{eq:RedRedGeo}, we obtain \eqref{eq:ReductionGeo}, and thus \eqref{eq:CoalescencePoint1}, for sufficiently large $L=L(\delta,\delta^{\prime})$ by applying now Lemma~\ref{lem:TransversalFluctuations} twice -- once to see that the geodesics $\Gamma_{\Z}(u,d_4)$ and $\Gamma(u,d_4)$ agree with probability tending to $1$ as $N \rightarrow \infty$ for all $u \in \mathbb{S}^{\lsem\delta^{\prime}N,(1-\delta^{\prime})N\rsem}_{\tau/2-\delta^{\prime}N}$, and once to bound the transversal fluctuations of $\Gamma_{\Z}(u,d_4)$ when crossing the interval $\mathbb{S}_{\textup{target}}^{L}$. Next, we argue that \eqref{eq:CoalescencePoint2} holds. 
Note that by Lemma~\ref{lem:TransversalFluctuations} and the choice of $\mathbb{S}_{\textup{target}}^{L}$ and $\mathbb{S}_{\textup{end}}^{L} $, for all $L>0$ fixed, 
\begin{equation*}
\lim_{N \rightarrow \infty} \P\left( \Gamma_{\Z}(d_1,d_4)=\Gamma(d_1,d_4) \text{ and } \Gamma_{\Z}(d_2,d_3)=\Gamma(d_2,d_3)  \right) = 1 .
\end{equation*} From Lemma~\ref{lem:CoalesenceOnZ2} with $a_i=d_i$ for all $i\in \lsem 4 \rsem$, and Lemma \ref{lem:TransversalFluctuations}, we get that
\begin{equation*}
\lim_{N \rightarrow \infty} \P\left(  \Gamma(d_1,d_4) \cap R_{N,I,\tau} \neq \emptyset \, \vee \, \Gamma(d_2,d_3) \cap R_{N,I,\tau} \neq \emptyset \right) = 0.
\end{equation*} Together with \eqref{eq:CoalescencePoint1} and the ordering of geodesics, this yields \eqref{eq:CoalescencePoint2}.
\end{proof}

\begin{corollary}\label{cor:TimeShifts} Assume that $|I|=(b-a) \gg N^{3/4}$.  Recall the coupling $\mathbf{P}$ for last passage percolation on the strip for $\alpha,\beta \geq \frac{1}{2}$, and that $T_{\alpha,\beta}(\cdot,\cdot)$ denotes the respective last passage times. Then 
\begin{equation}\label{eq:TimeShifts}
\lim_{N \rightarrow \infty} \mathbf{P}( \exists t_{\ast} \in \R \, \colon \,  T_{\alpha,\beta}((0,0),u) = t_{\ast} + T_{\frac{1}{2},\frac{1}{2}}((0,0),u) \text{ for all } u \in R_{N,I,N^{3/2}\log^{2}(N)} ) = 1 . 
\end{equation} 
\end{corollary}
\begin{proof}
Let $\mathcal{A}_{N,L}$ be the event defined as
\begin{equation*}
\begin{split}\mathcal{A}_{N,L} := \big\{ \exists v_{\ast} \in \Z^{2} \, &\colon \,  v_{\ast}\in \Gamma(u, w)  \text{ for all } u \in \mathbb{L}_{\tau/4} \text{ and } w\in  R_{N,I,\tau}   \big\} \\ &\cap \Big\{ \Gamma(d_1,d_4) \cap \partial_1 \mathcal{S}_N = \emptyset \text{ and } \Gamma(d_2,d_3) \cap \partial_2 \mathcal{S}_N = \emptyset  \Big\} . 
\end{split}
\end{equation*}
 Whenever $\mathcal{A}_{N,L}$ occurs with some site $v_{\ast}$, note that for all $\alpha,\beta \geq \frac{1}{2}$ and $u\in R_{N,I,\tau}$
\begin{equation}\label{eq:DecompositionVast}
T_{\alpha,\beta}((0,0),u) = T_{\alpha,\beta}((0,0),v_{\ast}) + T_{\alpha,\beta}(v_{\ast},u) = T_{\alpha,\beta}((0,0),v_{\ast}) + T_{\frac{1}{2},\frac{1}{2}}(v_{\ast},u)  . 
\end{equation}
By Lemma~\ref{lem:TransversalFluctuations} and  Lemma~\ref{lem:CoalesenceMaxCurrent}, for all $\delta>0$, there exists some $L=L(\delta)>0$ such that
\begin{equation}\label{eq:ALNevent}
\liminf_{N \rightarrow \infty } \P(\mathcal{A}_{N,L}) \geq 1 - \delta . 
\end{equation} 
Note that we can apply Lemma~\ref{lem:TransversalFluctuations} to bound the transversal fluctuation of $\Gamma(v_{\ast},u)$ uniformly in $u\in R_{N,I,\tau}$ and the choice of $\alpha$ and $\beta$ by the transversal fluctuations of $\Gamma(v_{\ast},d_4)$. As we can choose $\delta>0$ in \eqref{eq:ALNevent} arbitrarily close to $0$, we conclude.
\end{proof}

Recall from Lemma \ref{lem:CornerGrowthRepresentation} the one-to-one correspondence between the open TASEP and last passage percolation on the strip. In the following, we consider the TASEP with open boundaries $(\eta_t)_{t \geq 0}$ with respect to parameters $\alpha$ and $\beta$, and the TASEP with open boundaries $(\zeta_t)_{t \geq 0}$ in the triple point, where both boundary parameters equal~$\frac{1}{2}$. Both processes start from the empty initial configuration at time $0$ in a common last passage percolation environment. Note that by  combining Lemma~\ref{lem:CornerGrowthRepresentation} and Corollary~\ref{cor:TimeShifts}, we ensure that
\begin{equation}
\lim_{N \rightarrow \infty} \mathbf{P}( \exists s_{\ast} \in \R \, \colon \, \eta_{t}=\zeta_{t+s_{\ast}} ) = 1. 
\end{equation} for all $t\geq T^{1/2,1/2}_{\max}(0,N^{3/2}\log^{2}(N))$; see also Lemma 5.3 in \cite{SS:TASEPcircle}.
We will now eliminate the time change $s_{\ast}$ by applying the random extension and time shift technique introduced by Sly and the second author in  \cite{SS:TASEPcircle} in order to study mixing times for the TASEP on the circle; see also~\cite{ES:HighLow} for a similar argument for the mixing time in the high and the low density phase. Since we follow the arguments analogously to Section~5 of~\cite{SS:TASEPcircle} for periodic last passage percolation, we will only give a sketch of proof.
\begin{lemma}\label{lem:ConcludeFromCoalesence}
Assume that $|I|=(b-a) \gg N^{3/4}$. For all $N\in \N$, set $\tau=\tau(N)=N^{\frac{3}{2}}\log^2(N)$. There exists a coupling $\tilde{\mathbf{P}}$ between the open TASEPs $(\eta_t)_{t \geq 0}$ and  $(\zeta_t)_{t \geq 0}$ such that
\begin{equation}
\lim_{N \rightarrow \infty} \tilde{\mathbf{P}}\left( \eta_{\tau}(x) = \zeta_{\tau}(x) \text{ for all } x\in I \right) = 1 . 
\end{equation} 
%
%
%
\end{lemma}
\begin{proof}[Sketch of the proof] Let $(\omega^{\eta}_v)_{v \in \mathcal{S}_N}$ and $(\omega^{\zeta}_v)_{v \in \mathcal{S}_N}$ denote the environments corresponding to $(\eta_t)_{t \geq 0}$ and  $(\zeta_t)_{t \geq 0}$, respectively, under the coupling $\mathbf{P}$. For both processes, we construct families of last passage percolation environments, which we obtain by cutting the environment along the line $\mathbb{L}_{\lfloor \tau/2 \rfloor}$, and adding to both environments an extra number of rows $Y_{\eta}$ and $Y_{\zeta}$, respectively. More precisely, let $(\hat{\omega}^{\eta}_v)_{v \in \mathcal{S}_N}$ have the same law as $(\omega^{\eta}_v)_{v \in \mathcal{S}_N}$, but chosen independently, and define for all $i\in \N \cup \{0\}$ the environment $(\omega^{\eta,i}_v)_{v\in \mathcal{S}_N}$, with its law denoted by $\mathbf{P}_i$, by 
\begin{equation*}
\omega^{\eta,i}_v := \begin{cases}  \omega^{\eta}_v & \text{ if } v \in \mathbb{L}_{m} \text{ for some } m < \lfloor \tau/2 \rfloor , \\
 \omega^{\eta}_{v- (i,i)} & \text{ if } v \in \mathbb{L}_{m} \text{ for some } m \geq \lfloor \tau/2 \rfloor +2i ,  \\
 \hat{\omega}_v & \text{ otherwise} ,
\end{cases}
\end{equation*} for all $v\in \mathcal{S}_N$. The environments $(\omega^{\zeta,i}_v)_{v\in \mathcal{S}_N}$ for $i\in \N \cup \{0\}$ are defined analogously, and we denote by $T^{i,\eta}$ and $T^{i,\zeta}$ the corresponding last passage times. 
Let $\mathcal{A}$ be the event that there exists a site $v_{\ast}$ such
that in both environments $(\omega^{\eta}_v)_{v \in \mathcal{S}_N}$ and $(\omega^{\zeta}_v)_{v \in \mathcal{S}_N}$ according to $\mathbf{P}$, we have that $v_{\ast} \in \Gamma(u,w)$ for all $u \in \mathbb{L}_{\lfloor \tau/2 \rfloor}$ and $w\in R_{N,I,\tau}$, while the geodesics $\Gamma(v_{\ast},w)$ do not touch the boundary of $\mathcal{S}_N$. Assuming that $\mathcal{A}$ occurs, we fix such a site $v_{\ast}$, and note that by Lemma \ref{lem:VarianceUpperBoundRephased}, for all $N$ sufficiently large
\begin{equation}
\mathbf{P}\big( | T_{\frac{1}{2},\frac{1}{2}}((0,0),v_{\ast}) - T_{\alpha,\beta}((0,0),v_{\ast}) | \leq N^{\frac{11}{20}} \, \big| \, \mathcal{A} \big) \geq 1 - N^{-2} . 
\end{equation} Let us remark at this point that the choice of the exponent $\frac{11}{20}$, and of all similar exponents in the following, is not optimal, but sufficient for our purposes. By Lemma 5.7 of \cite{SS:TASEPcircle}, adjusted for last passage percolation on the strip, there exists a coupling of $Y_{\eta}$ and $Y_{\zeta}$ such that $Y_{\eta}$ and $Y_{\zeta}$ are both marginally uniformly distributed on
\begin{equation}
\mathbb{B}:=\left\{ \lfloor i N^{1/10} \rfloor \colon i \in  \lsem N^{1/2} \rsem \right\}   , 
\end{equation}
and whenever the event $\mathcal{A}$ occurs, we have that with probability at least $1-N^{-1/40}$
\begin{equation}\label{eq:RandomExtension}
\left| T^{Y_{\eta},\eta}((0,0),v_{\ast}+(Y_{\eta},Y_{\eta})) - T^{Y_{\zeta},\zeta}((0,0),v_{\ast}+(Y_{\zeta},Y_{\zeta})) \right| \leq N^{1/6} . 
\end{equation}
To remove the remaining discrepancy in the last passage times, we again modify the environments $(\omega^{\eta,Y_{\eta}}_v)_{v\in \mathcal{S}_N}$ and $(\omega^{\zeta,Y_{\zeta}}_v)_{v\in \mathcal{S}_N}$, respectively. For all $i\in \mathbb{B}$ and $u \in [0,1]$, we define the family of environments $(\tilde{\omega}^{\eta,i,u}_v)_{v \in \Z^2}$ by
\begin{equation*}
\tilde{\omega}^{\eta,i,u}_v := \begin{cases} (1+  uN^{-51/40})\omega^{\eta,i}_v & \text{ if } v \in  \mathbb{L}_n \text{ for some }  n \leq \tau/2 \\
\omega^{\eta,i}_v & \text{ otherwise} , 
\end{cases}
\end{equation*} and similarly for $(\tilde{\omega}^{\zeta,i,u}_v)_{v \in \Z^2}$. Let $T^{\eta,i,u}$ and $T^{\zeta,i,u}$ be the corresponding last passage times.  Lemma 5.8 in \cite{SS:TASEPcircle} guarantees that for all choices of $u\in [0,1]$ and $N$ sufficiently large,
\begin{equation}\label{eq:FixedUTVapproximation}
\TV{ \P\Big( (\omega^{\eta,Y_{\eta}}_v)_{v\in \mathcal{S}_N} \in \cdot \Big)  - \P\Big( (\tilde{\omega}^{\eta,Y_{\eta},u}_v)_{v \in \Z^2} \in \cdot \Big) } \leq N^{-1/50} , 
\end{equation} and similarly for $(\tilde{\omega}^{\zeta,u}_v)_{v \in \Z^2}$. Furthermore, note that for all $i\in \mathbb{B}$, the function
\begin{equation}
u \mapsto f_i(u):=T^{\eta,u,i}(0,v^{\ast}+(i,i))
\end{equation} is monotone increasing, convex, and piece-wise linear. Hence, by Lemma \ref{lem:VarianceUpperBoundRephased}
\begin{equation}\label{eq:SlopeCondition}
\lim_{ N\rightarrow \infty} \mathbf{P}_i\left( \frac{f_i(u_2)-f_i(u_1)}{(u_2-u_1)N^{51/40}}  \in \left(\frac{1}{2}\tau,2\tau \right)  \text{ for all } 0 <u_1 <u_2 < 1\right)   = 1 . 
\end{equation}
As in Section 5.4 of \cite{SS:TASEPcircle} for periodic last passage percolation -- see also the end of Section~4 in~\cite{ES:HighLow} for a similar argument for last passage percolation on the strip under observation \eqref{eq:SlopeCondition}  -- there exists now a coupling between $\tilde{\mathbf{P}}$ of $\mathcal{U}_1$ and $\mathcal{U}_2$ such that
$\mathcal{U}_1$ and $\mathcal{U}_2$ are  uniformly distributed on $[0,1]$, and we have that
\begin{equation}
\lim_{N \rightarrow \infty} \tilde{\mathbf{P}}\Big( T^{\eta,Y_{\eta},\mathcal{U}_1}((0,0),v_{\ast}) = T^{\zeta,Y_{\zeta},\mathcal{U}_2}((0,0),v_{\ast})   \, \Big| \,  \mathcal{A} \Big) = 1 . 
\end{equation}
Together with Lemma \ref{lem:CoalesenceMaxCurrent} and Corollary \ref{cor:TimeShifts} to bound the probability of the event $\mathcal{A}$, 
\begin{equation*}
\lim_{N \rightarrow \infty} \tilde{\mathbf{P}}\left(  T^{\eta,Y_{\eta},\mathcal{U}_1}((0,0),u+(Y_{\eta},Y_{\eta})) =  T^{\zeta,Y_{\zeta},\mathcal{U}_2}((0,0),u+(Y_{\zeta},Y_{\zeta})) \text{ for all } u \in R_{N,I,\tau} \right) = 1 . 
\end{equation*}
Since by Lemma \ref{lem:CornerGrowthRepresentation} and Lemma \ref{lem:ApproximationMaxTASEP} the law of $\eta_{\tau}$ and $\zeta_{\tau}$ on $I$ only depends on the last passage times to $R_{N,I,\tau}$ with probability tending to $1$ as $N \rightarrow \infty$, we conclude.
\end{proof}


\begin{proof}[Proof of \eqref{eq:TASEPStatementMaxCurrent} in Theorem \ref{thm:TASEP}] Without loss of generality, let $|I|=(b-a) \gg N^{3/4}$ as this only increasing the total variation distance. For $\tau=N^{3/2}\log^2(N)$, let  $(\eta_t)_{t \geq 0}$ and $(\zeta_t)_{t \geq 0}$ be two open TASEPs with boundary parameters $\alpha,\beta \geq  \frac{1}{2}$ for $(\eta_t)_{t \geq 0}$, and both boundary parameters equal to $\frac{1}{2}$ for $(\zeta_t)_{t \geq 0}$, respectively. Using the coupling representation of the total variation distance -- see for example Corollary 5.5 in \cite{LPW:markov-mixing} --  Lemma \ref{lem:ConcludeFromCoalesence}
ensures that 
\begin{equation}\label{eq:FirstCoupling}
\lim_{N \rightarrow \infty} \TV{ \P(\eta^{I}_{\tau} \in \cdot ) - \P(\zeta^{I}_{\tau} \in \cdot ) } = 0   .
\end{equation}
By Theorem~1.3 in~\cite{S:MixingTASEP}, stating that the total variation mixing time of $(\zeta_t)_{t \geq 0}$ is of order $N^{3/2}$, and the fact the invariant measure of $(\zeta_t)_{t \geq 0}$ is the uniform distribution on the state space $\Omega_N$, 
\begin{align}
\lim_{N \rightarrow \infty} \TV{ \P(\eta^{I}_{\tau} \in \cdot) - \mu_{\alpha,\beta}^{I}} =0 
\end{align}
as well as 
\begin{align}
\lim_{N \rightarrow \infty} \TV{ \P(\zeta^{I}_{\tau} \in \cdot) -  \textup{Ber}_{I}\left(\frac{1}{2}\right)} =0 . 
\end{align}
Using \eqref{eq:FirstCoupling} and the triangle inequality for the  total variation distance, we conclude the first part of Theorem \ref{thm:TASEP} on approximating the stationary distribution. 
\end{proof}

\subsection{The TASEP in the high and in the low density phase}\label{sec:TASEPhighlow}

We will only show  \eqref{eq:TASEPStatementLow} for the low density phase in Theorem \ref{thm:TASEP} as the proof of \eqref{eq:TASEPStatementHigh} follows by the same arguments.  Moreover, without loss of generality, we let $I=\lsem b \rsem$ with $N-b \gg N^{1/3}\log(N)$.
We start with the following basic observation, which is the analogue of Lemma \ref{lem:ApproximationMaxTASEP}. 

\begin{lemma}\label{lem:LocateInterfaceHighLow}
Consider the growth interface $G_t=(g_t^{i})_{i \in \lsem N +1\rsem}$ at time $\tau= N\log(N)$ when starting from the all empty initial configuration. Then for all $\alpha,\beta>0$ with $\alpha < \min(\frac{1}{2},\beta)$, 
\begin{equation}
\P\left( (g_{\tau}^{i})_{i \in \lsem 0,b \rsem} \subseteq  \bigcup_{n \in \lsem \tau, 3\alpha^{-1}\tau \rsem } \mathbb{S}^{\lsem 0,b \rsem}_n \right) \geq 1-N^{-3}
\end{equation} for all $N$ sufficiently large.
\end{lemma}
\begin{proof}
This follows from Lemma \ref{lem:VarianceUpperBoundRephased} and Lemma \ref{lem:PointToPointLaw}, dominating for the upper bound the environment on the strip $\mathcal{S}_N$ by an i.i.d.\ Exponential-$\alpha$-distributed environment on $\Z^2$.
\end{proof}

In the following, our goal is to show that for all $u \in \mathbb{S}^{\lsem 0,b\rsem }_n$ with some $n \leq 3\alpha^{-1} N$, the geodesic $\Gamma((0,0),u)$ does with high probability not intersect the boundary $\partial_2(\mathcal{S}_N)$. We start with the following uniform bound on last passage times.

\begin{lemma}\label{lem:ApproximationLowTASEP} Let $\alpha<\frac{1}{2}$ and $\beta=1$. There exist constants $c,\tilde{c}>0$ such that for all $M\in \lsem N^{3/2} \rsem$ and $u\in \mathbb{S}_{n}^{\lsem 0,b \rsem}$ with some $n\in \lsem \tilde{c}N,N^2-M \rsem$, 
\begin{equation*}
\P\left(  T((0,0),u+(M,M)) - T((0,0),u) - \frac{M}{\alpha(1-\alpha)} < - c \Big(N^{\frac{1}{3}}\log(N) + M^{\frac{1}{2}}\log(M)\Big) \right) < N^{-4} .  
\end{equation*}

\end{lemma}

In order to show Lemma \ref{lem:ApproximationLowTASEP}, we require the following result on moderate deviations for the last passage times and traversing probabilities on the strip. Its content is Proposition~3.4 and Lemma~4.1 in~\cite{ES:HighLow}, so we omit the proof.

\begin{lemma}\label{lem:HalfSpaceModerate}
Let $\alpha<\min(\frac{1}{2},\beta)$. There exists some $\theta_0,c>0$ such that for all $m,n\in \lsem N^2 \rsem$ with $m\geq n$,  all $\theta > \theta_0$, and all $N
$ sufficiently large \begin{equation}\label{eq:Halfspace1}
\P\left(  \Big|T((n,n),(m,m)) - \frac{m-n}{\alpha(1-\alpha)}\Big| \geq \theta (m-n)^{\frac{1}{2}} \right) \leq  \exp(-c\theta) . 
\end{equation}
Moreover, we have that
\begin{equation}\label{eq:Halfspace2}
\P(  \Gamma((n,n),(m,m)) \cap \partial_2(\mathcal{S}_N) =  \emptyset \text{ for all } m,n \in \lsem N^2 \rsem) \geq 1 - N^{-5} . 
\end{equation}
\end{lemma}

\begin{proof}[Proof of Lemma \ref{lem:ApproximationLowTASEP}] 

For the geodesic $\Gamma((0,0),u)$, let $v_{\ast}$ denote the last intersection point with the boundary $\partial_1(\mathcal{S}_N)$. By Lemma \ref{lem:PointToPointLaw} and Lemma \ref{lem:HalfSpaceModerate}, recalling the partial order $\succeq$ on $\Z^2$, there exist constants $c_1,c_2>0$ such that
\begin{equation}\label{eq:LocateTheIntersection}
\P\big(   (n-c_1 N,n-c_1 N)  \succeq  v_{\ast} \succeq  (n-c_2 N,n-c_2 N) \big) \geq 1 - N^{-6}
\end{equation} for all $N$ sufficiently large; see also Lemma 4.3 in \cite{ES:HighLow} for a more refined estimate on the intersection point $v_{\ast}$. Assume that the event in  \eqref{eq:LocateTheIntersection} holds for some $v_{\ast}=v_{\ast}(u)$. Then we combine an upper bound on the last passage time $T(v_{\ast},u)$ by Lemma \ref{lem:PointToPointLaw},  a lower bound on the last passage time $T(v_{\ast}, v_{\ast}+(m,m))$ by Lemma \ref{lem:HalfSpaceModerate}, and a lower bound on the last passage time $T(v_{\ast}+(m,m), u+(m,m))$ by Lemma \ref{lem:PointToPointLaw} 
in order to obtain the desired bound on the last passage times in Lemma \ref{lem:ApproximationLowTASEP}.
\end{proof}

\begin{proof}[Proof of \eqref{eq:TASEPStatementLow} in Theorem \ref{thm:TASEP}] Recall that we denote by $T_{\alpha,\beta}(v,w)$ the last passage time between $v$ and $w$ in the environment on the strip $\mathcal{S}_N$ with respect to boundary parameters $\alpha,\beta>0$. Set $\tau=N\log(N)$. Consider the open TASEP $(\eta_t)_{t \geq 0}$ with respect to boundary parameters $\alpha$ and $\beta$, and the open TASEP  $(\zeta_t)_{t \geq 0}$ with respect to boundary parameters $\alpha$ and $1-\alpha$, both started from the empty initial configuration and having their respective last passage percolation environments coupled according to $\mathbf{P}$. 
We claim that it suffices to show that for all $\alpha<\frac{1}{2}$ and $\beta>\alpha$
\begin{equation}\label{eq:HighLowAlmostThere}
\mathbf{P}\left( T_{\alpha,\beta}((0,0),u) = T_{\alpha,1-\alpha}((0,0),u)  \text{ for all } u\in \bigcup_{n \in \lsem \tau, 3\alpha^{-1}\tau \rsem} \mathbb{S}^{\lsem 0,b \rsem}_n \right)  \geq 1 - N^{-1}
\end{equation} for all $N$ sufficiently large, 
i.e.\ under the coupling $\mathbf{P}$ for different boundary parameters,  the geodesics do not intersect the boundary $\partial_2(\mathcal{S}_N)$. Assuming \eqref{eq:HighLowAlmostThere}, note that by Lemma~\ref{lem:LocateInterfaceHighLow} for $\tau=N\log(N)$, and the coupling representation of the total variation distance, 
\begin{equation}
\TV{ \P(\eta^{I}_\tau \in \cdot ) - \P(\zeta^{I}_\tau \in \cdot ) } \leq N^{-1} \, .
\end{equation} As a consequence of Theorem 1.1 in \cite{ES:HighLow}, stating a bound on the total variation mixing time of $(\zeta_t)_{t \geq 0}$ of order $N$, we have that
\begin{equation}
\lim_{N \rightarrow \infty} \TV{ \P(\zeta^{I}_\tau \in \cdot) -  \textup{Ber}_{I}\left(\alpha\right)} = 0 . 
\end{equation}
Using that the invariant measure of $(\zeta_t)_{t \geq 0}$ is a Bernoulli-$\alpha$-product measure, this allows us to conclude \eqref{eq:TASEPStatementLow}. It remains to verify that \eqref{eq:HighLowAlmostThere} holds. Fix  $u \in \mathbb{S}^{\lsem 0,b \rsem}_n$ with $N-b \gg N^{1/3}\log(N)$ for some $n\in \lsem \tau, 3\alpha^{-1}\tau \rsem$, and let
\begin{equation}
n_{\ast} := \max\left\{ n \in \N \cup \{0\} \, \colon \, u-(n,n) \in \Gamma_{\alpha,\beta}((0,0),u) \, \vee \, u-(n,n) \in \Gamma_{\alpha,1-\alpha}((0,0),u) \right\} .
\end{equation}
Note that if $T_{\alpha,\beta}((0,0),u) \neq T_{\alpha,1-\alpha}((0,0),u)$, we must have $n_{\ast} \gg N^{1/3}\log(N)$ by our choice of $b$. Let $\bar{T}(u-(m,m),u)$ be the last passage time from $u-(m,m)$ to $u$ when restricting to available space of lattice paths to not intersect $\partial_1(\mathcal{S}_N)$. From Lemma \ref{lem:PointToPointLaw} together with Lemma \ref{lem:HalfSpaceModerate}, we obtain that for all $m$ sufficiently large
\begin{equation}\label{eq:RestrictedGeodesicLow}
\P\left( \bar{T}(u-(m,m),u) > \frac{m}{\min(\beta,\frac{1}{2})(1-\min(\beta,\frac{1}{2}))}+ m^{1/2}\log^2(m) \right) \leq m^{-20} . 
\end{equation}  
Using Lemma \ref{lem:ApproximationLowTASEP} for a lower bound on the last passage time from $(0,0)$ to $u-(n_{\ast},n_{\ast})$ and $u$, respectively, and Lemma \ref{lem:HalfSpaceModerate} to rule out that $\Gamma((0,0),u)$ returns to $\partial_1(\mathcal{S}_N)$ after intersecting $\partial_2(\mathcal{S}_N)$, we see from \eqref{eq:RestrictedGeodesicLow} that for some $c>0$ and all $N$ sufficiently large, 
\begin{equation*}
\P(n_{\ast} > c N^{1/3}\log(N)) \leq N^{-4} . 
\end{equation*} 
A union bound over the sites $u$ in \eqref{eq:RestrictedGeodesicLow} yields \eqref{eq:HighLowAlmostThere}, and thus finishes the proof.
\end{proof}

\subsection*{Acknowledgment} EN was supported by NSF Grant DMS-2052659.  DS acknowledges the DAAD PRIME program for financial support. We are grateful to Amol Aggarwal, Matthias Birkner, Ivan Corwin, and Stefan Junk for helpful discussions. Moreover, we deeply thank Zongrui Yang and Zhengye Zhou for pointing out an error in Proposition \ref{pro:Schutz} in the first version of the manuscript.


\bibliographystyle{plain}

\bibliography{ASEP}


\appendix

\section{Approximation on finite intervals}

In this section, we give a proof of Proposition \ref{pro:Finite}. The argument follows along the same lines as Theorem~3.29 in Part III of \cite{L:Book2}, which covers the case $q=0$. We will in the following assume without loss of generality that $q,\alpha,\beta>0$. We start by recalling some basic results on the \textbf{current} of the open ASEP, that is for some $i \in \lsem N-1 \rsem$
\begin{equation}\label{def:Current}
 \mathcal{J}^{N} := \mu^{N,q,\alpha,\beta}( \eta(i)=1 \text{ and } \eta(i+1)=0 ) - q \cdot \mu^{N,q,\alpha,\beta}( \eta(i)=0 \text{ and } \eta(i+1)=1 )
\end{equation}
A simple computation, using the generator $\mathcal{L}$ from \eqref{def:Gen} shows that \eqref{def:Current} is in fact independent of the choice of $i$. The following result can be found in Section 4 of \cite{S:DensityProfilePASEP}.
\begin{lemma}[Sasamoto \cite{S:DensityProfilePASEP}]\label{lem:CurrentOpen} Let $\alpha,\beta>0$ and $q \in (0,1)$. Then we have that
\begin{equation}
\lim_{N \rightarrow \infty} \mathcal{J}^{N} =   \begin{cases} \alpha(1-\alpha)(1-q) & \text{ if } \alpha < \min\Big( \beta,\frac{1-q}{2}\Big) \\
\beta(1-\beta)(1-q) &  \text{ if } \beta < \min\Big( \alpha,\frac{1-q}{2}\Big) \\
\frac{1}{4}(1-q) & \text{ if } \min(\alpha,\beta) > \frac{1-q}{2}  .
\end{cases}
\end{equation}
\end{lemma}
Next, we require a way to compare the invariant measure for different parameters $\alpha,\beta>0$.
Recall that we denote by $\succeq_{\textup{c}}$ the component-wise partial ordering on the state space $\{ 0,1\}^{N}$. For two probability measures $\nu$ and $\nu^{\prime}$, we say that $\nu$ \textbf{stochastically dominates} $\nu^{\prime}$, and write $\nu \succeq \nu^{\prime}$ if there exists a coupling $\mathbf{P}_{\nu,\nu^{\prime}}$ between $\eta \sim \nu$ and $\eta^{\prime} \sim \nu^{\prime}$ such that $\mathbf{P}_{\nu,\nu^{\prime}}( \eta \succeq_{\textup{c}} \eta^{\prime})$. 
The following result can be found for example as Lemma 2.10 in \cite{GNS:MixingOpen}.

\begin{lemma}[Gantert et al.\ \cite{GNS:MixingOpen}]\label{lem:DominationAux} For  $\alpha,\beta>0$ and $q \in (0,1)$, recall the parameters $u,v$ defined in \eqref{def:uAndv}. Then 
\begin{equation}
 \textup{Ber}_N\Big(\max\Big(\frac{1}{1+u},\frac{v}{1+v}\Big) \Big) \succeq  \mu^{N,q,\alpha,\beta} \succeq \textup{Ber}_N\Big(\min\Big(\frac{1}{1+u},\frac{v}{1+v}\Big) \Big) . 
\end{equation}
\end{lemma} 
We have now all tools in order to show Proposition \ref{pro:Finite}.

\begin{proof}[Proof of Proposition \ref{pro:Finite}]
In the following, with a slight abuse of notation, we treat the measures $\mu^{N,q,\alpha,\beta}$ as measures on $\{0,1\}^{\Z}$ by extending to the left and right with the empty sites. For a measure $\nu$ on $\{0,1\}^{\Z}$, we denote by $\theta_{x}\nu$ the measure shifted by $x$. Without loss of generality, we assume that the weak limit
\begin{equation}\label{eq:Character1}
\bar{\mu}:= \lim_{N \rightarrow \infty} \theta_{a_N}\mu^{N,q,\alpha,\beta} 
\end{equation} exists as we can consider a suitable subsequence otherwise. Using \eqref{eq:LiggettAssumption}, we see that the measure $\bar{\mu}$ must be an invariant measure for the \textbf{asymmetric simple exclusion process on the integers}, that is the Markov process on $\{0,1\}^{\Z}$ whose generator is given by
\begin{align}\label{def:GenZ}
\bar{\mathcal{L}}f(\eta) &= \sum_{x \in \Z} \big(\eta(x)(1-\eta(x+1)) + q \eta(x+1)(1-\eta(x)) \big) \left[ f(\eta^{x,x+1})-f(\eta) \right]  . 
\end{align}
It is a classical result by Liggett that the set of extremal invariant measures of the asymmetric simple exclusion process on the integers consists only of Bernoulli-$\rho$-product measures $\textup{Ber}(\rho)$ for some $\rho \in [0,1]$, and a family of so-called blocking measures $(\nu_\theta)_{\theta \in \R}$, along which the current $\mathcal{J}$, as defined in \eqref{def:Current} for the open ASEP, is zero \cite{L:Coupling}. Hence, since we assume $\alpha,\beta>0$ and $q\in (0,1)$, we get that by Lemma \ref{lem:CurrentOpen} that
\begin{equation}
\bar{\mu} = \int_0^{1} \textup{Ber}(\rho) \gamma( \dif \rho)
\end{equation} for some probability measure $\gamma$ on $[0,1]$. Using  Lemma~\ref{lem:DominationAux}, we claim that
\begin{equation}\label{eq:TheLastClaim}
\bar{\mu} = \int_{\min\big(\frac{1}{1+u},\frac{v}{1+v}\big)}^{\max\big(\frac{1}{1+u},\frac{v}{1+v}\big)} \textup{Ber}(\rho) \gamma( \dif \rho) 
\end{equation} holds. To see this, note that we have  for all $n\in \N$
\begin{equation}
\bar{\mu}( \eta(i)=1 \text{ for all } i \in \lsem n \rsem ) \leq \max\Big(\frac{1}{1+u},\frac{v}{1+v}\Big)^{n} . 
\end{equation}
Now let $n \rightarrow \infty$ to conclude that it suffices to consider $\rho \leq \max\big(\frac{1}{1+u},\frac{v}{1+v}\big)$. A similar argument applies for $\rho \geq \min\big(\frac{1}{1+u},\frac{v}{1+v}\big)$, and thus gives \eqref{eq:TheLastClaim}. Moreover, using the definitions of the current and $\bar{\mu}$, we see that the measure $\gamma$ must satisfy
\begin{equation}
\int_{\min\big(\frac{1}{1+u},\frac{v}{1+v}\big)}^{\max\big(\frac{1}{1+u},\frac{v}{1+v}\big)} (1-q) \rho(1-\rho) \gamma( \dif \rho) = \lim_{N \rightarrow \infty} \mathcal{J}^{N} .  
\end{equation} Thus, combining the above observations and using Lemma \ref{lem:CurrentOpen}, a computation shows that
\begin{equation}\label{eq:Character2}
\bar{\mu} = \begin{cases}  \textup{Ber}\Big(\frac{\alpha}{1-q}\Big) & \text{ if } \alpha < \min\Big( \beta,\frac{1-q}{2}\Big) \\
\textup{Ber}\Big(1-\frac{\beta}{1-q}\Big) &  \text{ if } \beta < \min\Big( \alpha,\frac{1-q}{2}\Big) \\
\textup{Ber}\Big(\frac{1}{2}\Big) & \text{ if } \min(\alpha,\beta) > \frac{1-q}{2}  .
\end{cases}
\end{equation}
Since the size of the interval $\lsem a_N,b_N \rsem$ is uniformly bounded in $N$ by our assumptions, we conclude.
\end{proof}

\begin{remark}
Note that the same arguments also extend to the five parameter model of the open ASEP discussed at the end of Section \ref{sec:ASEPasPolymer}, using the results of  Section~6.1 in \cite{BECE:ExactSolutionsPASEP} on the current of the open ASEP instead of Lemma \ref{lem:CurrentOpen}.
\end{remark}

\section{A finite matrix product ansatz}\label{sec:FiniteMPAAppendix}

In the following, we show that the assumption $uvq^k=1$ implies that the invariant measure is given as a convex combination of Bernoulli shock measures. We reformulate our results in terms of the \textbf{matrix product ansatz}. We will only consider the special case of finite dimensional representations $D,E$, while for the general parameters, the matrices $D$ and $E$ are infinite-dimensional; see \cite{L:Book2} for a discussion in the special case where $q=0$. \\

We recall now the matrix product formulation; see \cite{BE:Nonequilibrium} for an introduction. We say that matrices $D$ and $E$, together with vectors $\langle W |$ and $| V \rangle$ satisfy the matrix product ansatz if 
\begin{align}\label{def:MPA}
\begin{split}
DE-qED&= (D+E)  \\
\beta D | V \rangle  &= (1-q) | V \rangle \\
\langle W | \alpha E  &= (1-q) \langle W | , 
\end{split}
\end{align} using the standard bra-ket notation.
Suppose that $ uv q^{k}=1$ holds for some finite $k\in \N$. Then Mallick and Sandow construct in \cite{MS:Finite} a solution to \eqref{def:MPA} by setting 
\begin{align}\label{def:Dmatrix} D= 
\begin{pmatrix} 
	1+v &  &  &  & & \\
	 & 1+v q &  & & & \\
	  &  & 1+v q^2 &  &  &\\ 
	  & &  & \ddots &  & \\ 	  
	  & &  &  & 	1+v q^{k-1} & \\ 
	    & &  &  & 	 &  1+v q^{k}  \\ 
\end{pmatrix}
\end{align} 
\begin{align}\label{def:Ematrix} E= 
\begin{pmatrix} 
	1+\frac{1}{v} &  &  &  & & \\
	1 & 1+\frac{1}{vq} &  & & & \\
	  & 1 & 1+\frac{1}{vq^2} &  &  &\\ 
	  & & 1 & \ddots &  & \\ 	  
	  & &  & \ddots & 	1+\frac{1}{vq^{k-1}}& \\ 
	    & &  &  & 1	 &  1+\frac{1}{vq^k}\\ 
\end{pmatrix}
\end{align}
with respect to the vectors
\begin{align*}
V = (1,0,0,\dots,0) \quad \text{ and } \quad 
W = (W_0,W_1,W_2,\dots,W_k) , 
\end{align*} where we set recursively for all $i\in [k] \cup \{ 0 \}$
\begin{equation}
W_{i-1} := W_i \Big(u(1-q^{k+1-i}) \Big)^{-1} . 
\end{equation}
We recall the following classical result on using the matrices $D$ and $E$ and the vectors $V$ and $W$ to represent the stationary distribution of the open ASEP. 

\begin{theorem}[Derrida et al.\ \cite{DEHP:ASEPCombinatorics}]\label{lem:FiniteMPAConverting}
Suppose that $uvq^{k}=1$ holds for some finite $k\in \N$. Then for all $\eta \in \{0,1\}^{N}$, the stationary distribution $\mu_N$ of the open ASEP on $\{0,1\}^N$ satisfies
\begin{equation}
\mu_N(\eta) = \frac{1}{\mathcal{Z}_N}  \Big\langle  W \Big| \prod_{i=1}^N \big(D\eta(i)+E(1-\eta(i)\big) \Big| V \Big\rangle
\end{equation} for some suitable normalization constant $\mathcal{Z}_N$, provided $D,E,V,W$ satisfy \eqref{def:MPA}. 
\end{theorem}

%

In the following, we study the structure of the matrices $D$ and $E$, and argue that they give rise to Bernoulli shock measures.  We start by observing that the normalization constants of the diagonal entries of the matrix $D+E$ are given by
\begin{equation}
Z_i := 2 + vq^{i} + (vq^{i})^{-1} = \frac{1-q}{vq^{i}} + \frac{1-q}{1-q-vq^{i}} = \frac{1}{\rho_{i}(1-\rho_{i})} ,
\end{equation} where we recall $\rho_i$ from \eqref{eq:FugacityRelation}.
Set $X_i=D \eta(i) + E (1-\eta(i))$ for $i \in [N]$ and $\eta \in \{0,1\}^N$, and we define the matrices
\begin{equation}
X^{(N)} := X_1 X_2 \cdots X_N . 
\end{equation}
We have now all tools in order to establish Proposition \ref{pro:Schutz}, following \cite{JM:Finite}, as well as the arguments in \cite{S:Schutz2} for the special case $k=1$.

\begin{proof}[Proof of Proposition \ref{pro:Schutz}]
We will only argue that \eqref{eq:ShockMeasureHigh} holds as the statement \eqref{eq:ShockMeasureLow} follows by the symmetry between particles and empty sites. 
By Theorem \ref{lem:FiniteMPAConverting}, the quantity $\mu^{N,q,u,v}(\eta)$ for a configuration $\eta$ can be written using only the matrix $X^{(N)}$ together the vectors $V$ and $W$. More precisely, observe that the matrices $X_i$ take the form 
\begin{align*}
 X_i= 
\begin{pmatrix} 
	Z_{k} \textup{Ber}_{\rho_{k}}(\eta(i))  &     & & \\
	\textup{Ber}_{\rho_{k}^{\ast}}(\eta(i)) & Z_{k-1} \textup{Ber}_{\rho_{k-1}}(\eta(i)) &   & & \\
	  &  \textup{Ber}_{\rho_{k-1}^{\ast}}(\eta(i))   & \ddots &\\ 
	  & &   \ddots & 	Z_{1} \textup{Ber}_{\rho_{1}}(\eta(i))& \\ 
	    & &    & \textup{Ber}_{\rho_{0}^{\ast}}(\eta(i))	 &  Z_{0} \textup{Ber}_{\rho_{0}}(\eta(i))\\ 
\end{pmatrix}
\end{align*} where we define the Bernoulli measures
\begin{equation}
\textup{Ber}_{\rho}(x) = \rho x  + (1-\rho) (1-x) . 
\end{equation}
Now we evaluate the matrix product $X^{(N)}$ to see that the entries
\begin{equation}
X^{(N)}_{i,j} := \textup{e}_i^{\textup{T}} X^{(N)} \textup{e}_j . 
\end{equation} 
contain the shock measures with exactly $i$ shock locations, using the only the densities $\rho_{\cdot}$ from location $k-i-j$ to $k-j$. More precisely,  recalling \eqref{def:shock}, we get that
\begin{equation}
X^{(N)}_{i,1} = \sum_{\mathbf{x} \in \Omega_{N,|i-j|}} \Big( \prod_{k=0}^{|i-j|} d_k^{x_k} \Big) \mu^{\mathbf{x}} . 
\end{equation}
Here, we use the recursion that
\begin{equation}
X^{(N)}_{i,j} =  X^{(N-1)}_{i,j} \textup{Ber}_{\rho_{k+1-j}}(\eta(N)) + \mathds{1}_{\{j < k\} }X^{(N-1)}_{i,j+1} \textup{Ber}_{\rho_{k-j}^{\ast}}(\eta(N)) .
\end{equation} In other words, to construct the shock measures with exactly $i-j$ shocks for the segment of length $N$, we can either take the shock measures with $|i-j|$ shocks of length $N-1$ and attach one site at the right end with density $\rho_{k-j}$, or we can take the shock measure with $|i-j-1|$ shocks (with densities $\rho_{k-i}$ to $\rho_{k-j+1}$) and attach one site at the right end with density~$\rho_{k-j}^{\ast}$. Since $V=(1,0,\dots,0)$, we consider only the entries $X^{(N)}_{i,1}$ which contain the shock measures using only the last $i$ shock densities. Weighting the respective shock measures according to the coefficients $Z_i$ and $W_i$, we conclude.
\end{proof}

\end{document}